\documentclass{amsart}
\usepackage{amsmath}
\usepackage{amssymb}
\usepackage[all]{xy}
\newtheorem{definition}{Definition}[section]
\newtheorem{theorem}[definition]{\bf Theorem}
\newtheorem{lemma}[definition]{\bf Lemma}
\newtheorem{corollary}[definition]{\bf Corollary}
\newtheorem{proposition}[definition]{\bf Proposition}
\newtheorem{remark}[definition]{\bf Remark}

\begin{document}
\title{Some Hopf Algebras related to  $\mathfrak{sl}_2$}

\author{Jing Wang}
\address{Mathematics Department, Beijing Forestry University, Beijing, 100083, P.R.China} \email{wang\_jing619@163.com}

\author{Zhixiang Wu}
\address{Mathematics Department, Zhejiang University,
Hangzhou, 310027, P.R.China} \email{wzx@zju.edu.cn}
\author{Yan Tan}
\address{College of Science, Zhejiang Agriculture and Forestry University, Hangzhou, 311300, P.R.China} \email{ytanalg@126.com}

\thanks{The author Jing Wang is supported by NNSFC (No.11901034) and the Fundamental Research Funds for the Central Universities (No.BLX201721). The work of Zhixiang Wu and Yan Tan is sponsored
by ZJNSF (No. LY17A010015) and NNSFC (No.11871421).}


%\date{January 1, 2001 and, in revised form, June 22, 2001.}



\begin{abstract}We give a series of infinite dimensional noncommutative and noncocommutative pointed Hopf algebras, which are Artin-Schelter Gorenstein  Hopf algebras with injective dimensions 3. Radford's Hopf algebra and Gelaki's Hopf algebra are homomorphic images of these Hopf algebras.
We determine the irreducible representations of them. We describe the Grothendieck rings of them. We obtain that two non-isomorphic Hopf algebras could have isomorphic  Grothendieck rings.
\end{abstract}

\maketitle

\section{Introduction}
The tensor product of representations of a Hopf algebra is important in the representation theory of Hopf algebras and quantum groups. In particular, the decomposition of the tensor product of indecomposable modules into a direct sum of indecomposable modules has received enormous attention. However,  little is known about how a tensor product of two indecomposable modules decomposes into a direct sum of indecomposable modules over a Hopf algebra or a quantum group.  There are some results for the decompositions of tensor products of modules over a Hopf algebra or a quantum group \cite{C1,CMS,Gu,KS,Y}.  Recently,  Green rings and their homomorphic images, which are called Grothendieck rings, of various finite-dimensional Hopf algebras have attracted numerous attentions \cite{CFZ,E,WLZ,ZWLC}. But most of them are considered in the case of finite dimensional Hopf algebras or quantum groups.

In this paper, we describe an infinite-dimensional pointed Hopf algebra $H_\beta$ for any $\beta\in({\bf K^*})^3$ and its Grothendieck ring, which differs from the Green ring. The Hopf algebra is constructed by adding group-like elements to the restricted quantum enveloping algebra of $\mathfrak {sl}_2$. The Hopf algebra $H_\beta$ has
many interesting finite dimensional images, such as Gelaki's Hopf algebra ${\mathcal
U}_{(n,N,\nu,q,\alpha,\beta,\gamma)}$ and Radford's Hopf algebra $U_{(N,\nu,\omega)}$. Here, all irreducible representations of $H_\beta$ and the decomposition of the tensor product of two irreducible $H_\beta$-modules are determined as in \cite{C}, but we consider the decomposition in the sense of Grothendieck group.  Thus we obtain the Grothendieck ring of the infinite-dimensional Hopf algebras $H_\beta$.  Meanwhile we obtain the Grothendieck rings
of Gelaki's Hopf algebra ${\mathcal
U}_{(n,N,\nu,q,\alpha,\beta,\gamma)}$ and Radford's Hopf algebra $U_{(N,\nu,\omega)}$.

This paper is organized as follows. In Section 2, we  introduce  the infinite dimensional noncommutative and noncocommutative Hopf algebra $H_\beta$, which is a pointed $AS$-Gorenstein Hopf algebra of injective dimension 3.  And we prove that the homological integral of $H_\beta$ is isomorphic to ${\bf K}$ as $H_\beta$-bimodules. To construct irreducible representations of $H_\beta$, we define a finite-dimensional quotient Hopf algebra $H_{\alpha,\beta}$ of Hopf algebra $H_\beta$ in Section 3, where $\alpha=(n,m,n_1,n_2,n_3)\in {\bf N}^5$,
$\beta=(\beta_1,\beta_2,\beta_3)\in{\bf K}^3$. We prove that the  algebra $H_{\alpha,\beta}$ could be decomposed into direct sum of matrix rings over some commutative local rings. Based on this result, all irreducible representations of $H_\beta$ are illustrated in Section 4. Besides, we point out that the category of all finite-dimensional representations of $H_\beta$ is not semisimple. Meanwhile we obtain all irreducible representations of Radford's Hopf
algebra $U_{(N,\nu,\omega)}$ and Gelaki's Hopf algebra ${\mathcal
U}_{(n,N,\nu,q,\alpha,\beta,\gamma)}$. In Section 5, the Grothendieck ring $G_0(H_\beta)$ of $H_\beta$ is constructed. We illustrate the structure of the Grothendieck ring $G_0(H_\beta)$ of $H_\beta$ in several cases. In details, Theorem~\ref{them52}(iv) describes the case in which $\beta_1=\beta_2=\beta_3=0$. Theorems~\ref{L55}, \ref{thmVI} and \ref{thmV} illustrate the cases that iff two the three equations $\beta_1=0$, $\beta_2=0$ and $\beta_3=0$ hold. In theorems \ref{V}, \, \ref{T9} and \ref{T10}, we give the cases in which iff one of the conditions $\beta_1=0$, $\beta_2=0$ and $\beta_3=0$ holds. In additional, we also obtain the Grothendieck rings $G_0(U_{(N,\nu,\omega)})$ and $G_0({\mathcal
U}_{(n,N,\nu,q,\alpha,\beta,\gamma)})$. We prove that there exists non-isomorphic Hopf algebras with isomorphic Grothendieck rings.

Throughout this paper, ${\bf K}$ denotes an algebraically closed field of characteristic zero, ${\bf K}^*$ the group of all nonzero elements of ${\bf K}$ with the product of ${\bf K}$ and $q\in{\bf K}$ a primitive $n$-th root of unity, $n\geq 2$. For any $z\in {\bf K}^*$, we denote $\langle z\rangle$ the subgroup of ${\bf K}^*$ generated by $z$, the quotient group ${\bf K}^*/N$ of ${\bf K}^*$ with respect to a normal subgroup $N$, and $\bar{z}$ the element $zN$.
We simplify the tensor product $V\otimes_{\bf K}W$ of two vector spaces over ${\bf K}$ as $V\otimes W$. The ring of integers is denoted by ${\mathbb Z}$ and $(N_1,\cdots,N_k)$ denotes the greatest common divisor of $N_1, \cdots,  N_k$, for any $N_1,\cdots, N_k\in\mathbb{Z}$.

\section{Definition of Hopf algebra $H_{\beta}$}\label{sec-2}
In this section, we give the definition of  Hopf algebra $H_\beta$. Then we prove that $H_\beta$ is an Artin-Schelter Gorenstein Hopf algebra with injective dimension 3.

%%%%%%%%%%%%%%%%%%%%%%%%%%%%%%%%%%%%%%%%%%%%%%%%%%%%%%%%%%%%%%%%%

\begin{definition}
Let  $n_1$ be a positive integer such that $2\nmid (n,n_1)$ and  $1\leq n_1\leq n$.  Assume that $q\in{\bf K}$ is a primitive $n$-th root of unity. Suppose that $H_{\beta}$ is an algebra generated by $a,b,c,x,y$ with the relations
$$ab=ba,\quad ac=ca,\quad bc=cb,\quad xa=qax,\quad ya=q^{-1}ay,\quad bx=xb,\quad cx=xc,\quad by=yb,\quad cy=yc,$$
$$yx-q^{-n_1}xy=\beta_3(a^{2n_1}-bc),\quad  x^n=\beta_1(a^{nn_1}-b^n), \quad   y^n=\beta_2(a^{nn_1}-c^n),$$
for  $\beta=(\beta_1,\beta_2,\beta_3)\in {\bf K}^3$. The coproduct $\Delta$ and counit $\varepsilon$ of $H_{\beta}$ is determined by
 $$\Delta(a)=a\otimes a,\quad 
 \Delta(b)=b\otimes b,\quad 
 \Delta(c)=c\otimes c,\quad 
 \Delta(x)=x\otimes a^{n_1}+b\otimes
 x,\quad 
 \Delta(y)=y\otimes a^{n_1}+c\otimes y,$$
 and
 $$\varepsilon(a)=\varepsilon(b)=\varepsilon(c)=1,\quad 
 \varepsilon(x)=\varepsilon(y)=0$$respectively.
 An anti-automorphism $s$ of $H_{\beta}$ is determined by
 $$s(a)=a^{-1},\quad 
 s(b)=b^{-1},\quad 
 s(c)=c^{-1},\quad 
 s(x)=-q^{-n_1}a^{-n_1}b^{-1}x,\quad s(y)=-q^{n_1}a^{-n_1}c^{-1}y.$$
\end{definition}
%%%%%%%%%%%%%%%%%%%%%%%%%%%%%%%%%%%%%%%%%%%%%%%%%%%%%%%%%%%%%%%%%
In the following theorem, we prove that $H_\beta$ is an infinite dimensional Hopf algebra with the above coproduct $\Delta,$ counit $ \varepsilon$ and antipode $s$.

\begin{theorem} \label{thm-basis}
The algebra $(H_{\beta},\ \Delta,\ \varepsilon,\ s)$ is a pointed Hopf algebra with a basis $$\{a^ib^jc^kx^uy^v|i,j,k\in{\mathbb Z},0\leq u,v\leq n-1\}.$$
\end{theorem}

\begin{proof} First, we prove that $\Delta$ can determine a homomorphism of
algebras from $H_{\beta}$ to $H_{\beta}\otimes H_{\beta}$. 

By the definition of $H_{\beta}$, we obtain that $\Delta(x)^k=\sum\limits^{k}_{l=0}\binom{k}{l}_{q^{n_1}}b^lx^{k-l}\otimes a^{(k-l)n_1}x^l$, where $\binom{k}{l}_{q^{n_1}}=\frac{(k)_{q^{n_1}}!}{(l)_{q^{n_1}}!(k-l)_{q^{n_1}}!}$ and $(p)_{q^{n_1}}!=(p)_{q^{n_1}}(p-1)_{q^{n_1}}\cdots (1)_{q^{n_1}}$ for $(p)_{q^{n_1}}=1+q^{n_1}+\cdots +q^{(p-1)n_1}$.
Hence
$$\begin{array}{lll}\Delta(x)^{n}&=&(x\otimes a^{n_1}+b\otimes x)^{n}
=x^{n}\otimes a^{n_1n}+b^n\otimes x^{n}\\
&=&\beta_1((a^{n_1n}-b^{n})\otimes
a^{n_1n}+b^{n}\otimes (a^{n_1n}-b^{n}))\\
&=&\beta_1(a^{n_1n}\otimes a^{n_1n}-b^{n}\otimes
b^{n})
=\Delta(x^{n}).\end{array} $$ Similarly, we can prove
$\Delta(y^n)=\Delta(y)^n$. Moreover,
$$\begin{array}{lll}&&\Delta(y)\Delta(x)-q^{-n_1}\Delta(x)\Delta(y)\\
&=&(y\otimes
a^{n_1}+c\otimes y)(x\otimes a^{n_1}+b\otimes x)
 -q^{-n_1}(x\otimes a^{n_1}+b\otimes
x)(y\otimes
a^{n_1}+c\otimes y)\\
&=&yx\otimes a^{2n_1}-q^{-n_1}xy\otimes a^{2n_1}+bc\otimes
yx-q^{-n_1}bc\otimes xy\\
&=&\beta_3((a^{2n_1}-bc)\otimes
a^{2n_1}+bc\otimes(a^{2n_1}-bc))\\
&=&\beta_3(a^{2n_1}\otimes a^{2n_1}-bc\otimes
bc)\\
&=&\Delta(yx-q^{-n_1}xy).\end{array}$$ It is easy to verify that
$$\Delta(x)\Delta(a)=q\Delta(a)\Delta(x),\quad\Delta(y)\Delta(a)=q^{-1}\Delta(a)\Delta(y),\quad\Delta(x)\Delta(b)=\Delta(b)\Delta(x),$$
$$\Delta(x)\Delta(c)=\Delta(c)\Delta(x),\quad\Delta(y)\Delta(b)=\Delta(b)\Delta(y),\quad\Delta(y)\Delta(c)=\Delta(c)\Delta(y),$$
$$\Delta(a)\Delta(b)=\Delta(b)\Delta(a),\quad\Delta(a)\Delta(c)=\Delta(c)\Delta(a),\quad\Delta(b)\Delta(c)=\Delta(c)\Delta(b).$$ Thus
$\Delta$ induces a homomorphism of algebras.
%%%%%%%%%%%%%%%%%%%%%%%%%%%%%%%%%%%%%%%%%%%%%%%%%%%%%%%%%%%%%%%%%%%%%%%%%%%%%
Second, it is easy to prove that $\varepsilon$ is a
homomorphism from the algebra $H_{\beta}$ to  ${\bf K}$ and 
$$(\Delta\otimes 1)\Delta(K)=(1\otimes \Delta)\Delta(K),\qquad
(1\otimes \varepsilon)\Delta(K)=(\varepsilon\otimes
1)\Delta(K)=K$$ 
for any $K\in \{a,b,c,x,y\}$. 
Thus $H_{\beta}$ is a bialgebra.

%%%%%%%%%%%%%%%%%%%%%%%%%%%%%%%%%%%%%%%%%%%%%%%%%%%%%%%%%%%%%%%%%%%%%%%%%%%%%
At last, we need to prove that $s$ is an antipode of the algebra $H_{\beta}$.

By $s(x)=-q^{-n_1}a^{-n_1}b^{-1}x,$
we get that $$\begin{array}{lll}(s(x))^n&=&(-q^{-n_1})^nb^{-n}(a^{-n_1}x)^n\\
&=&(-1)^nb^{-n}q^{\frac12n(n-1)n_1}a^{-n_1n}x^n\\
&=&(-1)^nq^{\frac12n(n-1)n_1}a^{-nn_1}b^{-n}\beta_1(a^{nn_1}-b^n)\\
&=&(-1)^{n+1}q^{\frac12n(n-1)n_1}\beta_1( a^{-nn_1}-b^{-n})\\
&=&(-1)^{n+1}q^{\frac12n(n-1)n_1}s(x^n).\end{array}$$ 

It is not difficult to prove that $s(x)^n=s(x^n)$. Indeed,
if $n$ is odd then $(-1)^{n+1}q^{\frac12n(n-1)n_1}=1$. 
If $n$ is even then we obtain that $n_1$ is odd by $2\nmid(n,n_1)$ and
$$(-1)^{n+1}q^{\frac12n(n-1)n_1}=(-1)^{n+1}q^{\frac12n(n-2)n_1
+\frac12nn_1}=(-1)^{n+n_1+1}= 1.$$

Similarly, we have
$s(y^n)=s(y)^n$. Moreover,
$$\begin{array}{lll}&&s(x)s(y)-q^{-n_1}s(y)s(x)\\
&=&a^{-n_1}b^{-1}xa^{-n_1}c^{-1}y-q^{-n_1}a^{-n_1}c^{-1}ya^{-n_1}b^{-1}x\\
&=&a^{-2n_1}(bc)^{-1}(q^{-n_1}xy-yx)
=-\beta_3a^{-2n_1}(bc)^{-1}(a^{2n_1}-bc)\\
&=&\beta_3( a^{-2n_1}-(bc)^{-1})
=s(yx-q^{-n_1}xy).\end{array}$$

 It is easy to verify that
$$s(a)s(x)=qs(x)s(a),\quad s(a)s(y)=q^{-1}s(y)s(a),\quad s(a)s(b)=s(b)s(a),$$
$$s(c)s(b)=s(b)s(c),\quad s(c)s(a)=s(a)s(c),\quad s(b)s(x)=s(x)s(b),$$
$$s(b)s(y)=s(y)s(b),\quad s(c)s(x)=s(x)s(c),\quad s(c)s(y)=s(y)s(c).$$ 
Thus $s$ is an anti-automorphism of $H_{\beta}$.
%the claim is true. 
%Finally we prove that $s$ satisfies the axiom of the antipode. 
Hence, we only need to verify the antipode axiom on the generators $x, y, a, b, c$. It is easy to verify this. 
Altogether, we have proved that $H_{\beta}$ is a Hopf algebra.

By \cite[Lemma 1]{R} or \cite[Lemma 1.1]{G}, we obtain that the Hopf algebra $H_{\beta}$ is pointed and
the set $\{a^ib^jc^k|a,b,c\in{\mathbb Z}\}$ is of all group-like elements of $H_{\beta}$.
Similarly to \cite{G,R}, we can prove that
$\{a^ib^jc^kx^uy^v|i,j,j\in{\mathbb Z},0\leq u,v\leq n-1\}$ is a basis
of $H_{\beta}$ by the Diamond Lemma \cite{B}.
\end{proof}

%%%%%%%%%%%%%%%%%%%%%%%%%%%%%%%%%%%%%%%%%%%%%%%%%%%%%%%%%%%%%%%%%%%%%%%%%%%%%
                       %Remark
%%%%%%%%%%%%%%%%%%%%%%%%%%%%%%%%%%%%%%%%%%%%%%%%%%%%%%%%%%%%%%%%%%%%%%%%%%%%%

\begin{remark}\label{rem-Gelaki}
%Let $N$ be an positive integer.
% such that $n|N$. %Assume that $(n,n_1)=1$ and $N\nmid nn_1$ if $\beta_1^2+\beta_2^2\neq0$. 
Notice that if the generators of $H_\beta$
satisfy the relations
$$a^N=1,\ b=c=1,\  N\in \mathbb{Z}\ \text{and}\ n\mid N, $$
%x^n=\beta_1(a^{nn_1}-1),\ y^n=\beta_2(a^{n_1 n}-1),\  xa=qax,\  ya=q^{-1}ay,$$
%and $yx-q^{-n_1}xy=\beta_3(a^{2n_1}-1)$, 
then the Hopf algebra $H_\beta$ is a Gelaki's Hopf algebra ${\mathcal U}_{(n,N,n_1,q,\beta_1,\beta_2,\beta_3)}$, which is defined in \cite{G}.
 %The coalgebra structure of ${\mathcal U}_{(n,N,n_1,q,\beta_1,\beta_2,\beta_3)}$ is determined by
 %$$\Delta(a)=a\otimes a,\  \Delta(x)=x\otimes a^{n_1}+1\otimes
 %x,\  \Delta(y)=y\otimes a^{n_1}+1\otimes y,$$
 %$$\varepsilon(a)=1,\  \varepsilon(x)=\varepsilon(y)=0.$$The antipode
 %of ${\mathcal U}_{(n,N,n_1,q,\beta_1,\beta_2,\beta_3)}$ is determined by  $$s(a)=a^{-1},\   s(x)=-q^{-n_1}a^{-n_1}x,\   s(y)=-q^{n_1}a^{-n_1}y.$$
\end{remark}

\begin{remark}[S. Gelaki \cite{G}]
Note that if $\omega$ is a primitive $N$-th root of unity and
 $N\nmid n_1^2$, then $$ {\mathcal
U}_{(N/(N,n_1),N,n_1,\omega^{n_1},0,0,\gamma)}\simeq
U_{(N,n_1,\omega)}$$ as Hopf algebras for any $\gamma\in {\bf
K}^*$.  
Especially, $$U_{(N,n_1,\omega)}={\mathcal  U}_{(N/(N,n_1),N,n_1,\omega^{n_1},0,0,1)}.$$
Here, $U_{(N,n_1,\omega)}$ is called Radford's Hopf algebra. 
 \end{remark}

%Let $(a^N-1,b-1,c-1)$ be an ideal of $H_{\beta}$ generated by $a^N-1,b-1,c-1$ for any $N$ such that $n|N$.  It is obvious that $(a^N-1,b-1,c-1)$ is a Hopf ideal of $H_{\beta}$ and $H_{\beta}/(a^N-1,b-1,c-1)$ is isomorphic to Gelaki's Hopf algebra ${\mathcal U}_{(n,N,n_1,q,\beta_1,\beta_2,\beta_3)}$.

%%%%%%%%%%%%%%%%%%%%%%%%%%%%%%%%%%%%%%%%%%%%%%%%%%%%%%%%%%%%%%%%%%%%%%%%%%%%%
                       %injective dimension 3
%%%%%%%%%%%%%%%%%%%%%%%%%%%%%%%%%%%%%%%%%%%%%%%%%%%%%%%%%%%%%%%%%%%%%%%%%%%%%

Let $\frak S=\{a^{kn}b^lc^t| k,l,t$ are nonnegative integers$\}$.
Then $\frak S$ is  a multiplicatively closed set of   ${\bf K}[a^{n},b,c]$. Since ${\bf K}[a^{\pm
n},b^{\pm 1},c^{\pm 1}]$ is the localization of ${\bf K}[a^{n},b,c]$ with respect to  $\frak S$, we get that the
Gelfand-Kirillov dimensions of ${\bf K}[a^{\pm n},b^{\pm 1},c^{\pm 1}]$  and ${\bf K}[a^{n},b,c]$ are equal by \cite[Proposition 8.2.13]{CR}. 
For simplification, we write the Gelfand-Kirillov dimension as GK dim in this section.

 Recall that a Hopf algebra $A$ over field ${\bf K}$ is Artin-Schelter Gorenstein (simply  AS-Gorenstein) defined in \cite[Definition 3.1]{WZ} if

(AS1) injdim $_AA=d< \infty,$

(AS2) dim$_{\bf K}Ext^d_A(_A{\bf K},_AA)=1,$ $Ext^i_A(_A{\bf
K},_AA)=0$ for all $i\neq d$,

(AS3) the right $A$-module versions of the conditions (AS1,AS2)
hold.
%%%%%%%%%%%%%%%%%%%%%%%%%%%%%%%%%%%%%%%%%%%%%%%%%%%%%%%%%%%%%%%%%%%%%%%%%%%%%

\begin{theorem} $H_{\beta}$ is an AS-Gorenstein Hopf algebra of injective dimension 3.
\end{theorem}
%%%%%%%%%%%%%%%%%%%%%%%%%%%%%%%%%%%%%%%%%%%%%%%%%%%%%%%%%%%%%%%%%%%%%%%%%%%%%%%%
\begin{proof} Let $T={\bf K}[a^{\pm n},b^{\pm 1},c^{\pm 1}]$. Then $H_{\beta}$ is finitely generated module over the Noetherian ring $T$. So $H_{\beta}$ is both right and left Noetherian. By \cite[Corollary 13.1.13]{CR}, $H_{\beta}$ is a $PI$ Hopf algebra. Moreover, by \cite[\ Propositions \ 8.2.9,\ 8.2.13 and 8.1.15]{CR} we obtain that
$$\begin{array}{llll}GKdim(H_{\beta})&=&GKdim({\bf K}[a^{\pm n},b^{\pm 1},c^{\pm 1}]),\\
&=&GKdim({\bf K}[a^{n},b,c]), \\
&=&3.\end{array}$$
By \cite[Theorem 0.1]{WZ}, $H_{\beta}$ is an AS-Gorenstein  Hopf algebra of injective dimension $3$.
\end{proof}
%%%%%%%%%%%%%%%%%%%%%%%%%%%%%%%%%%%%%%%%%%%%%%%%%%%%%%%%%%%%%%%%%%%%%%
 We recall the homological integral of $H$ from \cite{LWZ}. Let $H$ be an $AS$-Gorenstein Hopf algebra of injective dimension $d$ over field ${\bf K}$.
The vector space of left homological integrals of $H$ is defined as $\int_H^l:=Ext_H^d(_H{\bf K},_HH)$ and the right homological integrals of $H$  as $\int_H^r:=Ext_H^d({\bf K}_H,H_H)$.
Besides, homological integrals agree with the classical integrals for finite dimensional Hopf algebra in \cite[section 1]{LWZ}. Then we get the following  corollary.

\begin{corollary} $Ext^3_{H_{\beta}^{op}}({\bf K}_{H_\beta},{H_{\beta}}_{H_\beta})\simeq Ext^3_{H_{\beta}}(_{H_{\beta}}{\bf K},_{H_{\beta}}H_{\beta})\simeq {\bf K}$ as two sided $H_{\beta}$-modules.
\end{corollary}

\begin{proof}Let $H^{\prime}=H_{\beta}/(b-1)$,$H^{\prime\prime}=H^{\prime}/(c-1)$ and $H^{\prime\prime\prime}=H^{\prime\prime}/(a^{n(n-1)}-1)$. Then $H^{\prime\prime\prime}$ is isomorphic to a Gelaki's Hopf algebra ${\mathcal U}_{(n,n(n-1),n_1,q,\beta_1,\beta_2,\beta_3)}$. Since
$a^{(n-1)n}-1,b-1,$ $c-1$ are in the center of $H_{\beta}$, we get $\int^l_{H_{\beta}}=\int^l_{H^{\prime}}=\int^l_{H^{\prime\prime}}=\int^l_{H^{\prime\prime\prime}}$ by \cite[Lemma 2.6]{LWZ}.  
Since $H'''$ is finite dimensional, by \cite[Proposition 3.2]{G} we obtain that $\int^l_{H^{\prime\prime\prime}}={\bf K}\lambda$, where
$\lambda=\frac1{n(n-1)}(\sum\limits_{i=0}^{n(n-1)-1}a^i)x^{n-1}y^{n-1}$. It follows that ${\bf K}\lambda\simeq {\bf K}$ as $H_{\beta}$-bimodules, since $k\cdot\lambda=\varepsilon(k)\lambda$ for any $k\in H_\beta$ the same action as on $1\in{\bf K}$.
\end{proof}

%%%%%%%%%%%%%%%%%%%%%%%%%%%%%%%%%%%%%%%%%%%%%%%%%%%%%%%%%%%%%%%%%%%%%%%%%%%%%
                       %Properties 
%%%%%%%%%%%%%%%%%%%%%%%%%%%%%%%%%%%%%%%%%%%%%%%%%%%%%%%%%%%%%%%%%%%%%%%%%%%%%

\section{Properties of $H_{\alpha, \beta}$}
In this section, we construct and illustrate some properties of the Hopf algebra $H_{\alpha,\beta}$, which is a quotient of $H_{\beta}$. 
It is necessary to determine all irreducible representations of $H_{\beta}$ in the next section.

%%%%%%%%%definition of H_ab

\begin{definition}
Assume that  $N=mn(n-1)$, where $m\geq 1$. Let $\alpha=(n,m,n_1,n_2,n_3)\in
{\bf N}^5$, where $1\leq n_1<m(n-1)$, $0\leq n_2,n_3<n-1$. Let $I$ be an ideal of $H_{\beta} $ generated by $a^N-1, $ $b-a^{mnn_2}$ and $c-a^{mnn_3}$.
Let $$H_{\alpha,\beta}=H_{\beta} /I.$$  Then $H_{\alpha,\beta}$ is a Hopf algebra
generated by $a,x,y$ satisfying
$$a^{mn(n-1)}=1, \quad   x^{n}=\beta_1(a^{n_1n}-a^{mnn_2}),\quad 
y^{n}=\beta_2(a^{n_1n}-a^{mnn_3}),$$ $$  xa=qax,\quad    ya=q^{-1}ay,\quad 
yx-q^{-n_1}xy=\beta_3(a^{2n_1}-a^{mn(n_2+n_3)}).$$
The coalgebra structure of $H_{\alpha,\beta}$ is determined by
$$\Delta(a)=a\otimes a,\quad    \Delta(x)=x\otimes a^{n_1}+a^{mnn_2}\otimes
x,$$
$$\Delta(y)=y\otimes a^{n_1}+a^{mnn_3}\otimes
y,\quad   \varepsilon(a)=1,\quad    \varepsilon(x)=\varepsilon(y)=0.$$ The antipode  of $H_{\alpha,\beta}$ is determined by
$$s(a)=a^{-1}=a^{N-1},\quad s(x)=-q^{-n_1}a^{-mnn_2-n_1}x,\quad  s(y)=-q^{n_1}a^{-mnn_3-n_1}y.$$
\end{definition}

Besides, we can prove that $\lambda=\sum\limits_{i=0}^{N-1}\frac{a^i}{N}x^{n-1}y^{n-1}$ is a two-sided integral of $H_{\alpha,\beta}$. Indeed,
by $yx-q^{-n_1}xy=\beta_3(a^{2n_1}-a^{mn(n_2+n_3)})$ we get that
\begin{eqnarray}y^kx=q^{-kn_1}xy^k+\beta_3u_k(q^{-(k-1)n_1}a^{2n_1}-a^{mn(n_2+n_3)})y^{k-1},\end{eqnarray}
where $u_k=q^{-(k-1)n_1}+\cdots+q^{-n_1}+1.$ 

%Since $\varepsilon(\lambda)=0$, $H_{\alpha,\beta}$ is not semisimple. It
%is easy to verify that $s^2$ is an inner automorphism of
%$H_{\alpha,\beta}$ determined by $a^{n_1}$. Since $H_{\alpha,\beta}$ is
%unimodular and $s^2$ is an inner,  $H_{\alpha,\beta}$ is a symmetric
%algebra by \cite{K,L}, i.e., there exists a non-degenerate associative and symmetric bilinear form $(-,-):H_{\alpha,\beta}\times H_{\alpha,\beta}\rightarrow {\bf K}$.

%%%%%%%%%%%%%%%%%%%%%%%%%%%%%%%%%%%%%%%%%%%%
Let $\alpha=(n,m,n_1,n_2,n_3),\alpha'=(n',m',n_1',n_2',n_3')$,
$\beta=(1,1,\beta_3)$ and $\beta'=(1,1,\beta_3')$. Suppose that  $\beta_1\beta_2\neq 0$. Then we can assume that $\beta_1=\beta_2=1$.  Moreover, we assume that $x^n\neq 0$, $y^n\neq 0$ and $yx\neq q^{-n_1}xy$. Under these assumptions, we have the following proposition.

\begin{proposition}  $H_{\alpha,\beta}\simeq H_{\alpha',\beta'}$ if and only if $\alpha=\alpha'$ and $\beta_3=c\beta_3'$,
where $c^n=1$.
\end{proposition}

\begin{proof} Let $f$ be an isomorphism  from $H_{\alpha,\beta}$ to $H_{\alpha',\beta'}$ and $H_{\alpha',\beta'}$ be a Hopf algebra generated
by  $g=a+I'$,  $z=x+I'$ and $w=y+I'$, where $I'$ is the ideal generated by  $a^{n'(n'-1)m'}-1, $ $b-a^{n'm'n'_2}$ and $c-a^{m'n'n'_3}$.
Notice that the subset $G(H_{\alpha,\beta})$ of all group-like elements of $H_{\alpha,\beta}$ is equal to
$\{a^t|0\leq t\leq n(n-1)m-1\}$, and the subset $G(H_{\alpha',\beta'})$ of all group-like elements of $H_{\alpha',\beta'}$ is equal to
$\{g^t|0\leq t\leq n'(n'-1)m'-1\}$. Since
$f$ induces an isomorphism from $G(H_{\alpha,\beta})$ to
$G(H_{\alpha',\beta'})$ and $mn(n-1)=m'n'(n'-1)$. Moreover, we can assume that $f(a)=g$. Since $f(a)f(x)=qf(x)f(a)$, $f(x)=\sum\limits_{i=0}^{N-1}\sum\limits_{j=1}^{n'-1}x_{ij}g^iz^jw^{j-1}$ for some $x_{ij}\in {\bf K}$.
By $\Delta(f(x))=(f\otimes f)\Delta(x)$, we get $f(x)=x_{01}z$. It is obvious that $x_{01}\neq 0$.
Similarly, we can prove that $f(y)=uy$ for some nonzero $u\in {\bf
K}$. Since $f(s(x))=s(f(x))$, we have $q^{-n_1}a^{-mnn_2-n_1}x=q^{-n'_1}a^{-m'n'n'_2-n'_1}x$. Then $mnn_2=m'n'n'_2$ and $n_1=n'_1$. Since $x^n=a^{nn_1}-a^{mnn_2}$, we have
$$x_{01}^n(g^{n'n_1'}-g^{m'n'n'_2})=g^{nn_1}-g^{mnn_2} \quad \text{i.e.} \quad (x_{01}^ng^{n'-n}-1)g^{nn_1-mnn_2}=x_{01}^n-1.$$ Since $g^{nn_1-mnn_2}\neq0$, we have $x_{01}^n=1$ and $n=n'$. Thus $m=m'$ and $n_2=n'_2$. Similarly, we can prove that $u^n(g^{n'n_1'}-g^{m'n'n'_3})=g^{nn_1}-g^{mnn_3}$.
Hence $n_3=n_3'$ and $u^n=1$. Notice that
by applying the
function $f$ to the equation $yx-q^{-n_1}xy=\beta_3(a^{2n_1}-a^{mn(n_2+n_3)})$, we obtain the equation
$\beta_3=ux_{01}\beta_3'.$
\end{proof}

\begin{remark}
	The Hopf algebra $H_{\alpha,\beta}$ is finite dimensional and by Theorem~\ref{thm-basis} we get 
	$$dim_{\bf K} H_{\alpha,\beta}=n^3(n-1)m.$$ 
Especially, if $n_2=n_3=0$ then $H_{\alpha,\beta}$ is Gelaki's Hopf algebra by Remark \ref{rem-Gelaki}.
\end{remark}

Next, we determine the algebraic structure of $H_{\alpha,\beta}$. For this purpose, we construct some idempotent elements of $H_{\alpha,\beta}$. Suppose that  $\omega_0$ is a primitive $m(n-1)$-th root of unity and $$e_i=\frac1{m(n-1)}\sum_{j=0}^{m(n-1)-1}(\omega_0^ia^n)^j,\qquad i=0,1,\cdots,m(n-1)-1.$$ Note that $\sum\limits_{j=0}^{m(n-1)-1}(\omega_0^i)^j=0$ for any
$i\neq 0$. Thus  we have the following claim.

\begin{lemma} \label{lem41}$1=e_0+e_1+\cdots +e_{m(n-1)-1}$ is a sum of central idempotent elements.\end{lemma}

\begin{proof} Since $a^nx=xa^n,a^ny=ya^n$, $e_i$ are in the center
of $H_{\alpha,\beta}$. Hence
$$\begin{array}{lll}\sum\limits_{i=0}^{m(n-1)-1}e_i&=&\frac1{m(n-1)}\sum\limits_{i=0}^{m(n-1)-1}\sum\limits_{j=0}^{m(n-1)-1}\omega_0^{ij}a^{nj}\\
&=&\frac1{m(n-1)}\sum\limits_{j=0}^{m(n-1)-1}a^{nj}(\sum\limits_{i=0}^{m(n-1)-1}(\omega_0^j)^i)\\
&=&\frac1{m(n-1)}m(n-1)=1\end{array}$$and
$$\begin{array}{lll}e_ie_j&=&\frac1{(m(n-1))^2}
\sum\limits_{k=0}^{m(n-1)-1}\sum\limits_{l=0}^{m(n-1)-1}\omega_0^{ik+jl}a^{n(k+l)}\\
&=&\frac1{(m(n-1))^2}\sum\limits_{u=0}^{2(m(n-1)-1)}a^{nu}\omega_0^{ui}\sum\limits_{0\leq
l\leq min\{u,m(n-1)-1\},0\leq u-l\leq m(n-1)-1}\omega_0^{(j-i)l}\\
&=&\frac1{(m(n-1))^2}(\sum\limits_{u=0}^{(m(n-1)-1)}a^{nu}\omega_0^{ui}\sum\limits_{0\leq
l\leq
u}\omega_0^{(j-i)l}\\
&&+\sum\limits_{u=m(n-1)}^{2(mn-m-1)}a^{nu}\omega_0^{iu}\sum\limits_{u-(mn-m-1)\leq
l\leq (mn-m-1)}\omega_0^{(j-i)l})\\
&=&\frac1{(m(n-1))^2}(\sum\limits_{u=0}^{(m(n-1)-1)}a^{nu}\omega_0^{ui}\sum\limits_{0\leq
l\leq
u}\omega_0^{(j-i)l}\\
&&+\sum\limits_{u'=0}^{mn-m-1}a^{nu'}\omega_0^{iu'}\sum\limits_{1+u'\leq
l\leq {mn-m-1}}\omega_0^{(j-i)l})\\
&=&\frac1{(m(n-1))^2}(\sum\limits_{u=0}^{(m(n-1)-1)}a^{nu}\omega_0^{ui}\sum\limits_{0\leq
l\leq {mn-m-1}}\omega_0^{(j-i)l})\\
&=&\frac1{m(n-1)}\delta_{ij}\sum\limits_{u=0}^{m(n-1)-1}a^{nu}\omega_0^{iu}
=\delta_{ij}e_i,\end{array}$$where $\delta_{ij}$ is the
Kronecker symbol. The claim is true.\end{proof}

%%%%%%%%%definition of H_ab
%%%%%%%%%definition of H_ab
%%%%%%%%%definition of H_ab

Let $A_i=e_iH_{\alpha,\beta}$, where $i\in \{0,1,\cdots,m(n-1)-1\}$. Since
$e_ia^n=\omega_0^{-i}e_i$, we obtain that the set $\{e_ia^jx^ly^t|0\leq j,l,t\leq n-1\}$ is
a basis of $A_i$ and $dim_{\bf{K}}A_i=n^3$. It follows that $H_{\alpha,\beta}$ is a direct sum of algebras $A_i$ by Lemma \ref{lem41}.
 
 \begin{lemma} Let $\omega$ be a primitive $N$-th  root of unity such that $\omega^n=
\omega_0$. Then $A_i$ is generated by $g=\omega^ie_ia,x',y'$,
satisfying the following relations
$$g^n=e_i,\quad   x'^n=\beta_1'e_i,\quad    y'^n=\beta_2'e_i,\quad  gx'=q^{-1}x'g,\quad  gy'=qy'g$$
and
$$y'x'-q^{-n_1}x'y'=\left\{\begin{array}{ll}0&\beta_3=0\\
g^{2n_1}-\omega^{2n_1i-m(n_2+n_3)i}e_i&\beta_3\neq0\end{array}\right.,$$
where $\beta_1'=\beta_1(\omega_0^{-n_1i}-\omega_0^{-mn_2i}),\ $
$\beta_2'=\left\{\begin{array}{ll}\beta_2(\omega_0^{-n_1i}-\omega_0^{-mn_3i})&\beta_3=0\\
\beta_2\beta_3^{-n}(\omega_0^{n_1i}-\omega_0^{(2n_1-mn_3)i})&\beta_3\neq0\end{array}\right.$.
\end{lemma}
\begin{proof}Let $x'=xe_i$ and $y'=\left\{\begin{array}{ll}ye_i&\beta_3=0\\
\beta_3^{-1}\omega^{2n_1i}ye_i&\beta_3\neq 0\end{array}\right.$.
It is obvious that $A_i$ is generated by $g,x',y'$. It is easy to
check the relations on these generators $g, x', y'$ in this lemma.
\end{proof}

For the rest of this section, we illustrate the properties of $A_i$.
For simplification, we denote $A$ as an algebra
generated by $g,x,y$, satisfying
\begin{equation}\label{eq-yn}
	g^n=1,\quad   x^n=\beta_1',\quad    y^n=\beta_2',\quad gx=q^{-1}xg,\quad   gy=qyg
\end{equation}
and
$$yx-q^{-n_1}xy=\left\{\begin{array}{ll}0&\beta_3=0\\
g^{2n_1}-\omega^{2n_1i-mn(n_2+n_3)i}&\beta_3\neq0\end{array}\right..$$
Let $f_i=\frac1n\sum\limits_{j=0}^{n-1}(q^ig)^j,i=0,1,\cdots,n-1$. Then
$$f_0+f_1+\cdots+f_{n-1}=1$$ and $f_if_j=\delta_{ij}$. Thus
$$A=Af_0 +Af_1+\cdots+Af_{n-1}$$as a direct sum of left ideals.
\begin{lemma}\label{lem43} Suppose that  either $\beta_1'\neq 0$ or $\beta_2'\neq 0$. Then A is isomorphic to $$M_n(R_1)\oplus M_n(R_2)\oplus\cdots\oplus
M_n(R_t),$$ where $R_i$ are commutative local rings.\end{lemma}

\begin{proof} It is easy to verify that $f_ix=xf_{i-1}$ and
$f_iy=yf_{i+1}$. If $\beta_1'\neq 0$, then $Af_i$ is isomorphic to
$Af_{i-1}$ as left $A$-modules by multiplying $x$ from the right.
If $\beta_2'\neq 0$, then $Af_i$ is isomorphic to $Af_{i+1}$ as
left $A$-modules by multiplying $y$ from the right. Thus $A$ is
isomorphic to $Hom_A(A,A)\simeq M_n(End_A(Af_i))$. Since
$$f_ig=\frac1n\sum_{j=1}^nq^{ij}g^{j+1}=\frac1nq^{-i}\sum_{j=1}^nq^{i(j+1)}g^{j+1}=q^{-i}f_i,$$
$End_A(Af_i)=f_iAf_i=span\{f_ix^ty^tf_i|t=0,1,\cdots,n-1\}$.

Notice that
$f_lyxf_l-q^{-n_1}f_lxyf_l=
(q^{-2n_1l}-\omega_0^{2n_1i-m(n_2+n_3)i})f_l$ when $\beta_3\neq0$.
Let
$\gamma_l=0$ if $\beta_3=0$ and $\gamma_l=q^{-2n_1l}-\omega_0^{2n_1i-m(n_2+n_3)i}$ if $\beta_3\neq0$.
 Then
 $$\begin{array}{lll}(f_ixyf_i)^2&=&f_ixf_{i-1}yxf_{i-1}yf_i\\
 &=&f_ix(q^{-n_1}f_{i-1}xyf_{i-1}+\gamma_{i-1}f_{i-1})yf_i\\
 &=&q^{-n_1}f_ix^2y^2f_i+\gamma_{i-1} f_ixyf_i.\end{array}$$
 Similarly, we can prove that
 $(f_ixyf_i)^k=q^{-(k-1)n_1}f_ix^ky^kf_i+b_kf_ix^{k-1}y^{k-1}f_i+\cdots+b_1$
 for $k\geq 2$. Hence $R=End_A(Af_i)$ is generated by
$f_ixyf_i$ over the field ${\bf K}$ as an algebra. Since $dim_{\bf{K}}
R=n$, there exists a minimal polynomial $p(x)$ such that
$p(f_ixyf_i)=0$. Therefore, $R=R_1\oplus\cdots\oplus R_t$ is a
direct sum of local rings.\end{proof}

\begin{lemma}\label{lem44} Suppose that $\beta_1'=\beta_2'=\beta_3=0$. Then the Jacobson radical $J(A)$ of $A$ is
spanned by $\{g^ix^jy^k|0\leq i,j,k\leq n-1,j+k>0\}$
and $A/J(A)\cong {\bf K}[{\mathbb Z}_n]$.
\end{lemma}

\begin{proof} Since the ideal $I=\langle x,y\rangle$ is nilpotent, we have $I\subseteq J(A).$ Besides, there is a $\bf{K}$ algebra isomorphism $A/I={\bf{K}}[g]/\langle g^n-1 \rangle\cong {\bf{K}}[{\mathbb Z}_{n}]$. It follows that $J(A)=I.$
\end{proof}

\begin{lemma} \label{lem45}Suppose that $\beta_1'=\beta_2'=0$ and $\beta_3\not=0$. Then $A$ is generated by $K,E,F$ with relations
$$K^n=1,\  E^n=F^n=0,\   EK=q^{-1}KE,\   FK=qKF,$$
$$EF-FE=K^{2n_1}\frac{K^{n_1}-\rho K^{-n_1}}{q^{n_1}-q^{-n_1}},$$
where $\rho=\omega_0^{2n_1i-mn(n_2+n_3)i}$.
\end{lemma}
\begin{proof} Let $E=\frac 1{q^{n_1}-1}yg^{n-n_1},
F=\frac1{q^{n_1}+1}x$, $K=g$, and $K^{-1}=g^{n-1}$. It is obvious that
$A$ is generated by $E,F,K$. Thus we have $EF-FE=K^{2n_1}\frac{K^{n_1}-\rho K^{-n_1}}{q^{n_1}-q^{-n_1}}$ from $yx-q^{-n_1}xy=g^{2n_1}-\omega^{2n_1i-mn(n_2+n_3)i}.$ 
By $E=\frac 1{q^{n_1}-1}yg^{n-n_1}$, then
$E^n=(\frac{1}{q^{n_1}-1}yg^{n-n_1})^n=\frac{1}{(q^{n_1}-1)^n}q^{\frac{n(n-1)(n-n_1)}{2}}y^{n}g^{(n-n_1)n}$. Hence $E^n=0$ by the relation (\ref{eq-yn}) $y^n=\beta_2'=0$ of $A$.
The other relations hold trivially.
\end{proof}

Notice that a weak Hopf algebra related to the algebra $A_i$ in Lemma \ref{lem45} has been studied in \cite{YW}. Especially, if $N|[2n_1i-mn(n_2+n_3)i]$ then the algebra $A$ is isomorphic to Radford's Hopf algebra $U_{(n,n_1,\omega)}$.

%%%%%%%%%%%%%%%%%%%%%%%%%%%%%%%%%%%%%%%%%%%%%%%%%%%%%%%%%%%%%%%%%%%%%%%%%%%%%
                       %Section 4 : irreducible representations
%%%%%%%%%%%%%%%%%%%%%%%%%%%%%%%%%%%%%%%%%%%%%%%%%%%%%%%%%%%%%%%%%%%%%%%%%%%%%

\section{Irreducible representations of $H_{\beta}$}

In this section, we  determine all irreducible  representations of $H_{\beta}$. The following key Lemma is necessary.

\begin{lemma} \label{lem51}
Every irreducible representation of $H_{\beta}$ is finite-dimensional.
\end{lemma}
\begin{proof} The proof is similar to that of \cite[Corollary 1.2]{W}. For the sake of completeness, we give a sketch. Let $M$ be a simple module over $H_{\beta}$ and $T:={\bf
K}[a^{\pm n},b^{\pm1},c^{\pm1}]$. Since $T$ is contained in the center
of $H_{\beta}$, every element $t$ of $T$ induces an endomorphism of the
$H_{\beta}$-module $M$. We denote $\varphi(t)$ as the endomorphism
of $M$ induced by $t\in T$. Suppose that ${\mathcal K}=\{\varphi(t)|t\in
T\}$. Then $kM=M$ for any nonzero $k\in {\mathcal K}$. Since $H_{\beta}$ is finitely
generated over $T$, $M$ is a finitely generated ${\mathcal
K}$-module. Let $m_1,m_2,\cdots,m_r$ be the generators of $M$. 
By $kM=M$, we obtain
$$\left\{\begin{array}{l}
m_1=ka_{11}m_1+ka_{12}m_2+\cdots+ka_{1r}m_r\\
m_2=ka_{21}m_1+ka_{22}m_2+\cdots+ka_{2r}m_r\\
 \quad  \quad    \quad  \quad   \vdots\\
m_r=ka_{r1}m_1+ka_{r2}m_2+\cdots+ka_{rr}m_r
\end{array}\right.$$
for some $a_{ij}\in {\mathcal {K}}$.
Then $det(I-k(a_{ij}))=0$, where $I$ is the identity matrix with
order $r$. Thus, there exists $h\in{\mathcal
K}$ such that $kh=1$. Hence ${\mathcal K}$ is a field. Notice that
${\mathcal K}={\bf K}[\varphi(a^{\pm n}),\varphi(b^{\pm
1}),\varphi(c^{\pm 1})]$. It is an algebraic extension of ${\bf K}$.
Since ${\bf K}$ is an algebraically closed field, ${\bf
K}={\mathcal K}$. Consequently, $M$ is  finite-dimensional.
\end{proof}

%%%%%%%%%%%%%%%%%%%%%%%%%%%%%%%%%%%%%%%%%%%%%%%%%%%%%%%%%%%%%%%%%%%%%%%%%%%%%%%%
From the proof of Lemma \ref{lem51}, we can construct an algebra homomorphism  $\lambda_M\in
Hom_{\bf K}(T,{\bf K})$ for any  simple module $M$. 
Let
$\lambda_M(a^n)=\gamma_1,\lambda_M(b)=\gamma_2$ and $\lambda_M(c)=\gamma_3$.
Since $a,b,c$ are invertible in $T$,   $\gamma_i\neq 0$ for $i=1,2,3$.
 Then $M$ can be
viewed as a module over $H_M:=H_{\beta}/(a^n-\gamma_1,b-\gamma_2,c-\gamma_3)$.
Let $a'=\frac1{\sqrt[n]{\gamma_1}}a$,\
$b'=\frac1{\sqrt[n]{\gamma_1^{n_1}}}b$,\
$c'=\frac1{\sqrt[n]{\gamma_1}^{n_1}}c$,\ $x'=x$ and
$y'=\left\{\begin{array}{ll}y&\beta_3=0\\
\beta_3^{-1}\gamma_1^{-\frac{2n_1}n}y&\beta_3\neq
0\end{array}\right.$. Then $H_{\beta}$ is generated by $a',b',c',x',y'$
with the relations
$$a'b'=b'a',\quad  a'c'=c'a',\quad   b'c'=c'b',\quad  x'a'=qa'x',\quad  y'a'=q^{-1}a'y',\quad   b'x'=x'b',\quad $$ $$  c'x'=x'c', \quad   b'y'=y'b',\quad    c'y'=y'c',\quad x'^n=\beta_1'(a'^{nn_1}-b'^n),\quad   y'^n=\beta_2'(a'^{nn_1}-c'^n),$$
$$y'x'-q^{-n_1}x'y'=\left\{\begin{array}{ll}0&\beta_3=0\\
a'^{2n_1}-b'c'&\beta_3\neq 0\end{array}\right.,$$
where
$\beta_1'=\gamma_1^{n_1}\beta_1,$\ $\beta_2'=\left\{\begin{array}{lll}\gamma_1^{n_1}\beta_2&&\beta_3=0\\
\beta_3^{-n}\gamma_1^{-n_1}\beta_2&&\beta_3\neq
0\end{array}\right.$.  Thus the generators $a',x',y'$ in
$H_M$ satisfy
$$x'a'=qa'x',\quad   y'a'=q^{-1}a'y',\quad   a'^n=1,\quad   x'^n=\beta_1'(1-\frac{\gamma_2^n}{\gamma_1^{n_1}}),\quad
y'^n=\beta_2'(1-\frac{\gamma_3^n}{\gamma_1^{n_1}}),$$
$$y'x'-q^{-n_1}x'y'=\left\{\begin{array}{ll}0&\beta_3=0\\
a'^{2n_1}-\frac{\gamma_2\gamma_3}{\sqrt[n]{\gamma_1^{2n_1}}}&\beta_3\neq
0\end{array}\right..$$

In the following we illustrate all irreducible representations of $H_\beta$.
 
\begin{lemma} \label{lem52}
Let $\beta_1^{\prime\prime}=\beta_1'(1-\frac{\gamma_2^n}{\gamma_1^{n_1}}),\
\beta_2^{\prime\prime}=\beta_2'(1-\frac{\gamma_3^n}{\gamma_1^{n_1}}),$ and
$$\beta_3^{\prime\prime}(i)=\left\{\begin{array}{ll}0&\beta_3=0\\
\gamma_1^{-\frac{2n_1}n}(\gamma_1^{\frac{2n_1}n}q^{2n_1i}-\gamma_2\gamma_3)&\beta_3\neq
0\end{array}\right.$$
\begin{itemize}
	\itemsep=0pt
	\item [(I)]Suppose that $\beta_1^{\prime\prime}\neq 0$. Then the irreducible
$H_\beta$-module $M$ has a basis  $\{m_0,m_1,\cdots,m_{n-1}\}$ such that the actions of $a,b,c,x,y$ on $M$ with respect to this basis are given by  
\begin{eqnarray}\label{eq8}am_j=\sqrt[n]{\gamma_1}q^{i-j}m_j,\quad
bm_j=\gamma_2m_j,\  cm_j=\gamma_3m_j\qquad for\quad 0\leq j\leq n-1;
\end{eqnarray}
\begin{eqnarray}\label{eq9}xm_j=m_{j+1},\qquad for\quad 0\leq j\leq n-2;\qquad
xm_{n-1}=(\gamma_1^{n_1}-\gamma_2^n)\beta_1m_0;
\end{eqnarray}
\begin{eqnarray}\label{eq10}ym_j=k_{n-j+1}m_{j-1},\qquad for\quad 1\leq j\leq
n-1;\qquad ym_{0}=k_1m_{n-1},
\end{eqnarray}
where $k_1,\cdots,k_n $ are determined by $$k_l=q^{(l-1)n_1}\beta_1(\gamma_1^{n_1}-\gamma_2^n)k_1+\sum^{n-l}_{j=0}q^{-jn_1}\beta_3(\gamma_1^{\frac{2n_1}{n}}q^{(2i+l)n_1}-\gamma_2\gamma_3)$$ for $2\leq l\leq n$ and $k_1k_2\cdots k_n=(\gamma_1^{n_1}-\gamma_3^n)\beta_2$. We denote this irreducible  module $M$ by $V_I(\gamma_1,\gamma_2,\gamma_3;i)$.

	\item [(II)]
	Suppose that $\beta_{1}^{\prime\prime}=0$ and $\beta_2^{\prime\prime}\neq 0$. Then the irreducible
$H_\beta$-module $M$ has a basis  $$\{m_0,m_1,\cdots,m_{n-1}\}$$ such that the actions of $a,b,c,x,y$ on $M$ with respect to this basis satisfying
\begin{eqnarray}\label{eq11}
am_j=\sqrt[n]{\gamma_1}q^{i+j}m_j,\quad
bm_j=\gamma_2m_j,\quad  cm_j=\gamma_3m_j\qquad for\quad 0\leq j\leq n-1;
\end{eqnarray}
\begin{eqnarray}\label{eq12}
ym_j=m_{j+1},\qquad for\quad 0\leq j\leq n-2;\qquad
ym_{n-1}=(\gamma_1^{n_1}-\gamma_3^n)\beta_2m_0;
\end{eqnarray}
\begin{eqnarray}\label{eq13}
xm_j=k_{j}m_{j-1},\qquad for\quad 1\leq j\leq
n-1;\qquad xm_0=k_nm_{n-1},
\end{eqnarray}
where $k_1,\cdots,k_n $ are determined by $$k_l=q^{ln_1}(\beta_2(\gamma_1^{n_1}-\gamma_3^n)k_n-\sum^{l-1}_{j=0}q^{-jn_1}\beta_3(\gamma_1^{\frac{2n_1}{n}}q^{(2i+l-1)n_1}-\gamma_2\gamma_3))$$
 for $1\leq l\leq n-1$ and $k_1k_2\cdots k_n=(\gamma_1^{n_1}-\gamma_2^n)\beta_1$.
We denote this irreducible  module $M$ by $V_{II}(\gamma_1,\gamma_2,\gamma_3;i)$.

\end{itemize}
  
\end{lemma}
\begin{proof}The proof of (I) is similar to that of (II). We only
give the proof of (II). Similarly to the proof of Lemma \ref{lem43}, we can
prove that $H_M$ is a direct sum of $M_n(R_t)$, where $R_t$ are
commutative local rings. Hence every irreducible module over $H_M$
is of dimension $n$. Since $H_\beta$ module $M$ can be viewed as a $H_M$ module, $dim(M)=n$. Let $m_0$ be an eigenvector of $a'$ with
eignvalue $q^i$ for some $0\leq i\leq n-1$. Since
$y'^nm_0=\beta_2^{\prime\prime}m_0\neq 0$,
$\{m_0,y'm_0,\cdots,y'^{n-1}m_0\}$ is a basis of $M$. Moreover
$a'x'm_0=q^{-1}x'a'm_0=q^{i-1}x'm_0$. Therefore $x'm_0=k'_ny'^{n-1}m_0$
for some $k'_n\in{\bf K}$. Thus $M$ is a $H_\beta$ module with a basis $\{m_0,ym_0,\cdots,y^{n-1}m_0\}$ and the action of $a$ and
$y$ on $M$ with respect to this basis is given by (\ref{eq11}) and (\ref{eq12}) respectively. Meanwhile, $xm_0=k_ny^{n-1}m_0$
for some $k_n\in{\bf K}$. Since $xa=qax$ and
$y^lx-q^{-ln_1}xy^l=\beta_3\sum^{l-1}_{j=0}q^{-jn_1}(q^{-(l-1)n_1}a^{2n_1}-bc)y^{l-1}$, the action of $x$ can be
realized by $${\mathcal
X}=\left(\begin{array}{ccccc}0&k_1&0&\cdots&0\\
0&0&k_2&\dots&0\\
\vdots&\vdots&\vdots&\ddots&\vdots\\
0&0&0&\dots&k_{n-1}\\
k_n&0&0&\cdots&0\end{array}\right),$$ where
$$k_l=q^{ln_1}(\beta_2(\gamma_1^{n_1}-\gamma_3^n)k_n-\sum^{l-1}_{j=0}q^{-jn_1}\beta_3(\gamma_1^{\frac{2n_1}{n}}q^{(2i+l-1)n_1}-\gamma_2\gamma_3)).$$
Since ${\mathcal X}^n=k_1k_2\cdots k_nI_n=\beta_1(\gamma_1^{n_1}-\gamma_2^n)I_n$, $k_n$ is a root of the equation
$$q^{\frac{n(n+1)}{2}n_1}\prod^n_{l=1}(\beta_2(\gamma_1^{n_1}-\gamma_3^n)k_n-\sum^{l-1}_{j=0}q^{-jn_1}\beta_3(\gamma_1^{\frac{2n_1}{n}}q^{(2i+l-1)n_1}-\gamma_2\gamma_3))=\beta_1(\gamma_1^{n_1}-\gamma_2^n).$$
Hence, the claim is true.
\end{proof}

\begin{lemma} \label{lem53} Suppose that $\beta_1^{\prime\prime}=\beta_2^{\prime\prime}=
\beta_3^{\prime\prime}(i)=0$ for some $0\leq i\leq n-1$. Then the
irreducible $H_\beta$ module $M={\bf K}$ is determined by $a\cdot 1=\sqrt[n]{\gamma_1}q^i,\ b\cdot 1=\gamma_2,\   c\cdot1=\gamma_3$ and
$x\cdot1=y\cdot1=0$ for some $0\leq i\leq n-1$. We denote this irreducible  module $M$ by $V_0(\gamma_1,\gamma_2,\gamma_3;i)$ in the sequel.
\end{lemma}

\begin{proof} We first consider $M$ as the irreducible $H_M$ module. Suppose that  $m_0$ is an eigenvector of $a'$. Since $x'^nm_0=0$, we can assume that
$x'm_0=0$. If $y'm_0\neq 0$, then there exist $j\in\{1,\cdots,n-1\}$ such that $y'^jm_0\neq 0$ and $y'^{j+1}m_0=0$. Hence $\{m_0, y'm_0,y'^2m_0,\cdots, y'^{j}m_0\}$ is a basis of the module $M$. $<y'^{j}m_0>$ is the submodule of $M$ by $x'y'^{j}m_0=q^{jn_1}y'^{j}x'm_0=0$. This gives a contradiction. Thus $x\cdot m_0=y\cdot m_0=0$ and $M={\bf K}m_0\cong {\bf K}$.
\end{proof}

Suppose that  $\beta_3(i)''=\beta_3(q^{2n_1i}-\gamma_1^{-\frac{2n_1}n}\gamma_2\gamma_3)\neq 0$. If there is an integer $v$ such that $$q^{(2i-v+1)n_1}= \gamma_1^{-\frac{2n_1}n}\gamma_2\gamma_3,$$ then there is a minimal positive integer $r$ such that $q^{(2i-r+1)n_1}= \gamma_1^{-\frac{2n_1}n}\gamma_2\gamma_3$. As $\beta_3(i)^{\prime\prime}\neq 0$, we have $q^{(2i-r+1)n_1}\neq q^{2n_1i}$ and $q^{(1-r)n_1}\neq 1$. Thus $2\leq r\leq t,$ where $t=|q^{n_1}|$ the order of $q^{n_1}$.
If  there exists no integer $v$ such that $q^{(2i-v+1)n_1}= \gamma_1^{-\frac{2n_1}n}\gamma_2\gamma_3$,
 then we define $r:=t$. 

\begin{lemma} \label{lem54}Suppose that $\beta_1^{\prime\prime}=\beta_2^{\prime\prime}=0$,
but $\beta_3^{\prime\prime}(i)\not=0$ for some  $0\leq i\leq n-1$. Then the irreducible $H_\beta$ module $M$ has a basis $\{m_i|0\leq i\leq r-1\}$ such that the actions of $a,b,c,x,y$ are given by
 \begin{eqnarray}\label{eq15}
am_j=\sqrt[n]{\gamma_1}q^{i-j}m_j,\quad
bm_j=\gamma_2m_j,\quad  cm_j=\gamma_3m_j\qquad for\quad 0\leq j\leq r-1\
\end{eqnarray}
\begin{eqnarray}\label{eq16}
ym_j=k_jm_{j-1},\qquad for\quad 1\leq j\leq r-1;\qquad
ym_{0}=0;
\end{eqnarray}
\begin{eqnarray}\label{eq17}
xm_j=m_{j+1},\qquad for\quad 0\leq j\leq
r-2;\qquad xm_{r-1}=0,
\end{eqnarray}
where $k_1,\cdots,k_r $ are determined by $$k_l=\sum^{l-1}_{j=0}q^{-jn_1}\beta_3(q^{(2i-l+1)n_1}\gamma^{\frac{2n_1}{n}}_1-\gamma_2\gamma_3)$$
 for $1\leq l\leq r-1.$
\end{lemma}
\begin{proof} We first consider $M$ as the irreducible $H_M$ module. Suppose that  $m_0$ is an eigenvector of $a'$. Since $y'^nm_0=0$, we can assume that
$y'm_0=0$. Moreover, we can assume that $x'^jm_0\neq 0$ for $0\leq j\leq
r'-1$ and $x'^{r'}m_0=0$. It is obvious that $r'\leq n$ and
$a'm_0=q^im_0$ for some $1\leq i\leq n$. It is easy to check that
$$y'x'^um_0=q^{-un_1}x'^uy'm_0+\frac{1-q^{-un_1}}{1-q^{-n_1}}(q^{(2i-u+1)n_1}-
\gamma_1^{-\frac{2n_1}n}\gamma_2\gamma_3)x'^{u-1}m_0.$$ Since
$$0=y'x'^{r'}m_0=\frac{1-q^{-{r'}n_1}}{1-q^{-n_1}}(q^{(2i-r'+1)n_1}-
\gamma_1^{-\frac{2n_1}n}\gamma_2\gamma_3)x'^{r'-1}m_0,$$ either $q^{r'n_1}=1$, or
$q^{(2i-r'+1)n_1}=\gamma_1^{-\frac{2n_1}n}\gamma_2\gamma_3$.   If $r'>t$, then
$span\{x'^tm_0,\cdots,x'^{r'-1}m_0\}$ is a nonzero submodule of the simple module $M$. This is impossible.
Thus $r'\leq t$. Similarly, we have $r\leq r'$. Hence
$r'=min\{t,r\}.$ Since
$y'x'^jm_0=(\sum\limits_{v=0}^{j-1}q^{(2i-v)n_1}-\sum\limits_{v=0}^{j-1}q^{vn_1}\gamma_1^{-\frac{2n_1}n}\gamma_2\gamma_3)x'^{j-1}m_0$
for all $j\geq 1$, $y'(x'^lm_0)=k'_l(x'^{l-1}m_0)$ for $1\leq l\leq r-1$, where $k'_l=\sum^{l-1}_{j=0}q^{-jn_1}(q^{(2i-l+1)n_1}-\gamma^{-\frac{2n_1}{n}}_1\gamma_2\gamma_3).$
Finally, we prove that
$M=span\{x'^vm_0|0\leq v \leq r-1\}$ is a simple $H_M$-module. Let $V'$ be a nonzero module of $M$. Then there exists
$x'^{i-s}m_0\in V'$ for some $0\leq s\leq r-1$. If $s=i$, then $m_0\in
V'$. If $s\neq i$, then $y'^{i-s}x'^{i-s}m_0=k'_1k'_2\cdots
k'_{i-s}m_0\in V'$. Thus $m_0\in V'$
and $V'=M$.
Since $a'=\frac1{\sqrt[n]{\gamma_1}}a$,
$b'=\frac1{\sqrt[n]{\gamma_1^{n_1}}}b$,
$c'=\frac1{\sqrt[n]{\gamma_1}^{n_1}}c$, $x'=x$, and
$y'=\beta_3^{-1}\gamma_1^{-\frac{2n_1}n}y$, the action of $a,b,c,x,y$ on $M$ is given by (\ref{eq15})(\ref{eq16})(\ref{eq17}).
\end{proof}

  Recall the modules $V_I(\gamma_1,\gamma_2,\gamma_3;i)$,\ $V_{II}(\gamma_1,\gamma_2,\gamma_3;i)$ and $V_{0}(\gamma_1,\gamma_2,\gamma_3;i)$  defined in Lemma~\ref{lem52} and Lemma~\ref{lem53}. 
  We denote the irreducible module described in Lemma \ref{lem54} by $V_r(\gamma_1,\gamma_2,\gamma_3;i)$.
 By all the above lemmas,  we obtain the following classification theorem.

\begin{theorem} Let $M$ be an irreducible representation of $H_{\beta}$. Then
$M$ is isomorphic to one of the following types:
$V_I(\gamma_1,\gamma_2,\gamma_3;i)$,  $V_{II}(\gamma_1,\gamma_2,\gamma_3;i)$,   $V_{0}(\gamma_1,\gamma_2,\gamma_3;i)$, and
$V_{r}(\gamma_1,\gamma_2,\gamma_3;i)$.
\end{theorem}
\begin{proof} It is easy to verify that $V_I(\gamma_1,\gamma_2,\gamma_3;i)$,
$V_{II}(\gamma_1,\gamma_2,\gamma_3;i)$,
$V_{0}(\gamma_1,\gamma_2,\gamma_3;i)$, and
$V_{r}(\gamma_1,\gamma_2,\gamma_3;i)$ are irreducible
representations of $H_{\beta}$.  Let $M$ be an irreducible representation of
$H_{\beta}$. Then there exists an algebra homomorphism  $\lambda\in Hom_{\bf{K}}(T,{\bf K})$ such that
$\lambda(a^n)=\gamma_1\neq 0$, $\lambda(b)=\gamma_2\neq 0$ and
$\lambda(c)=\gamma_3\neq 0$. Since $\gamma_1,\gamma_2,\gamma_3$
lie in one of the above four cases, $M$ can be viewed as a
representation of $H_M$ described by Lemma \ref{lem52}, Lemma \ref{lem53} and Lemma
\ref{lem54}. Being careful with the change of the generators,  $M$ is
isomorphic to one of the above four kinds of irreducible representations.
\end{proof}

Next proposition shows that the category of all $H_\beta$ modules is not semisimple.

\begin{proposition} Suppose that  $\gamma_1^{-\frac {2n_1}n}\gamma_2\gamma_3\neq q^{v n_1}$
for  any $v$ and
$(\gamma_1^{n_1}-\gamma_2^n)\beta_1=(\gamma_1^{n_1}-\gamma_3^n)\beta_2=0$.
Then
$Ext_{H_\beta}(V_{t}(\gamma_1,\gamma_2,\gamma_3;i-t+n),V_{t}(\gamma_1,\gamma_2,\gamma_3;i))\ne
0 $ for all $1\le i\le n-1$, where $t=|q^{n_1}|$.
\end{proposition}

\begin{proof} Let
$k_u(i)=\frac{1-q^{-un_1}}{1-q^{-n_1}}(q^{(2i-u+1)n_1}-\gamma_1^{-\frac{2n_1}n}\gamma_2\gamma_3)$ and $L$ be a vector
space with a basis $\{v_1,\cdots,v_{2t}\}$.  Set
$$\begin{array}{llll} av_p=q^{i-p+1}v_p,&
av_{t+p}=q^{i-r+n-p+1}v_{t+p},& & 1\le p\le t,\\
xv_{2t}=0,& xv_p= v_{p+1},& &1\leq p\leq 2t-1,\\
yv_1=0,& yv_{t+1}=0,&&\\
yv_p=k_{p-1}(i)v_{p-1},&yv_{t+p}=k_{p-1}(i-t+n)v_{t+p-1},& & 2\leq
p\leq t,\\
bv_p=\gamma_2v_p,&cv_p=\gamma_3v_p,&&1\le p\le 2t,
\end{array}$$
Then $span\{v_{t+1},\cdots,v_{2t}\}$ is isomorphic to $V_{t}(\gamma_1,\gamma_2,\gamma_3;i-t+n)$ and
$$L/V_{t}(\gamma_1,\gamma_2,\gamma_3;i-t+n)\cong V_{t}(\gamma_1,\gamma_2,\gamma_3;i).$$ Obviously, the sequence
$0\rightarrow V_{t}(\gamma_1,\gamma_2,\gamma_3;i-t+n)\rightarrow
L\rightarrow V_{t}(\gamma_1,\gamma_2,\gamma_3;i)\rightarrow 0$ is
not splitting. Hence
$Ext_{H_\beta}(V_{t}(\gamma_1,\gamma_2,\gamma_3;i-t+n),V_{t}(\gamma_1,\gamma_2,\gamma_3;i))\ne
0$.\end{proof}

\begin{remark} Since $H_\beta$ is a Hopf algebra, the dual of any
$H_\beta$ module is still a $H_\beta$ module. It is easy to prove that the dual
of any irreducible representation is irreducible. Therefore the dual
module  $V_{I}(\gamma_1,\gamma_2,\gamma_3;i)^*$,
$V_{II}(\gamma_1,\gamma_2,\gamma_3;i)^*$,
$V_{0}(\gamma_1,\gamma_2,\gamma_3;i)^*$, and
$V_{r}(\gamma_1,\gamma_2,\gamma_3;i)^*$ are also irreducible
representations of $H_\beta$.
\end{remark}

%%%%%%%%%%%%%%%%%%%%%%%%%%%%%%%%%%%%%%%%%%%%%%%%%%%%%%%%%%%%%%%%%%%%%%%%%%%%%
                       %Section 5 : iGrothendieck ring
%%%%%%%%%%%%%%%%%%%%%%%%%%%%%%%%%%%%%%%%%%%%%%%%%%%%%%%%%%%%%%%%%%%%%%%%%%%%%

\section{Grothendieck ring $G_0(H_\beta)$}

%we determine the Grothendieck ring of the Hopf algebra $H_\beta$.

We recall some notations. Let $R$ be an algebra over field {\bf K}.
Recall that the Grothendieck group $G_0(R)$  is the abelian group generated by the
set $\{[M]|$ $M$ is a left $R$-module$\}$ of isomorphic classes of finite-dimensional $R$-modules with relations $[B]=[A]+[C]$ for every short exact sequence
$0\to A\to B\to C\to 0$ of R-modules. Unlike  ordinary algebras, the Grothendieck group $G_0(H)$ of a Hopf algebra $H$ has a product given by $[M][N]=[M\otimes N]$. With this product, Grothendieck group $G_0(H)$ becomes a ring, which is called the Grothendieck ring of the Hopf algebra $H$.

Assume that $H_\beta$ is the Hopf algebra defined in Section~\ref{sec-2} and the modules 
$$V_I(\gamma_1,\gamma_2,\gamma_3;i),\,
 V_{II}(\gamma_1,\gamma_2,\gamma_3;i), \,
 V_{0}(\gamma_1,\gamma_2,\gamma_3;i) \, \,
 \text{and}\; 
  V_r(\gamma_1,\gamma_2,\gamma_3;i)$$
  are defined in Lemma~\ref{lem52}, Lemma~\ref{lem53} and Lemma~\ref{lem54}. In this section, we determine the Grothendieck ring $G_0(H_\beta)$ of the Hopf algebra $H_\beta$ in several cases, which  is generated by
$$[V_I(\gamma_1,\gamma_2,\gamma_3;i)],\,
[V_{II}(\gamma_1,\gamma_2,\gamma_3;i)],\,
[V_0(\gamma_1,\gamma_2,\gamma_3;i)],\,
\text{and}\; [V_r(\gamma_1,\gamma_2,\gamma_3;i)],$$ 
where $0\leq i\leq n-1$, $2\leq r\leq t$, 
$(\gamma_1,\gamma_2,\gamma_3)\in(\bf{K}^*)^3$ and $t=|q^{n_1}|$ is the order of $q^{n_1}$. The classification results of all cases are illustrated in several lemmas and theorems. In details, the Grothendieck rings in cases that iff two the three equations $\beta_1=0$, $\beta_2=0$ and $\beta_3=0$ hold are stated in theorems~\ref{L55},\,\ref{thmVI} and \ref{thmV}. To do that, lemmas~\ref{L54},\, \ref{L52} and \ref{L53} are necessary. In theorems \ref{V}, \, \ref{T9} and \ref{T10}, we give the cases in which iff one of the conditions $\beta_1=0$, $\beta_2=0$ and $\beta_3=0$ holds. In additional, we also obtain properties of the Grothendieck rings $G_0(U_{(N,\nu,\omega)})$ and $G_0({\mathcal
U}_{(n,N,\nu,q,\alpha,\beta,\gamma)})$ in several corollaries. We also observe that there exist non-isomorphic Hopf algebras with isomorphic Grothendieck rings.

In order to discuss $G_0(H_\beta)$ in which $\beta_1=\beta_2=\beta_3=0$,  the analyze to the subring of $G_0(H_\beta)$ generated by all $[V_0(\gamma_1,\gamma_2,\gamma_3;i)]$ for all nonzero $\gamma_i\in{\bf K}$ and $i\in\mathbb{Z}_n$ is crucial. 
We present the results of this case in Theorem~\ref{them52} and the following Lemma \ref{L51} is necessary.
%%%%%%%%%%%%%%%%%%%%%%%%%%%%%%%%%%%%%%%%%%%%%%%%%%%%%%%%%%%%%%%
\begin{lemma}\label{L51}
The product of the Grothendieck ring $G_0(H_\beta)$ satisfies
$$[V_0(\gamma_1,\gamma_2,\gamma_3;i)][V_0(\gamma'_1,\gamma'_2,\gamma'_3;j)]= [V_0(\gamma_1\gamma'_1,\gamma_2\gamma'_2,\gamma_3\gamma'_3;i+j)]$$ 
for any $i,j\in\mathbb{Z}_n$ and
$\gamma_k,\gamma_k'\in{\bf K}^*$, $k\in\{1,2,3\}$.
\end{lemma}
\begin{proof}
By the definition of $V_0(\gamma_1,\gamma_2,\gamma_3;i)$ in Lemma~\ref{lem54}, we obtain that the product $$V_0(\gamma_1,\gamma_2,\gamma_3;i)\otimes V_0(\gamma'_1,\gamma'_2,\gamma'_3;j)$$ is generated by $1\otimes 1$. Since $$\beta_1(\gamma_1^{n_1}-\gamma_2^n)=\beta_2(\gamma_1^{n_1}-\gamma_3^n)=\beta_3(\gamma_1^\frac{2n_1}{n}q^{2n_1i}-\gamma_2\gamma_3)=0$$ and 
$$\beta_1(\gamma_1'^{n_1}-\gamma_2'^n)=\beta_2(\gamma_1'^{n_1}-\gamma_3'^n)=\beta_3(\gamma_1'{^\frac{2n_1}{n}}q^{2n_1j}-\gamma_2'\gamma_3')=0,$$
we have
$$a\cdot(1\otimes 1)=(a\cdot 1)\otimes (a\cdot 1)=\sqrt[n]{\gamma_1\gamma'_1}q^{i+j}1\otimes 1,$$
$$b\cdot(1\otimes 1)=(b\cdot 1)\otimes (b\cdot 1)=(\gamma_2\gamma'_2)1\otimes 1,\quad
c\cdot(1\otimes 1)=(c\cdot 1)\otimes (c\cdot 1)=(\gamma_3\gamma'_3)1\otimes 1,$$
$$x\cdot(1\otimes 1)=(x\cdot 1)\otimes (a^{n_1}\cdot 1)+ (b\cdot 1)\otimes (x\cdot 1) =0,$$
$$y\cdot(1\otimes 1)=(y\cdot 1)\otimes (a^{n_1}\cdot 1)+ (c\cdot 1)\otimes (y\cdot 1) =0.$$
Therefore, $V_0(\gamma_1,\gamma_2,\gamma_3;i)\otimes V_0(\gamma'_1,\gamma'_2,\gamma'_3;j)\cong V_0(\gamma_1\gamma'_1,\gamma_2\gamma'_2,\gamma_3\gamma'_3;i+j).$
\end{proof}
%%%%%%%%%%%%%%%%%%%%%%%%%%%%%%%%%%%%%%%%%%%%%%%%%%%%%%%%%%%%%%%%%%%%%%%%%%%%
%%%%%%%%%%%%%%%%%%%%%%%%%%%%Theorem %%%%%%%%%%%%%%%%%%%%%%%%%%%%%%%%%%%%%
%%%%%%%%%%%%%%%%%%%%%%%%%%%%%%%%%%%%%%%%%%%%%%%%%%%%%%%%%%%%%%%%%%%%%%%%%%%%
%%%%%%%%%%%%%%%%%%%%%%%%%%%%%%%%%%%%%%%%%%%%%%%%%%%%%%%%%%%%%%%%%%%%%%%%%%%%
By Lemma \ref{L51}, we can prove the following result.
\begin{theorem}\label{them52}
Let $R_1$ be the subring of the Grothendieck ring $G_0(H_\beta)$ generated by all $[V_0(\gamma_1,\gamma_2,\gamma_3;i)]$ for all nonzero $\gamma_i\in{\bf K}$ and $i\in\mathbb{Z}_n$. Then we have the following.
\begin{itemize}
	\itemsep=0pt
\item [(i)]  If  at most one of $\beta_1, \beta_2, \beta_3$ equals to zero, then  $R_1$ is isomorphic to $\mathbb{Z}[ \mathbb{Z}_{n_1}\times \mathbb{Z}_n^2\times {\bf K^*}]$;
\item [(ii)]   If only $\beta_1\neq 0$ or $\beta_2\neq 0$, then   $R_1$ is isomorphic to $\mathbb{Z}[ \mathbb{Z}_{n_1}\times \mathbb{Z}_n\times {(\bf K^*)}^2]$;
 \item [(iii)] If only $\beta_3\neq 0$, then  $R_1$ is isomorphic to $\mathbb{Z}[\mathbb{Z}_{2n_1}\times \mathbb{Z}_n\times {(\bf K^*)}^2]$;
 \item [(iv)]   If $\beta_1=\beta_2=\beta_3=0$, then   $R_1$ is isomorphic to $\mathbb{Z}[ \mathbb{Z}_n\times {\bf K^*}\times {\bf K^*}\times {\bf K^*}]$.\\\noindent Thus $G_0(H_\beta)\cong\mathbb{Z}[ \mathbb{Z}_n\times {\bf K^*}\times {\bf K^*}\times {\bf K^*}]$ if $\beta=(\beta_1,\beta_2,\beta_3)=0$.

\end{itemize}
\end{theorem}
%%%%%%%%%%%%%%%%%%%%%%%%%%%%%%%%%%%%%%%%%%%%%%%%%%%%%%%%%%%%%%%%%%%%%%%%%%%%
%%%%%%%%%%%%%%%%%%%%%%%%%%%%%%%%%%%%%%%%%%%%%%%%%%%%%%%%%%%%%%%%%%%%%%%%%%%%
\begin{proof}
We prove this claim by several cases.
\begin{itemize}
\itemsep=0pt
\item[(i)] 
First, we consider the case where $\beta_1\beta_2\beta_3\neq0$. By $\beta ''_1=0$ and $\beta''_2=0$, we obtain $\beta_1^{\prime\prime}=\beta_1'(1-\frac{\gamma_2^n}{\gamma_1^{n_1}})=0	$ and $\beta_2^{\prime\prime}=\beta_2'(1-\frac{\gamma_3^n}{\gamma_1^{n_1}})=0$. 
	Hence $\gamma_1=\gamma_2^{\frac{n}{n_1}}\bar{\omega}^p=\gamma_3^{\frac{n}{n_1}}\bar{\omega}^l$ and $\gamma_3=q^u\gamma_2$ for some intergers $p,l,u$, where $\bar{\omega}$ is the primitive $n_1$-th unitary root.  Thus
\begin{equation}\label{EQ1}\begin{array}{lll}
V_0(\gamma_1,\gamma_2,\gamma_3,i)&\cong &V_0(\bar{\omega}^p,1,q^u;i)\otimes V_0(\gamma_2^{\frac{n}{n_1}},\gamma_2, \gamma_2;  0)\\
&\cong&V_0(\bar{\omega},1,1;0)^{\otimes p}\otimes V_0(1,1,q^u;0)\otimes V_0(1,1,1;1)^{\otimes i}\otimes V_0(\gamma_2^{\frac{n}{n_1}},\gamma_2,\gamma_2;  0)\end{array}
\end{equation} for some nonzero $\gamma_2$ and $i\in \mathbb{Z}_n$. Then $R_1$ is isomorphic to the group algebra $\mathbb{Z}[ \mathbb{Z}_{n_1}\times \mathbb{Z}^2_n\times {\bf K^*}]$, where ${\bf K^*}$ is the multiplicative group of all nonzero elements of ${\bf K}$.

If $\beta_2\beta_3\neq 0$ and $\beta_1=0$,  then $\gamma_1=\gamma_3^{\frac{n}{n_1}}\bar{\omega}^p$ for some integer $p$. Since $\gamma_1^{-\frac{2n_1}n}\gamma_2\gamma_3=q^{2n_1i}$,  $\gamma_2=q^{2n_1i+u}\gamma_3$ for some integer $u$. Thus
$$\begin{array}{lll}V_0(\gamma_1,\gamma_2,\gamma_3,i)&\cong &V_0(\bar{\omega}^p,q^{2n_1i+u},  1;i)\otimes V_0(\gamma_3^{\frac{n}{n_1}}, \gamma_3, \gamma_3;  0)\\
&\cong&V_0(\bar{\omega},1,1;0)^{\otimes p}\otimes V_0(1, q^{2n_1i+u},1;0)\otimes V_0(1,1,1;1)^{\otimes i}\otimes V_0(\gamma_3^{\frac{n}{n_1}}, \gamma_3, \gamma_3;  0).
\end{array}$$
Hence  $R_1$ is isomorphic to the group algebra $\mathbb{Z}[ \mathbb{Z}_{n_1}\times\mathbb{Z}_n^2\times {\bf K^*}]$. Similarly, we can prove that $ R_1$ is also isomorphic to $\mathbb{Z}[\mathbb{Z}_{n_1}\times\mathbb{Z}_n^2\times {\bf K}^*]$  if  $\beta_1\beta_3\neq0$ and $\beta_2=0$.

If  $\beta_1\beta_2\neq 0$ and $\beta_3=0$, then $\gamma_1=\gamma_3^{\frac{n}{n_1}}\bar{\omega}^p$ and $\gamma_2=\gamma_3q^u$ for some integers  $p, u$.
Thus 
$$\begin{array}{lll}V_0(\gamma_1,\gamma_2,\gamma_3,i)&\cong &V_0(\bar{\omega},1,1;0)^{\otimes p}\otimes V_0(1,1,  1;i)\otimes V_0(\gamma_3^{\frac{n}{n_1}}, q^{u}\gamma_3, \gamma_3;  0)\\
&\cong&V_0(\bar{\omega},1,1;0)^{\otimes p}\otimes V_0(1,q,1;0)^{\otimes u}\otimes V_0(1,  1,1;1)^{\otimes i}\otimes V_0(\gamma_3^{\frac{n}{n_1}}, \gamma_3, \gamma_3;  0).
\end{array}$$
Hence $R_1$ is isomorphic to $\mathbb{Z}[\mathbb{Z}_{n_1}\times \mathbb{Z}^2_n\times {\bf K}^*]$.

\item[(ii)] 
If $\beta_1\neq 0$ and $\beta_2=\beta_3=0$, then $\gamma_1=\gamma_2^{\frac{n}{n_1}}\bar{\omega}^p$ for some integer $p$.  Thus $$\begin{array}{lll}V_0(\gamma_1,\gamma_2,\gamma_3,i)&\cong &V_0(\bar{\omega}^p,1,  1;i)\otimes V_0(\gamma_2^{\frac{n}{n_1}}, \gamma_2, \gamma_3;  0)\\
&\cong& V_0(\bar{\omega},1,1;0)^{\otimes p}\otimes V_0(1, 1,1;1)^{\otimes i}\otimes V_0(\gamma_2^{\frac{n}{n_1}}, \gamma_2, \gamma_2;  0)\otimes V_0(1,1,\gamma_2^{-1}\gamma_3;0).
\end{array}$$
Hence $R_1$ is isomorphic to $\mathbb{Z}[\mathbb{Z}_{n_1}\times \mathbb{Z}_n\times{({\bf K}^*)}^2]$. Similarly, we can prove that $R_1$ is isomorphic to $\mathbb{Z}[\mathbb{Z}_{n_1}\times \mathbb{Z}_n\times{({\bf K}^*)}^2]$ if $\beta_2\neq0$ and $\beta_1=\beta_3=0$.

\item[(iii)]
If $\beta_1=\beta_2=0$ and $\beta_3\neq 0$, then $\gamma_1=(\gamma_2\gamma_3)^{\frac{n}{2n_1}}\bar{\omega}^{\frac p2}$ for some integer $p$. Thus $$\begin{array}{lll}V_0(\gamma_1,\gamma_2,\gamma_3,i)&\cong &V_0(\bar{\omega}^{\frac p2},1,1;0)\otimes V_0(1,1,  1;i)\otimes V_0((\gamma_2\gamma_3)^{\frac{n}{2n_1}}, \gamma_2, \gamma_3;  0)
\end{array}.$$
Hence $R_1$ is isomorphic to $\mathbb{Z}[\mathbb{Z}_{2n_1}\times \mathbb{Z}_n\times{({\bf K}^*)}^2]$.

\item[(iv)]
Finally, if $\beta=0$, then $\begin{array}{lll}V_0(\gamma_1,\gamma_2,\gamma_3,i)&\cong &V_0(1,1,  1;i)\otimes V_0(\gamma_1, \gamma_2, \gamma_3;  0).
\end{array}$
Hence $R_1$ is isomorphic to $\mathbb{Z}[\mathbb{Z}_n\times{\bf K}^*\times{\bf K}^*\times{\bf K}^*]$.
\end{itemize}
\end{proof}

Since $H_{\beta}/(a^N-1,b-1,c-1)$ is isomorphic to Gelaki's Hopf algebra ${\mathcal U}_{(n,N,n_1,q,\beta_1,\beta_2,\beta_3)}$, where $n|N$, we obtain the following properties of the Gelaki's Hopf alegbra.
%%%%%%%%%%%%%%%%%%%%%%%%%%%%%%%%%%%%%%%%%%%%%%%%%%%%%%%%%%%%%%%%%%%%%%%%%%%%
%%%%%%%%%%%%%%%%%%%%%%%%%%%%%%%%%%%%%%%%%%%%%%%%%%%%%%%%%%%%%%%%%%%%%%%%%%%%%
%                Corollary
%%%%%%%%%%%%%%%%%%%%%%%%%%%%%%%%%%%%%%%%%%%%%%%%%%%%%%%%%%%%%%%%%%%%%%%%%%%
%%%%%%%%%%%%%%%%%%%%%%%%%%%%%%%%%%%%%%%%%%%%%%%%%%%%%%%%%%%%%%%%%%%%%%%%%%%%
\begin{corollary}\label{cor53}Suppose that $n|N$. Let $\mathfrak{q}$ be a primitive $N$-th root of unity and $S_1$ be the subring of $R_1$ generated by $[V_0(\gamma,1,1;i)]$, where $\gamma^{\frac{N}{n}}=1$. Then we have the following result.
\begin{itemize}
	\itemsep=0pt
\item[(1)] 
If either $\beta_1\beta_2\neq 0$ and $\beta_3=0$, or $\beta_1\neq 0$ and $\beta_2=\beta_3=0$, or $\beta_2\neq 0$ and $\beta_1=\beta_3=0$, then $S_1=\mathbb{Z}[g,h]$ for $g=[V_0(1,1,1;1)]$ and $h=[V_0(\mathfrak{q}^{\frac{N}{(N/n,n_1)}},1,1;0)]$. Moreover, 
$$g^n=h^{(N/n,n_1)}=1 \quad \text{and} \quad \mathbb{Z}[g,h]\simeq \mathbb{Z}[\mathbb{Z}_{(N/n,n_1)}\times \mathbb{Z}_n].$$
\item[(2)] 
 If $\beta_3\neq 0$ and $\beta_1=\beta_2=0$, then $S_1=\mathbb{Z}[g,h_1]$ for $g=[V_0(1,1,1;1)]$ and $$h_1=[V_0(\mathfrak{q}^{\frac{Nn}{(N,2n_1)}},1, 1;0)].$$ Moreover,  $g^n=h_1^{\mu}=1$ and $\mathbb{Z}[g,h_1]\simeq \mathbb{Z}[\mathbb{Z}_{\mu}\times \mathbb{Z}_n]$, where $\mu=\frac{(N,2n_1)}{(n,2n_1)}$.
\item[(3)]
 If either $\beta_1\beta_3\neq 0$ and $\beta_2=0$, or $\beta_2\beta_3\neq 0$ and $\beta_1=0$, or $\beta_1\beta_2\beta_3\neq 0$, then $S_1=\mathbb{Z}[g,h_2]$ for $g=[V_0(1,1,1;1)]$ and $h_2=[V_0(\mathfrak{q}^{n\nu},1,1;0)]$, where $\nu=\frac{N}{(N,2n_1,nn_1)}$. Moreover, $g^n=h_2^{\mu'}=1$ and $\mathbb{Z}[g,h_2]\simeq \mathbb{Z}[\mathbb{Z}_{\mu'}\times\mathbb{Z}_n]$, where $\mu'=\frac{(N,2n_1,nn_1)}{(n,2n_1)}$.
\item[(4)] 
If either $\beta_1=\beta_2=\beta_3=0$, then $S_1=\mathbb{Z}[g,h_3]$ for $g=[V_0(1,1,1;1)]$ and $h_3=[V_0(\mathfrak{q}^n,1,1;0)]$. Moreover, $g^n=h_3^\frac{N}{n}=1$ and $\mathbb{Z}[g,h_3]\simeq \mathbb{Z}[\mathbb{Z}_{N/n}\times \mathbb{Z}_n].$
\item[(5)]
If $\beta_1=\beta_2=\beta_3=0$,  then the Grothendieck ring $G_0({\mathcal U}_{(n,N,n_1,q,\beta_1,\beta_2,\beta_3)})$ of the Gelaki's Hopf algebra ${\mathcal U}_{(n,N,n_1,q,0,0,0)}$  is equal to $\mathbb{Z}[g,h_3]$.
\end{itemize}
\end{corollary}
\begin{proof}
If  $\beta_1\beta_2\neq 0$ and $\beta_3=0$, then $\gamma=\bar{\omega}^p$ satisfying $\gamma^{\frac{N}{n}}=1$, where $\bar{\omega}$ is a primitive $n_1$-th root of unity.  Thus $n_1|\frac{N}{n}p$ and $\frac{n_1}{(n_1,\frac{N}{n})}|p$.
Therefore $\gamma\in\langle\bar{\omega}^{\frac{n_1}{(N/n,n_1)}}\rangle=\langle\mathfrak {q}^{\frac{N}{(N/n,n_1)}}\rangle
$, where $\mathfrak{q}$ is a primitive $N$-th root of unity. Hence $S_1=\mathbb{Z}[g,h]\simeq \mathbb{Z}[\mathbb{Z}_{{(N/n,n_1)}}\times \mathbb{Z}_{n}]$ by (\ref{EQ1}), where $g=[V_0(1,1,1;1)]$, $h=[V_0( \mathfrak{q}^{\frac{N}{(N,n_1)}},1,1;0)]$. It is obvious that $g^n=h^{{(N/n,n_1)}}=1$.

If $\beta_3\neq 0$ and $\beta_1=\beta_2=0$,  then $\gamma^{-\frac{2n_1}n}=1$ for $V_0(\gamma,1,1;0)$. Since $\gamma^{\frac{N}n}=1$, $\gamma=\mathfrak{q}^{nr}$ for some integer $r$, where $\mathfrak{q}$ is a primitive $N$-th  root of unity. Then $\mathfrak{q}^{-2rn_1}=1$ and $N|2rn_1$. Thus $r=\frac{N}{(N,2n_1)}p$ for any integer $p$. Hence $\gamma\in\langle \mathfrak{q}^{\frac{Nn}{(N,2n_1)}}\rangle$. Thus
$V_0(\gamma,1,1;i)\cong V_0(\mathfrak {q}^{\frac{Nn}{(N,2n_1)}},1,1;0)^{\otimes p}\otimes V_0(1,1,1;1)^{\otimes i}.$
Hence  $S_1=\mathbb{Z}[g,h_1]$ for $h_1=[V_0(\mathfrak{q}^{\frac{Nn}{(N,2n_1)}},1,1;0)].$  It is obvious that $S_1$ is isomorphic to the group algebra $\mathbb{Z}[ \mathbb{Z}_{\mu}\times\mathbb{Z}_n]$, where $\mu=\frac{(N,2n_1)}{(n,2n_1)}$.

If either $\beta_1\beta_3\neq 0$, or $\beta_2\beta_3\neq 0$, or $\beta_1\beta_2\beta_3\neq 0$, then $\gamma^{\frac{N}n}=\gamma^{n_1}=\gamma^{\frac{2n_1}n}=1$ for $V_0(\gamma,1,1;0)$. Thus $\gamma=\mathfrak{q}^{nr}$ for some integer $r$ and $\mathfrak{q}^{nrn_1}=\mathfrak{q}^{2n_1r}=1$. Hence $r=\frac{N}{(N,nn_1)}p=\frac{N}{(N,2n_1)}p'$ for some integers $p,p'$. Let $\bar{d}=(\frac{N}{(N,nn_1)},\frac{N}{(N,2n_1)})$ and $\nu=\frac{N^2}{\bar{d}(N,nn_1)(N,2n_1)}=\frac{N}{(N,2n_1,nn_1)}$. Then $\gamma\in\langle\mathfrak{q}^{n\nu}\rangle$. Thus $S_1=\mathbb{Z}[g,h_2]$ for $h_2=[V_0(\mathfrak{q}^{n\nu},1,1;0)].$  It is obvious that $S_1$ is isomorphic to the group algebra $\mathbb{Z}[ \mathbb{Z}_{\mu'}\times\mathbb{Z}_n]$, where $\mu'=|\mathfrak{q}^{n\nu}|=\frac{(N,2n_1,nn_1)}{(n,2n_1)}$.

If $\beta_1\neq 0$ and $\beta_2=\beta_3=0$, then $\gamma=\bar{\omega}^p$ for some integer $p$.  Thus  $S_1=\mathbb{Z}[g,h]$, which is isomorphic to $\mathbb{Z}[\mathbb{Z}_{(N/n, n_1)}\times \mathbb{Z}_n]$.  Similarly, we can prove that $S_1$ is isomorphic to $\mathbb{Z}[\mathbb{Z}_{(N/n,n_1)}\times \mathbb{Z}_n]$ if $\beta_2\neq0$ and $\beta_1=\beta_3=0$.

Finally, if $\beta=0$, then $\begin{array}{lll}V_0(\gamma,1,1,i)&\cong &V_0(1,1,  1;i)\otimes V_0(\gamma, 1, 1;  0).
\end{array}$
Hence $S_1=\mathbb{Z}[g,h_3]$, where $h_3=[V_0(\mathfrak{q}^n,1,1;0)].$  Thus $S_1$  is isomorphic to $\mathbb{Z}[\mathbb{Z}_{N/n}\times \mathbb{Z}_n]$.
\end{proof}

Let  $\mathfrak{q}$ is a primitive $N$-th root of unity. Set
  $g:=[V_0(1,1,1;1)],\;
 \mathfrak{s}:=\sum\limits_{p=0}^{n-1}g^p, \; 
 \mathfrak{s}':=\sum\limits_{k=1}^{u}g^{kt},\\
 \mathfrak{s}'':=\sum\limits_{k=0}^{t-1}g^{n-k}, \; 
 h:=[V_0( \mathfrak{q}^{\frac{N}{(N,n_1)}},1,1;0)],\;
 h_1:=[V_0( \mathfrak{q}^{\frac{Nn}{(N,2n_1)}},1,1;0)],\;
 h_2:=[V_0(\mathfrak{q}^{n\nu},1,1;0)],\\
 h_3:=[V_0(\mathfrak{q}^{n},1,1;0)],\;$
 and 
$ g_{\gamma_1,\gamma_2,\gamma_3}: =[V_0(\gamma_1,\gamma_2,\gamma_3;0)]\in G_0(H_\beta).$

 To determine the Grothendieck ring $G_0(H_\beta)$ for $\beta_3\neq 0$, the following lemma is necessary.
 
\begin{lemma}\label{L54}
Suppose that $\beta_3(q^{2n_1i}-\gamma_1^{-\frac{2n_1}n}\gamma_2\gamma_3)\neq 0$, $\beta_1(\gamma_1'^{n_1}-\gamma_3'^{n})=\beta_2(\gamma_2'^{n}-\gamma_3'^{n})=0$ and
$\beta_3(\gamma_1'^{-\frac{2n_1}{n}}\gamma_2'\gamma_3'-q^{2n_1j})=0$. Then
\begin{eqnarray}
[V_r(\gamma_1,\gamma_2,\gamma_3;i)][V_0(\gamma'_1,\gamma'_2,\gamma'_3;j)]&=&[V_r(\gamma_1\gamma'_1,\gamma_2\gamma'_2,\gamma_3\gamma'_3;i+j)]\nonumber\\ &=&[V_0(\gamma'_1,\gamma'_2,\gamma'_3;j)][V_r(\gamma_1,\gamma_2,\gamma_3;i)].\nonumber
\end{eqnarray}
\end{lemma}

\begin{proof} As $\beta_3(\gamma_1^\frac{2n_1}{n}q^{2n_1i}-\gamma_2\gamma_3)\neq0$, we have $\beta_3\neq0$.
Moreover, 
$$\gamma_1'^{\frac{2n_1}{n}}q^{2n_1j}=\gamma_2'\gamma_3', ((\gamma_1\gamma_1')^{\frac{2n_1}{n}}q^{2n_1(i+j)}-(\gamma_2\gamma_2')(\gamma_3\gamma_3'))\beta_3=\beta_3(\gamma_1^{\frac{2n_1}{n}}q^{2n_1i}-\gamma_2\gamma_3)(\gamma_2'\gamma'_3)\neq0.$$
Since $\beta_1(\gamma_1^{n_1}-\gamma_2^n)=\beta_2(\gamma_1^{n_1}-\gamma_3^n)=0$ and $\beta_1(\gamma_1'^{n_1}-\gamma_2'^n)=\beta_2(\gamma_1'^{n_1}-\gamma_3'^n)=0$, we have $$\beta_1((\gamma_1\gamma'_1)^n_1-(\gamma_2\gamma'_2)^n)
=\beta_2((\gamma_1\gamma_1')^{n_1}-(\gamma_3\gamma_3')^n)=0.$$ 
If $\gamma_1^{-\frac {2n_1}n}\gamma_2\gamma_3= q^{n_1(2i-r+1)}$, then $(\gamma_1\gamma'_1)^{- \frac {2n_1}n}(\gamma_2\gamma_2')(\gamma_3\gamma_3)= q^{n_1(2(i+j)-r+1)}$. If $\gamma_1^\frac{-2n_1}{n}\gamma_2\gamma_3\neq q^{n_1v}$ for any  integer $v$, then $(\gamma_1\gamma'_1)^{-\frac {2n_1}n}(\gamma_2\gamma_2')(\gamma_3\gamma_3)\neq q^{n_1v}$ for any  integer $v$.
Let $\{m_0,m_1,\cdots,m_{r-1}\}$ be the basis of $V_r(\gamma_1,\gamma_2,\gamma_3;i)$ with the action given by (\ref{eq15})-(\ref{eq17}). Then $\{m_0\otimes 1,m_1\otimes 1,\cdots, m_{r-1}\otimes 1\}$ is a basis of $V_r(\gamma_1,\gamma_2,\gamma_3;i)\otimes V_0(\gamma'_1,\gamma'_2,\gamma'_3;j)$ and the action on this basis is as follows:
$$a\cdot(m_l\otimes 1)=(a\cdot m_l)\otimes (a\cdot 1)=\sqrt[n]{\gamma_1\gamma'_1}q^{i+j-l}m_l\otimes 1,$$
$$b\cdot(m_l\otimes 1)=(b\cdot m_l)\otimes (b\cdot 1)=(\gamma_2\gamma'_2)m_l\otimes 1,\quad
c\cdot(m_l\otimes 1)=(c\cdot m_l)\otimes (c\cdot 1)=(\gamma_3\gamma'_3)m_l\otimes 1,$$
$$x\cdot(m_l\otimes 1)=(x\cdot m_l)\otimes (a^{n_1}\cdot 1)+ (b\cdot m_l)\otimes (x\cdot 1) =\gamma'^{\frac{n_1}{n}}_1q^{jn_1}m_{l+1}\otimes 1,\; for\; 0\leq 0\leq r-2,$$
$$x\cdot(m_{r-1}\otimes 1)=(x\cdot m_{r-1})\otimes (a^{n_1}\cdot 1)+ (b\cdot m_{r-1})\otimes (x\cdot 1) =0,$$
$$y\cdot(m_l\otimes 1)=(y\cdot m_l)\otimes (a^{n_1}\cdot 1)+ (c\cdot m_l)\otimes (y\cdot 1) =k_l\gamma'^{\frac{n_1}{n}}_1q^{jn_1}m_{l-1}\otimes 1,\; for\; 1\leq l\leq r-1,$$
$$y\cdot(m_0\otimes 1)=(y\cdot m_0)\otimes (a^{n_1}\cdot 1)+ (c\cdot m_0)\otimes (y\cdot 1) =0.$$
Let $u_l:=\gamma'^{\frac{n_1l}{n}}_1q^{jn_1l}m_l\otimes 1$, then $\{u_0,\cdots, u_{r-1}\}$ is a new basis of $V_r(\gamma_1,\gamma_2,\gamma_3;i)\otimes V_0(\gamma'_1,\gamma'_2,\gamma'_3;j)$. Under this basis,
\begin{eqnarray*}au_k=\sqrt[n]{\gamma_1\gamma_1'}q^{i+j-k}u_k;\ \
bu_k=\gamma_2\gamma_2'u_k;\  \  cu_k=\gamma_3\gamma_3'u_k\qquad for\quad 0\leq k\leq r-1;
\end{eqnarray*}
\begin{eqnarray*}xu_k=u_{k+1},\qquad for\quad 0\leq k\leq r-2;\qquad
xu_{r-1}=0;
\end{eqnarray*}
\begin{eqnarray*}yu_p=\gamma_1'^{\frac{2n_1}{n}}q^{2jn_1}k_pu_{p-1},\qquad for\quad 1\leq p\leq
r-1;\qquad yu_{0}=0.
\end{eqnarray*}
Hence $V_r(\gamma_1,\gamma_2,\gamma_3;i)\otimes V_0(\gamma'_1,\gamma'_2,\gamma'_3;j)\cong V_r(\gamma_1\gamma'_1,\gamma_2\gamma'_2,\gamma_3\gamma'_3;i+j)$.
Similarly, we get  $$V_0(\gamma'_1,\gamma'_2,\gamma'_3;j)\otimes V_r(\gamma_1,\gamma_2,\gamma_3;i)\cong V_r(\gamma_1\gamma'_1,\gamma_2\gamma'_2,\gamma_3\gamma'_3;i+j).$$
\end{proof}

Next, we discuss the case in which there is an integer $v$ such that $q^{(2i-v+1)n_1}=\gamma_1^{-\frac{2n_1}{n}}\gamma_2\gamma_3$. Let $r$ be the minimal positive integer such that $q^{(2i-r+1)n_1}=\gamma_1^{-\frac{2n_1}{n}}\gamma_2\gamma_3$.  From Lemma \ref{L54}, we obtain
\begin{equation}\label{EQ41}
V_r(\gamma_1,\gamma_2,\gamma_3,i)\cong V_r(1,1,1;0)\otimes V_0(\gamma_1,\gamma_2,\gamma_3; i), \ \  \text{for} \  2\leq r\leq t.
\end{equation}
By Equations  (\ref{eq16}) and
(\ref{eq17}), the action of $H_\beta$ on $V_r(1,1,1;0)$ is given by $a\cdot m_j=q^{\frac12(r-1)-j}m_j$, $b\cdot m_j=c\cdot m_j=m_j$, where $k_1,\cdots,k_r $ are determined by $k_l=\sum\limits^{l-1}_{j=0}q^{-jn_1}\beta_3(q^{(r-l)n_1}-1)$
 for $1\leq l\leq r-1.$

In the case that $\gamma_1^{\frac{2n_1}{n}}q^{vn_1}\neq\gamma_2\gamma_3$ for any integer $v$, we have
\begin{equation}\label{EQ42}
V_t(\gamma_1,\gamma_2,\gamma_3,i)\cong
\begin{cases}
V_t(\gamma_1(\gamma_2\gamma_3)^{-\frac{n}{2n_1}},1,1;0)\otimes V_0((\gamma_2\gamma_3)^{\frac{n}{2n_1}},\gamma_2, \gamma_3; i), \ &  \beta_1=\beta_2=0\ \text{or} \ \beta_1\beta_2\neq0\\
V_t(1,\gamma_2\gamma_3^{-1},1;0)\otimes V_0(\gamma_1,\gamma_3, \gamma_3; i), \ &  \beta_1=0, \beta_2\neq0\\
V_t(1,1,\gamma_3\gamma_2^{-1};0)\otimes V_0(\gamma_1,\gamma_2, \gamma_2; i), \ &  \beta_1\neq 0, \beta_2=0
\end{cases}\nonumber
\end{equation}
where  $\gamma_1'=(\gamma_3\gamma_2)^{-\frac{n}{2n_1}}\gamma_1 \neq \bar{\omega}^{\frac{p(n,n_1)}2}$ for any integer $p$.

Let $R_2$ be the subring of the Grothendieck ring $G_0(H_\beta)$ for $\beta_3\neq 0$ generated by
$R_1$ and $V_r(\gamma_1,\gamma_2,\gamma_3;i)$. 
Then $R_2$ is isomorphic to a quotient  ring of $R_1[z_r, z_{\xi}'| 2\leq r\leq t, \xi\in\widetilde{{\bf K}_0}]$, where
$\widetilde{{\bf K}_0}:=({\bf K}^*/\{\bar{\omega}^{\frac{p(n,n_1)}{2}}\mid p\in \mathbb{Z}\})\setminus \{\overline{1}\}$, $z_r=[V_r(1,1,1;0)]$ for $2\leq r\leq t$ and $z_{\xi}'=[V_t(\xi,1,1;0)]$ for $\xi\in\widetilde{{\bf K}_0}$. Moreover, if $\beta_3\neq 0$ and $\beta_1=\beta_2=0$, then $R_2=G_0(H_\beta)$.

Let $r_p=(r-2p)\ mod (t)$ and $1\leq r_p\leq t$, where $r$ is the minimal positive integer such that  $\xi\xi'=q^{(-r+1)n_1}$. Then we obtain the following theorem about the Grothendieck ring $G_0(H_\beta)$.
%  provided that $\xi\xi'=q^{un_1}$ for some integer $u$.
\begin{theorem}\label{L55} 
Suppose that $\beta_1=\beta_2=0$, $\beta_3\neq 0$. Then the Grothendieck ring $G_0(H_\beta)$ is isomorphic to $R_2:=R_1[z_2, z_{\xi}'| \xi\in\widetilde{{\bf K}_0}]$ with relations \begin{eqnarray}\label{Eq*1}\sum\limits_{v=0}^{[\frac{t}{2}]}(-1)^v\frac{t}{t-v}\binom{r-v}{v}g^{(n-1)v}z_2^{t-2v}-g^{n-t}-1=0,\end{eqnarray} 

$$z_2z_\xi'=z_{\xi q^\frac{n}{2}}'+g^{n-1}z_{\xi q^\frac{n}{2}}',\qquad z_\xi'z_{\xi'}'=\mathfrak{s}''z'_{\xi\xi'} \text { for }\  \xi\xi'\in \widetilde{\bf K}_0, $$
and
 $$z_\xi'z_{\xi'}'=\sum\limits^{t-1}_{p=0,r_p<t}(g^{n-p-r_p}z_{t-r_p}+g^{n-p}z_{r_p})+\sum\limits_{i=0,r_p=t}^{t-1}g^{n-p}z_t, \text{ for }\  \xi\xi'\in\{q^{un_1}\mid u\in{\mathbb{Z}}\}$$
where  $z_r=\sum\limits_{v=0}^{[\frac{r-1}{2}]}(-1)^v\binom{r-1-v}{v}g^{(n-1)v}z_2^{r-1-2v}$ for $3\leq r\leq t$.
\end{theorem}

\begin{proof} (1) Suppose that $\{m_0,\cdots,m_{r-1}\}$ and $\{m_0',m_{1}'\}$ are the basis of $V_{r}(1,1,1;0)$ and $V_{2}(1,1,1;0)$, respectively. Then $\{m_l\otimes m'_k,\; 0\leq l\leq r-1, 0\leq k\leq 1\}$ is a basis of $V_{r}(1,1,1;0)\otimes V_{2}(1,1,1;0)$. With this basis, we have
$$a\cdot(m_l\otimes m'_k)=(a\cdot m_l)\otimes (a\cdot m'_k)=q^{\frac{r}2-l-k}m_l\otimes m'_k,$$
$$b\cdot(m_l\otimes m'_k)=(b\cdot m_l)\otimes (b\cdot m'_k)=m_l\otimes m'_k,\quad
c\cdot(m_l\otimes m'_k)=(c\cdot m_l)\otimes (c\cdot m'_k)=m_l\otimes m'_k,$$ 
$$y\cdot(m_0\otimes m'_0)=(y\cdot m_0)\otimes (a^{n_1}\cdot m'_0)+ (c\cdot m_0)\otimes (y\cdot m'_0) =0.$$
Since
$\Delta(x^k)=\sum\limits^k_{l=0}\binom{k}{l}_{q^{n_1}}b^lx^{k-l}\otimes a^{(k-l)n_1}x^l
$  in $H_\beta$, where $\binom{k}{l}_{q^{n_1}}=\frac{(k)_{q^{n_1}}!}{(l)_{q^{n_1}}!(k-l)_{q^{n_1}}!}$ and $(p)_{q^{n_1}}!=(p)_{q^{n_1}}(p-1)_{q^{n_1}}\cdots (1)_{q^{n_1}}$ for $(p)_{q^{n_1}}=1+q^{n_1}+\cdots +q^{(p-1)n_1}$, we have
\begin{eqnarray}
x^{k'}(m_0\otimes m'_0)&=&(\sum\limits^{k'}_{l=0}\binom{k'}{l}_{q^{n_1}}b^lx^{k'-l}\otimes a^{(k'-l)n_1}x^l)(m_0\otimes m'_0)\nonumber\\
&=&\sum\limits^{k'}_{l=0}\binom{k'}{l}_{q^{n_1}}q^{(\frac{1}2-l)(k'-l)n_1}m_{k'-l}\otimes m'_l\nonumber.
\end{eqnarray}
Thus $x^{k'}(m_0\otimes m'_0)=0$ only if $k'\geq l_0:=min\{t,r+1\}$.

 If $2\leq r\leq t-1$, then  $x^{r+1}(m_0\otimes m_0')=0$. In addition,
\begin{eqnarray}\label{Eq21}\begin{array}{lll}
(yx^{k'})(m_0\otimes m'_0)&=&(q^{-n_1k'}x^{k'}y+\sum\limits^{k'-1}_{v=0}q^{-(k'-v-1)n_1}\beta_3x^{k'-1}(q^{-2n_1v}a^{2n_1}-bc))(m_0\otimes m'_0)\\
&=&\sum\limits^{k'-1}_{v=0}q^{-vn_1}\beta_3(q^{(r-k'+1)n_1}-1)x^{k'-1}\cdot(m_0\otimes m'_0) .
\end{array}\end{eqnarray}
Let $u_l=x^l(m_0\otimes m_0')$. Then $\{u_0,\cdots, u_{r}\}$ is a basis of $H_\beta(m_0\otimes m'_0)$, which is isomorphic to  $V_{r+1}(1,1,1;0)$ for $r\leq t-1$. Let
$A_1:=H_\beta(m_0\otimes m'_1)/H_\beta(m_0\otimes m^\prime_0)$. Then $$y(m_0\otimes m'_1)=(y\cdot m_0)\otimes (a^{n_1}\cdot m'_1)+ (c\cdot m_0)\otimes (y\cdot m'_1) =k'_1m_0\otimes m'_{0}=\overline{0}.$$
 It is easy to prove that $x^p(m_0\otimes m_1')=q^{\frac{pn_1}2}m_p\otimes m_1'\notin H_\beta(m_0\otimes m_0')$ for $0\leq p\leq r-2$ and $x^{r-1}(m_0\otimes m_1')\in H_\beta(m_0\otimes m_0')$.  Let $u_l=x^l(\overline{m_0\otimes m_1'})$. Then $\{u_0,\cdots, u_{r-2}\}$ is a basis of $H_\beta(\overline{m_0\otimes m'_1})$. Since $a(m_0\otimes m_1')=q^{\frac12(r-2)}m_0\otimes m_1',$  $H_\beta(\overline{m_0\otimes m_1'})\cong V_{r-1}(1,1,1;0)\otimes V_0(1,1,1,n-1)$.
Thus $z_rz_2=z_{r+1}+g^{n-1}z_{r-1}$ for $2\leq r\leq t-1$. In particular, $z_2^2=z_3+g^{n-1}.$
So $z_{r+1}=z_rz_2-g^{n-1}z_{r-1}$ for $2\leq r\leq t-1$.

 If $r=t$, then $yx(m_0\otimes m_0')=0$ by (\ref{Eq21}). Since $ax(m_0\otimes m_0')=q^{\frac12(t-2)}x(m_0\otimes m_0')$,
$H_{\beta}x(m_0\otimes m_0')\cong V_{t-1}(1,1,1;0)\otimes V_0(1,1,1;n-1)$.  Thus $H_\beta(m_0\otimes m'_0)/H_{\beta}x(m_0\otimes m_0')\cong V_0(1,1,1;0)$ and
$[H_\beta(m_0\otimes m_0')]=g^{n-1}z_{t-1}+1$, where $g=[V_0(1,1,1;1)]$.  Let
$A_1:=H_\beta(m_0\otimes m'_1)/H_\beta(m_0\otimes m^\prime_0)$. Then
$y(\overline{m_0\otimes m'_1})=\overline{0}$ and
$$x^{k'}(m_0\otimes m'_1)\in H_\beta(m_0\otimes m^\prime_0)$$ only if $k'\geq t$. Since $yx^{t-1}(m_0\otimes m_1')=0$. We have $H_{\beta}x^{t-1}(\overline{m_0\otimes m_1'})\cong V_0(1,1,1;n-t)$ and $H_\beta(\overline{m_0\otimes m'_1})/H_{\beta}x^{t-1}(\overline{m_0\otimes m_1'})\cong V_{t-1}(1,1,1;n-1)\cong V_{t-1}(1,1,1;0)\otimes V_0(1,1,1;n-1)$. Hence $z_tz_2=2g^{n-1}z_{t-1}+g^{n-t}+1$.
Consequently, we have $z_r=\sum\limits_{v=0}^{[\frac{r-1}{2}]}(-1)^v\binom{r-1-v}{v}g^{(n-1)v}z_2^{r-1-2v}$ for $3\leq r\leq t$ and $\sum\limits_{v=0}^{[\frac{t}{2}]}(-1)^v\frac{t}{t-v}\binom{r-v}{v}g^{(n-1)v}z_2^{t-2v}-g^{n-t}-1=0$.

(2)  Now we consider the product $z_2z_{\xi}'$. Similar to (1), we get $H_{\beta}(m_0\otimes m_0')\cong V_t(\xi q^\frac{n}{2},1,1;0)$ and $H_\beta(\overline{m_0\otimes m'_1})\cong V_t(\xi q^\frac{n}{2},1,1;n-1)\cong V_t(\xi q^\frac{n}{2},1,1;0)\otimes V_0(1,1,1;n-1)$. Thus $z_2z_\xi'=z_{\xi q^\frac{n}{2}}'+g^{n-1}z_{\xi q^\frac{n}{2}}'$.

(3) Here, we determine the product $z_{\xi}'z_{\xi'}'$. Similar to the proof of (1), we get $l_p=t$ and $V_t(\xi,1,1;0)\otimes V_t(\xi',1,1;0)=\sum\limits_{p=0}^{t-1}H_\beta(m_0\otimes m_p')$. Suppose that $\xi\xi'\in \widetilde{\bf K}_0$. Hence we get that $$H_\beta(\overline{m_0\otimes m_p'})\cong V_t(\xi\xi',1,1;n-p)\cong V_t(\xi\xi',1,1;0)\otimes V_0(1,1,1;n-p).$$
Thus $z_{\xi}'z_{\xi'}'=\sum\limits_{p=0}^{t-1}g^{n-p}z_{\xi\xi'}'$, where $g=[V_0(1,1,1;1)]\in R_1$.

 (4) Finally, we assume that $\xi\xi'\in\{\bar{\omega}^{\frac{p(n,n_1)}{2}}\mid p\in \mathbb{Z}\}$.
Let  $r$ be the minimal positive integer  such that $(\xi\xi')^{\frac{2n_1}{n}}q^{(-r+1)n_1}=1$. Then $(\xi\xi')^{\frac{2n_1}{n}}q^{(-2p-r_p+1)n_1}=1$ for $1\leq r_p\leq t$ satisfying  $r_p\equiv (r-2p)\ mod(t)$.  Hence
$H_\beta(\overline{m_0\otimes m_p'})\cong V_t(\xi\xi',1,1;n-p)\cong V_t(\xi\xi',1,1;0)\otimes V_0(1,1,1;n-p)$ if $r_p=t$. Otherwise, we have $H_{\beta}x^{r_p}(\overline{m_0\otimes m_p'})\cong V_{t-r_p}(\xi\xi',1,1;n-p-r_p)\cong V_{t-r_p}(\xi\xi',1,1;0)\otimes V_0(1,1,1;n-p-r_p)$ and $H_\beta(\overline{m_0\otimes m'_p})/H_{\beta}x^{r_p}(\overline{m_0\otimes m_p'})\cong V_{r_p}(\xi\xi',1,1;n-p)\cong V_{r_p}(\xi\xi',1,1;0)\otimes V_0(1,1,1;n-p)$.
Thus $z_\xi'z_{\xi'}'=\sum\limits^{t-1}_{p=0}W'(p), $ where
$W'(p)=
g^{n-p}z_t$ if $r_p=t$ and
$W'(p)=g^{n-p-r_p}z_{t-r_p}+g^{n-p}z_{r_p}$ if $r_p<t$.
The proof is completed.
\end{proof}

Especially, we obtain the structure of the Grothendieck ring  of the Gelaki's Hopf alegbra where $\beta_3\neq 0$.

\begin{corollary}\label{cor56} 
Suppose that  $\beta_3\neq 0$. 
\begin{itemize}
\item[(1)] If $N\nmid 2n_1t$, then one of the following satisfies. 
\item[(1.1)] If either $2n\mid N$, or $2n\nmid N$ and $2(n,n_1)\nmid n$, then the Grothendieck ring $G_0(\mathcal{U}_{(n,N,n_1,q, 0,0, \beta_3)})$ is isomorphic to $S_2:=\mathbb{Z}[g,h_1,z_2, z']$ with relations (\ref{Eq*1}), $z_2z'=(1+g^{n-1})z'$ and
\begin{eqnarray}\label{EQ23} {z'}^{v_0}={\mathfrak{s}''}^{v_0-2}\left(\sum\limits^{t-1}_{p=0,r_p<t}(g^{n-p-r_p}z_{t-r_p}+g^{n-p}z_{r_p})+\sum\limits_{i=0,r_p=t}^{t-1}g^{n-p}z_t\right),\end{eqnarray}
where $z_0=0$, $z_1=1$, $z_p=\sum\limits_{v=0}^{[\frac{p-1}{2}]}(-1)^v\binom{p-1-v}{v}g^{(n-1)v}z_2^{p-1-2v}$ for $3\leq p\leq t$ and $v_0=\frac{N}{(2n_1t,N)}$.

\item[(1.2)] If $2n\nmid N$ and $2(n,n_1)\mid n$, then the Grothendieck ring $G_0(\mathcal{U}_{(n,N,n_1,q, 0,0, \beta_3)})$ is isomorphic to $S_2:=\mathbb{Z}[g,h_1,z_3, z']$ with relations  (\ref{EQ23}), $z_3z'=(1+g^{n-1}+g^{n-2})z'$ and
\begin{equation}\label{T3}
\sum\limits_{v=0}^{[\frac{t-2}{4}]}(-1)^v\binom{\frac{t-2}{2}-v}{v}g^{(n-2)v}(z_3-g^{n-1})^{\frac{t-2}{2}-2v}-g^{n-1}z_{t-1}-g^{n-t}-1=0,
\end{equation}
where $z_0=0$, $z_1=1$, $v_0=\frac{N}{(2n_1t,N)}$ and   $z_{2l+1}=\sum\limits_{v=0}^{[\frac{l-1}{2}]}(-1)^v\binom{l-1-v}{v}g^{(n-2)v}(z_3-g^{n-1})^{l-1-2v}$ for $3\leq 2l+1\leq t-1$.

\item[(2)] If $N\mid 2n_1t$, then the Grothendieck ring $G_0(\mathcal{U}_{(n,N,n_1,q, 0,0, \beta_3)})$ is isomorphic to
$$\begin{cases}
\mathbb{Z}[g,h_1,z_2], & \text{if}\ \ 2n\mid N \ \text{or}\ 2n\nmid N \ \text{and}\ 2(n,n_1)\nmid n\\
\mathbb{Z}[g,h_1,z_3], & \text{if}\ \ 2n\nmid N\ \text{and}\ 2(n,n_1)\mid n
\end{cases}$$
 which is a subring of $S_2$.
 \end{itemize}

\end{corollary}

\begin{proof}  
Since $\gamma^{\frac{N}n}=1$, $\gamma=\mathfrak{q}^{np}$ for some integer $p$. If $\gamma^{\frac{2n_1}{n}}= q^{vn_1}$ for some integer $v$, then $\mathfrak{q}^{2n_1p}=q^{vn_1}$. Hence $\mathfrak{q}^{2n_1pt}=1$, i.e. $N\mid 2n_1pt$. Thus $\frac{N}{(N,2n_1t)}\mid p$.
Then $z'_{\gamma}\in S_2$ only if $\gamma\in \langle \mathfrak{q^n}\rangle$ and $\gamma\notin  \langle \mathfrak{q}^{nv_0}\rangle$ for $v_0=\frac{N}{(2n_1t,N)}$. If either $2n\mid N$, or $2n\nmid N$ and $2(n,n_1)\nmid n$, then we have generator $z_2=[V_0(1,1,1;0)]$. If $v_0=1$, that is $N(n,n_1)|2nn_1$, then $S_2=S_1[z_2]$. Otherwise, $S_2=S_1[z_2,z']$ for $z'=[V_t(\mathfrak{q}^n,1,1;0)]$. If $2n\nmid N$ and $2(n,n_1)\mid n$, then $[V_{r}(\gamma,1,1;0)]\in \mathcal{U}_{(n,N,n_1,q, 0,0, \beta_3)}$ for $\gamma^{\frac{2n_1}{n}}\in\langle q^{n_1}\rangle$ implies that $r$ is an odd integer. Let  $r_p\equiv (r-2p)\ mod(t)$ for $1\leq p\leq t$ and $z'=[V_t(\mathfrak{q}^{n},1,1;0)]$. We further assume that $1\leq r_p\leq t$. Then $(z')^{v_0}=z'(z')^{v_0-1}=
{\mathfrak{s}''}^{v_0-2}z'z'_{\mathfrak{q}^{v_0-1}}$ by Theorem \ref{L55}.  Since 
$$\mathfrak{q}^{2n_1v_0}=q^{\frac{2nn_1}{(2n_1t,N)}}=q^{(r-1)n_1}$$ for some $2\leq r\leq t$, $z'z_{\mathfrak{q}^{v_0-1}}'=\sum\limits^{t-1}_{p=0,r_p<t}(g^{n-p-r_p}z_{t-r_p}+g^{n-p}z_{r_p})+\sum\limits_{i=0,r_p=t}^{t-1}g^{n-p}z_t,$ where  $$z_p=\sum\limits_{v=0}^{[\frac{p-1}{2}]}(-1)^v\binom{p-1-v}{v}g^{(n-1)v}z_2^{p-1-2v}$$ for $3\leq p\leq t$.
Since $\mathfrak{q}^nq^{\frac{n}2}=\mathfrak{q}^n$, $z_2z'=(g^{n-1}+1)z'$ by Theorem \ref{L55}. Similarly to the proof of Theorem \ref{L55}, we have $z_3z'=(g^{n-2}+g^{n-1}+1)z'$,
 \begin{equation}
\sum\limits_{v=0}^{[\frac{t-2}{4}]}(-1)^v\binom{\frac{t-2}{2}-v}{v}g^{(n-2)v}(z_3-g^{n-1})^{\frac{t-2}{2}-2v}-g^{n-1}z_{t-1}-g^{n-t}-1=0\nonumber
\end{equation}
and  $z_{2l+1}=\sum\limits_{v=0}^{[\frac{l-1}{2}]}(-1)^v\binom{l-1-v}{v}g^{(n-2)v}(z_3-g^{n-1})^{l-1-2v}$ for $3\leq 2l+1\leq t-1$.
\end{proof}

We always assume that $z_0=0$, $z_1=1$, $z_2=[V_2(1,1,1;0)]$ and $$z_r=\sum\limits_{v=0}^{[\frac{r-1}{2}]}(-1)^v\binom{r-1-v}{v}g^{(n-1)v}z_2^{r-1-2v}\ \text{ for}\  3\leq r\leq t.$$
To determine the Grothendieck ring $G_0(H_\beta)$ for $\beta_1\neq 0$, we need the following lemma.
\begin{lemma}\label{L52}
 Suppose that $\beta_1(\gamma_1^{n_1}-\gamma_2^n)\neq 0$. Then
\begin{eqnarray}
[V_I(\gamma_1,\gamma_2,\gamma_3;i)][V_0(\gamma'_1,\gamma'_2,\gamma'_3;j)]&=&[V_I(\gamma_1\gamma'_1,\gamma_2\gamma'_2,\gamma_3\gamma'_3;i+j)]\nonumber\\
&=&[V_0(\gamma'_1,\gamma'_2,\gamma'_3;j)][V_I(\gamma_1,\gamma_2,\gamma_3;i)].\nonumber
\end{eqnarray}
\end{lemma}
\begin{proof} Since $(\gamma_1^{n_1}-\gamma_2^n)\beta_1\neq 0$, $\beta_1\neq 0$. Thus $\gamma_1'^{n_1}=\gamma_2'^{n}$ and $\gamma_2'=\gamma_1'^{\frac{n_1}n}q^s$ for some $0\leq s\leq n-1$. Hence $((\gamma_1\gamma_1')^{n_1}-(\gamma_2\gamma_2')^n)\beta_1=(\gamma_1')^{n_1}(\gamma_1^{n_1}-\gamma_2^n)\beta_1\neq 0$.
Let $\{m_0,m_1,\cdots,m_{n-1}\}$ be the basis of $V_I(\gamma_1,\gamma_2,\gamma_3;i)$ with the action given by Equations (\ref{eq8})-(\ref{eq10}). Then $\{m_0\otimes 1,m_1\otimes 1,\cdots,m_{n-1}\otimes 1\}$ is a basis of $V_I(\gamma_1,\gamma_2,\gamma_3;i)\otimes V_0(\gamma_1',\gamma_2',\gamma_3';j)$. With this basis, we have
$$a\cdot(m_l\otimes 1)=(a\cdot m_l)\otimes (a\cdot 1)=\sqrt[n]{\gamma_1\gamma'_1}q^{i+j-l}m_l\otimes 1,$$
$$b\cdot(m_l\otimes 1)=(b\cdot m_l)\otimes (b\cdot 1)=(\gamma_2\gamma'_2)m_l\otimes 1,\quad
c\cdot(m_l\otimes 1)=(c\cdot m_l)\otimes (c\cdot 1)=(\gamma_3\gamma'_3)m_l\otimes 1,$$
$$x\cdot(m_l\otimes 1)=(x\cdot m_l)\otimes (a^{n_1}\cdot 1)+ (b\cdot m_l)\otimes (x\cdot 1) =\gamma'^{\frac{n_1}{n}}_1q^{jn_1}m_{l+1}\otimes 1,\; for\; 0\leq l\leq n-2,$$
$$x\cdot(m_{n-1}\otimes 1)=(x\cdot m_{n-1})\otimes (a^{n_1}\cdot 1)+ (b\cdot m_{n-1})\otimes (x\cdot 1) =\beta_1(\gamma_1^{n_1}-\gamma_2^n)\gamma'^{\frac{n_1}{n}}_1q^{jn_1}m_0\otimes 1,$$
$$y\cdot(m_l\otimes 1)=(y\cdot m_l)\otimes (a^{n_1}\cdot 1)+ (c\cdot m_l)\otimes (y\cdot 1) =k_{n-l+1}\gamma'^{\frac{n_1}{n}}_1q^{jn_1}m_{l-1}\otimes 1,\; for\; 1\leq l\leq n-1,$$
$$y\cdot(m_0\otimes 1)=(y\cdot m_0)\otimes (a^{n_1}\cdot 1)+ (c\cdot m_0)\otimes (y\cdot 1) =k_1\gamma'^{\frac{n_1}{n}}_1q^{jn_1}m_{n-1}\otimes 1.$$
Let $u_l=\gamma'^{\frac{n_1l}{n}}_1q^{jn_1l}m_l\otimes 1$. Then $\{u_0,\cdots, u_{n-1}\}$ is a new basis of $V_I(\gamma_1,\gamma_2,\gamma_3;i)\otimes V_0(\gamma_1',\gamma_2',\gamma_3';j)$. Under this basis, we have
\begin{eqnarray*}au_k=\sqrt[n]{\gamma_1\gamma_1'}q^{i+j-k}u_k;\ \
bu_k=\gamma_2\gamma_2'u_k;\  \  cu_k=\gamma_3\gamma_3'u_k\qquad for\quad 0\leq k\leq n-1;
\end{eqnarray*}
\begin{eqnarray*}xu_k=u_{k+1},\qquad for\quad 0\leq k\leq n-2;\qquad
xu_{n-1}=((\gamma_1\gamma_1')^{n_1}-(\gamma_2\gamma_2')^n)\beta_1u_0;
\end{eqnarray*}
\begin{eqnarray*}yu_p=\gamma_1'^{\frac{2n_1}{n}}q^{2jn_1}k_{n-j+1}u_{p-1},\qquad for\quad 1\leq p\leq
n-1;\qquad yu_{0}=\gamma_1'^{\frac{n_1(2-n)}{n}}q^{2jn_1}k_1u_{n-1}.
\end{eqnarray*}
Thus $V_I(\gamma_1,\gamma_2,\gamma_3;i)\otimes V_0(\gamma'_1,\gamma'_2,\gamma'_3;j)\cong V_I(\gamma_1\gamma'_1,\gamma_2\gamma'_2,\gamma_3\gamma'_3;i+j)$.
Similarly, we get $$V_0(\gamma'_1,\gamma'_2,\gamma'_3;j)\otimes V_I(\gamma_1,\gamma_2,\gamma_3;i)\cong V_I(\gamma_1\gamma'_1,\gamma_2\gamma'_2,\gamma_3\gamma'_3;i+j).$$
\end{proof}

\begin{remark}
	 Lemma \ref{L52} implies that
\begin{equation}\label{EQ2}
V_I(\gamma_1,\gamma_2,\gamma_3,i)\cong V_I(\gamma_1',1,\gamma_3';0)\otimes V_0(\gamma_2^{\frac{n}{n_1}},\gamma_2,\gamma_2q^{2n_1i};  i),
\end{equation}
where $\gamma_1'=\gamma_1\gamma_2^{-\frac{n}{n_1}}$ satisfying $\gamma_1'^{n_1}\neq 1$, $\gamma_3'=\gamma_3\gamma_2^{-1}q^{-2n_1i}$.  
If $\beta_2=\beta_3=0$, then $V_I(\gamma_1',1,\gamma_3';0)\cong V_I(\gamma_1',1,1;0)\otimes V_0(1,1,\gamma_3';0)$. If $\beta_2\neq 0$, $\beta_3=0$ and $\gamma_3'=q^{u}$ for some integer $u$, then $V_I(\gamma_1',1,\gamma_3';0)\cong V_I(\gamma_1',1,1;0)\otimes V_0(1,1,\gamma_3';0)$. If $\beta_2=0$, $\beta_3\neq 0$ and $(\gamma_3')^{\frac{n}{2}}=1$, then  $V_I(\gamma_1',1,\gamma_3';0)\cong V_I(\gamma_1'(\gamma_3')^{-\frac{n}{2n_1}},$ $1,1;0)\otimes V_0((\gamma_3')^{\frac{n}{2n_1}},1,\gamma_3';0)$.  If $\beta_2\beta_3\neq 0$ and $(\gamma_3')^{\frac{n}{2}}=\gamma_3'^{n}=1$, then  $V_I(\gamma_1',1,\gamma_3';0)\cong V_I(\gamma_1'(\gamma_3')^{-\frac{n}{2n_1}},$ $1,1;0)\otimes V_0((\gamma_3')^{\frac{n}{2n_1}},1,\gamma_3';0)$.
\end{remark}

\begin{remark}
	Let $R_3$ be the subring of the Grothendieck ring $G_0(H_\beta)$ for $\beta_1\neq 0$ generated by
$R_2$ and $\{[V_I(\gamma_1,\gamma_2,\gamma_3;i)]|\gamma_1^{n_1}\neq \gamma_2^n,i\in\mathbb{Z}_n \}$. Then $R_3$ is isomorphic to a quotient ring of $R_2[x_{\zeta_1,\zeta_2}|\zeta_2\in{\bf K}^* ,\zeta_1\in \widehat{{\bf K}_0}]$, where
$\widehat{{\bf K}_0}=({\bf K}^*/\{\bar{\omega}^r|r\in\mathbb{Z}\})\setminus\{\overline{1}\}$ and $x_{\zeta_1,\zeta_2}=[V_I(\zeta_1,1,\zeta_2;0)]$. Moreover, if $\beta_2=0$ and $\beta_1\beta_3\neq 0$, then $R_3$ is the Grothendieck ring of $G_0(H_\beta)$.
Suppose that $\beta_2=\beta_3=0$ and $\beta_1\neq 0$. Then the Grothendieck ring $G_0(H)$  is isomorphic to a quotient ring of the algebra $R_2'=R_1[x_{\zeta,1}|\zeta\in \widehat{{\bf K}_0}]$.

\end{remark}

 Let  $r_p=(r-2p)\ mod(t)$  and $1\leq r_p\leq t$, where $r$ is the minimal positive integer such that  $\zeta_2\zeta_2'(\zeta_1\zeta_1')^{-\frac{2n_1}{n}}=q^{(-r+1)n_1}$. Then $r_p$ is the minimal positive integer such that $$\zeta_2\zeta_2'(\zeta_1\zeta_1')^{-\frac{2n_1}{n}}=q^{(-2p-r+1)n_1}.$$

\begin{theorem}\label{V} Suppose that $\beta_1\beta_3\neq 0$, $\beta_2=0$ and ${\bf K}^{**}=({\bf K}^*/\{\zeta\in{\bf K}|\zeta^{\frac{n}{2}}=1\})\setminus\{\bar{1}\}$. Let $z_\xi''=[V_t(1,1,\xi;0)]$ for $\xi\in\overline{{\bf K}^*}:=({\bf K}^* /\langle q^{n_1}\rangle)\setminus\{\bar{1}\}$. Then the Grothendieck ring $R_3:=G_0(H_\beta)$ is isomorphic to the commutative ring
$R_1[z_2, x_{\zeta_1,\zeta_2},z_{\xi}''|\zeta_1\in \widehat{{\bf K}_0},\zeta_2\in{\bf K}^{**},\xi\in \overline{{\bf K}^*}]$ with relations (\ref{Eq*1}),

 $$z_2x_{\zeta_1,\zeta_2}=x_{\zeta_1q^\frac{n}{2},\zeta_2}+g^{n-1}x_{\zeta_1q^\frac{n}{2},\zeta_2}, \quad z_2z_\xi''=\eta(z_{\xi q^{-n_1}}''+g^{n-1}z_{\xi q^{-n_1}}''),$$
 $$x_{\zeta_1,\zeta_2}z_\xi''=\mathfrak{s}''x_{\zeta_1,\zeta_2\xi}, \quad z_\xi''z_{\xi'}''=\mathfrak{s}''z_{\xi\xi'}'' \ for\ \xi\xi'\in \overline{{\bf K}^*},$$
$$z_\xi''z_{\xi'}''=\sum\limits^{t-1}_{p=0,r_p'<t}g^{n-p}(g^{-r_p'}g_{q^{-\frac{(t-r_p'-1)n}{2}},1,\xi\xi'}z_{t-r_p'}+g_{q^{-\frac{(r_p'-1)n}{2}},1,\xi\xi'}z_{r_p'})+\sum\limits_{i=0,r_p'=t}^{t-1}g^{n-p}g_{q^{-\frac{(t-1)n}{2}},1,\xi\xi'}z_t$$ for $r'$ is the minimal positive integer such that $\xi\xi'=q^{r'n_1}$, and
\begin{equation*}
x_{\zeta_1,\zeta_2}x_{\zeta_1',\zeta_2'}=
\begin{cases}
\mathfrak{s}x_{\zeta_1\zeta_1',\zeta_2\zeta_2'}, \ & \text{if}\ \ (\zeta_1\zeta_1')^{n_1}\neq 1\\
\mathfrak{s}'g_{\zeta_1\zeta_1',1,\zeta_2\zeta_2'}\sum\limits_{p=0}^{n-1}(g^{n-p-r_p}z_{t-r_p}+ g^{n-p}z_{r_p}), \ &\text{if}\ \  (\zeta_1\zeta_1')^{n_1}=1,(\zeta_1\zeta_1')^{-\frac{2n_1}{n}}\zeta_2\zeta_2\in\langle q^{n_1}\rangle,\\
u\mathfrak{s}g_{\zeta_1\zeta_1',1,1}z_{\zeta_2\zeta_2'}'', \ &\text{if}\ \  (\zeta_1\zeta_1')^{n_1}=1,(\zeta_1\zeta_1')^{-\frac{2n_1}{n}}\zeta_2\zeta_2\in \overline{{\bf K}^*},
\end{cases}
\end{equation*}
where $\mathfrak{s}'=\sum\limits_{k=1}^ug^{kt}\in R_1$, 
$u=\frac{n}{t}$, $z_r=\sum\limits_{v=0}^{[\frac{r-1}{2}]}(-1)^v\binom{r-1-v}{v}g^{(n-1)v}z_2^{r-1-2v}$ for $3\leq r\leq t$ and $\eta=[V_0(q^\frac{n}{2},1,q^{n_1};0)]$.
\end{theorem}

\begin{proof}  Since $\beta_2=0$, $G_0(H_\beta)$ is generated by $[V_I(\gamma_1,\gamma_2,\gamma_3;i)], [V_0(\gamma_1,\gamma_2,\gamma_3;i)]$ and $[V_r(\gamma_1,\gamma_2,\gamma_3;i)]$ for $r\leq t$. Suppose that $\gamma_1^{n_1}=\gamma_2^n$ and $\gamma_1^{-\frac{2n_1}{n}}\gamma_2\gamma_3\neq q^{n_1v}$ for any integer $v$.  
Then $$[V_t(\gamma_1,\gamma_2,\gamma_3;i)]=[V_t(1,1,\gamma_3\gamma_2^{-1};0)][V_0(\gamma_1,\gamma_2,\gamma_2;i)],$$
where $\gamma_2^{-1}\gamma_3\neq q^{n_1v}$ for any integer $v$. In addition, we have  
$$[V_t(1,1,\, \gamma_2^{-1}\gamma_3 q^{{n_1}}; 0)]=[V_t(1,1,\gamma_2^{-1}\gamma_3,0)][V_0(1,1,q^{{n_1}};0)].$$
Therefore $G_0(H_\beta)$ is generated by $R_1$, $z_2$, $x_{\zeta_1,\zeta_2}$ and $z_{\zeta}''=[V_t(1,1,\xi;0)]$ for $\xi\in({\bf K}^*/\langle q^{n_1}\rangle)\setminus\{\bar{1}\}$.

In the following, we determine the relations of the generators in several subcases. 

(1) Similar to the proof of Theorem \ref{L55}, we have Equation (\ref{Eq*1}), $$z_2z_\xi''=\eta(z_{\xi q^{-n_1}}''+g^{n-1}z_{\xi q^{-n_1}}''),\qquad z_\xi''z_{\xi'}''=\sum\limits_{i=0}^{t-1} g^{n-p}z_{\xi\xi'}'' \text { for }\  \xi\xi'\in \overline{{\bf K}^*},$$
 and 
 $$z_\xi''z_{\xi'}''=\sum\limits^{t-1}_{p=0,r_p'<t}g^{n-p}(g^{-r_p'}g_{q^{-\frac{(t-r_p'-1)n}{2}},1,\xi\xi'}z_{t-r_p'}+g_{q^{-\frac{(r_p'-1)n}{2}},1,\xi\xi'}z_{r_p'})+\sum\limits_{i=0,r_p'=t}^{t-1}g^{n-p}g_{q^{-\frac{(t-1)n}{2}},1,\xi\xi'}z_t$$
for $\xi\xi'=q^{r'n_1}$, where  $z_r=\sum\limits_{v=0}^{[\frac{r-1}{2}]}(-1)^v\binom{r-1-v}{v}g^{(n-1)v}z_2^{r-1-2v}$ for $3\leq r\leq t$, $\eta=[V_0(q^\frac{n}{2},1,q^{n_1};0)]$,  and $r_p'=(r'-2p)(mod\ t)$ for $1\leq p\leq t$, $r'$ is the minimal positive integer such that $\xi\xi'=q^{r'n_1}$.\\

(2) Let $\{m_0,\cdots,m_{n-1}\}$ (resp. $\{m_0',m_1'\}$) be the basis of $V_I(\zeta_1,1,\zeta_2;0)$ (resp. $V_2(1,1,1;0)$) with the actions given by Equations (\ref{eq8})-(\ref{eq10})(resp. Equations (\ref{eq15})-(\ref{eq17}), where $k_i$ is replaced by $k_i'$). Then $\{m_l\otimes m'_v,\; 0\leq l\leq n-1,0\leq v\leq 1\}$ is a basis of $V_I(\zeta_1,1,\zeta_2;0)\otimes V_2(1,1,1;0)$. With this basis, we have
$$a\cdot(m_l\otimes m'_v)=(a\cdot m_l)\otimes (a\cdot m'_v)=\zeta_1^{\frac{1}{n}}q^{\frac{1}{2}-l-v}m_l\otimes m'_v,$$
$$b\cdot(m_l\otimes m'_v)=(b\cdot m_l)\otimes (b\cdot m'_v)=m_l\otimes m'_v,\quad
c\cdot(m_l\otimes m'_v)=(c\cdot m_l)\otimes (c\cdot m'_v)=\zeta_2m_l\otimes m'_v,$$
\begin{eqnarray}y(m_l\otimes m'_v)=
\begin{cases}
q^\frac{n_1}{2}k_{n-l+1}m_{l-1}\otimes m'_0,\ & \text{if} \ 1\leq l\leq n-1,\;v=0\\
q^\frac{n_1}{2}k_1m_{n-1}\otimes m'_0, \ & \text{if}\ l=v=0\\
q^{\frac{n_1}{2}-kn_1}k_{n-l+1}m_{l-1}\otimes m'_1+\zeta_2k'_1m_l\otimes m'_0,\ & \text{if} \ 1\leq l\leq n-1,\;v=1\\
q^{\frac{n_1}{2}-kn_1}k_1m_{n-1}\otimes m'_1+\zeta_2k'_1m_0\otimes m'_0,\ & \text{if} \ l=0,\;v=1
\end{cases}\nonumber
\end{eqnarray}
and
\begin{eqnarray}x^t(m_l\otimes m'_v)=
\begin{cases}
m_{l+t}\otimes m'_v, &\text{if}\  0\leq l\leq n-t-1
\cr \beta_1(\gamma_1^{n_1}-1)m_{l+t-n}\otimes m'_v, &\text{if}\  n-t\leq l\leq n-1
\end{cases}.\nonumber
\end{eqnarray}
If $(n_1,n)=u$, then $|q^{n_1}|=|q^u|=\frac{n}{u}$. Hence $n=ut$ and
\begin{eqnarray}
x^{n-1}(m_l\otimes m'_v)=\sum\limits^{1-v}_{l'=0}\binom{t-1}{l'}_{q^{n_1}}q^{\frac{(n-1-l')n_1}{2}+(-l-l')(t-1-l')n_1}x^{n-1-l'}m_{l}\otimes m'_{v+l'}.\nonumber
\end{eqnarray}
Let $v_0=m_0\otimes m'_0+\alpha_1m_{n-1}\otimes m'_1$ , $v_1=m_0\otimes m'_1$ and  $yv_0=\theta x^{n-1}v_0$. Then
$$\begin{array}{lll}
yv_0&=&\alpha_1q^{-\frac{n_1}{2}}k_2m_{n-2}\otimes m'_1+(q^{\frac{n_1}{2}}k_1+\alpha_1\zeta_2k'_1)m_{n-1}\otimes m'_0,
\end{array}$$
and
$$\begin{array}{lll}
x^{n-1}v_0&=&q^\frac{(n-1)n_1}{2}m_{n-1}\otimes m'_0-q^{\frac{(n-2)n_1}{2}-(t-1)n_1}m_{n-2}\otimes m'_1\\
&&+\alpha_1\beta_1(\zeta_1^{n_1}-1)q^{\frac{(n-1)n_1}{2}-(t-1)n_1}m_{n-2}\otimes m'_1.
\end{array}$$
Thus
\begin{equation}
 \alpha_1k_2-q^{-(t-1)n_1}(q^{\frac{n_1}{2}}k_1+\alpha_1\zeta_2k'_1)(\alpha_1\beta_1(\zeta_1^{n_1}-1)q^{\frac{n_1}{2}}-1)=0.\nonumber
\end{equation}
Hence we can determine the $\alpha_1$.  In addition, $H_\beta v_0\cong V_I(\zeta_1q^\frac{n}{2},1,\zeta_2;0)$ is an irreducible submodule of $V_I(\zeta_1,1,\zeta_2;0)\otimes V_2(1,1,1;0)$ by the proof of Lemma \ref{lem52}.

By the same method, we prove that $$(H_\beta v_1+H_\beta v_0)/H_\beta v_0\cong V_I(\zeta_1q^\frac{n}{2},1,\zeta_2;n-1)\cong V_I(\zeta_1q^\frac{n}{2},1,\zeta_2;0)\otimes V_0(1,1,1;n-1)$$ is an irreducible submodule of $V_I(\zeta_1,1,\zeta_2;0)\otimes V_2(1,1,1;0)/H_\beta v_0$.
Hence 
\begin{equation*}
x_{\zeta_1,\zeta_2}z_2=x_{\zeta_1q^\frac{n}{2},\zeta_2}+g^{n-1}x_{\zeta_1q^\frac{n}{2},\zeta_2},
\end{equation*}
where $g=[V_0(1,1,1;1)]$. Likewise, we have 
$z_2x_{\zeta_1,\zeta_2}=x_{\zeta_1q^\frac{n}{2},\zeta_2}+g^{n-1}x_{\zeta_1q^\frac{n}{2},\zeta_2}.$ \\

(3) Similarly, we obtain $$[V_I(\zeta_1,1,\zeta_2;0)][V_t(1,1,\xi;0)]=\sum\limits^{t-1}_{p=0}[V_I(\zeta_1,1,\zeta_2\xi;n-p)]=[V_t(1,1,\xi;0)][V_I(\zeta_1,1,\zeta_2;0)].$$
Hence $x_{\zeta_1,\zeta_2}z_\xi''=\sum\limits^{t-1}_{p=0}g^{n-p}x_{\zeta_1,\zeta_2\xi}$.

(4) Let $\{m_0,\cdots,m_{n-1}\}$ (resp. $\{m_0',\cdots,m_{n-1}'\}$) be a basis of $V_I(\zeta_1,1,\zeta_2;0)$ (resp. $V_I(\zeta_1',1,\zeta_2';0)$) with the action given by (\ref{eq8})-(\ref{eq10}) (resp.  (\ref{eq8})-(\ref{eq10}) with $k_i$ replaced by $k_i'$). Then $\{m_l\otimes m'_v,\; 0\leq l,v\leq n-1\}$ is a basis of $V_I(\zeta_1,1,\zeta_2;0)\otimes V_I(\zeta_1',1,\zeta_2';0)$ and the action on this basis is given by
$$a\cdot(m_l\otimes m'_v)=(a\cdot m_l)\otimes (a\cdot m'_v)=(\zeta_1\zeta_1')^\frac{1}{n}q^{-l-v}m_l\otimes m'_v,$$
$$b\cdot(m_l\otimes m'_v)=(b\cdot m_l)\otimes (b\cdot m'_v)=m_l\otimes m'_v,$$
$$c\cdot(m_l\otimes m'_v)=(c\cdot m_l)\otimes (c\cdot m'_v)=(\zeta_2\zeta_2')m_l\otimes m'_v,$$
\begin{eqnarray}y(m_l\otimes m'_v)=
\begin{cases}
\zeta_1'^{\frac{n_1}{n}}k_1m_{n-1}\otimes m'_0+\zeta_2k'_1m_0\otimes m'_{n-1}, &\text{if}\ l=v=0;
\cr \zeta_1'^{\frac{n_1}{n}}q^{-vn_1}k_1m_{n-1}\otimes m'_v+\zeta_2k'_{n-v+1}m_0\otimes m'_{v-1},  &\text{if}\ l=0,\;1\leq v\leq n-1;
\cr \zeta_1'^{\frac{n_1}{n}}k_{n-l+1}m_{l-1}\otimes m'_0+\zeta_2k'_1m_l\otimes m'_{n-1}, &\text{if}\ 1\leq l\leq n-1;\;
\cr \zeta_1'^{\frac{n_1}{n}}q^{-vn_1}k_{n-l+1}m_{l-1}\otimes m'_v+\zeta_2k'_{n-v+1}m_l\otimes m'_{v-1}, &\text{if}\ 1\leq l,v\leq n-1;
\end{cases}\nonumber
\end{eqnarray}
\begin{eqnarray}x^t(m_l\otimes m'_v)=
\begin{cases}
m_l\otimes m'_{t+v}+\zeta_1'^{\frac{tn_1}{n}}m_{l+t}\otimes m'_v, &\text{if}\  t+v,t+l<n;
\cr m_l\otimes m'_{t+v}+\beta_1(\zeta_1^{n_1}-1)\zeta_1'^{\frac{tn_1}{n}}m_{l+t-n}\otimes m'_v,
 &\text{if}\  t+v<n,t+l\geq n;
\cr \beta_1((\zeta_1'^{n_1}-1)m_l\otimes m'_{t+v-n}+(\zeta_1^{n_1}-1)\zeta_1'^{\frac{tn_1}{n}}m_{l+t-n}\otimes m'_v),
 &\text{if}\  t+v,t+l\geq n;
\cr \beta_1(\zeta_1'^{n_1}-1)m_l\otimes m'_{t+v-n}+\zeta_1'^{\frac{tn_1}{n}}m_{l+t}\otimes m'_v,
 &\text{if}\  t+v\geq n,t+l<n.
\end{cases}\nonumber
\end{eqnarray}

(4.1) If $((\zeta_1\zeta_1')^{n_1}-1)\beta_1\neq0$, let $v_p=m_0\otimes m'_p+\sum\limits^{n-1}_{l=p+1}\alpha_lm_{n+p-l}\otimes m'_l$ for $0\leq p\leq n-1$. Assume that $yv_0=\theta x^{n-1}v_0$, similar to (1), we can determine $\alpha_l, 1\leq l\leq n-1$.
Hence, $A_0:=H_\beta v_0\cong V_I(\zeta_1\zeta_1',1,\zeta_2\zeta_2';0)$ is an irreducible submodule of $V_I(\zeta_1,1,\zeta_2;0)\otimes V_I(\zeta_1',1,\zeta_2';0)$.

By using the same method, we can determine that $$A_p:=(H_\beta v_p+\sum^{p-1}_{l=0}H_\beta v_l)/\sum\limits^{p-1}_{l=0}H_\beta v_l\cong V_I(\zeta_1\zeta_1',1,\zeta_2\zeta_2';n-p)$$ is an irreducible submodule of $V_I(\zeta_1,1,\zeta_2;0)\otimes V_I(\zeta_1',1,\zeta_2';0)/\sum^{p-1}_{l=0}H_\beta v_l$ for $1\leq p\leq n-1$.
Hence $[V_I(\zeta_1,1,\zeta_2;0)][V_I(\zeta_1',1,\zeta_2';0)]=\sum\limits^{n-1}_{p=0}[V_I(\zeta_1\zeta_1',1,\zeta_2\zeta_2';n-p)]=\sum\limits^{n-1}_{p=0}[V_I(\zeta_1\zeta_1',1,\zeta_2\zeta_2';0)][ V_0(1,1,1;n-p)]$.
Consequently, $x_{\zeta_1,\zeta_2}x_{\zeta_1',\zeta_2'}=\sum\limits^{n-1}_{p=0}g^{n-p}x_{\zeta_1\zeta_1',\zeta_2\zeta_2'}= \sum\limits^{n-1}_{p=0}g^{p}x_{\zeta_1\zeta_1',\zeta_2\zeta_2'}$.

(4.2) We discuss the case in which $((\zeta_1\zeta_1')^{n_1}-1)\beta_1=0$. Since $\beta_2=0$, we have $y^n=0$. There exists an element $v'\in V_I(\zeta_1,1,\zeta_2;0)\otimes V_I(\zeta_1',1,\zeta_2';0)$ such that $yv'=0$.  We  assume that $v'=\sum\limits^{d}_{l=0}\alpha_lm_{d-l}\otimes m'_l+\sum\limits^{n-1}_{l=d+1}\alpha_lm_{n+d-l}\otimes m'_l$ for $ay=qya$. Let $$yv'=(m_{d-1}\otimes m'_1,m_{d-2}\otimes m'_2,\cdots,m_d\otimes m'_0)C(\alpha_0,\cdots,\alpha_{n-1})^T,$$ then $det(C)=0$.
Let $v_p=\sum\limits^{p}_{l=0}\alpha_lm_{p-l}\otimes m'_l+\sum\limits^{n-1}_{l=p+1}\alpha_lm_{n+p-l}\otimes m'_l$ for $0\leq p\leq n-1$,
$$yv_p=(m_{p-1}\otimes m'_0,m_{p-2}\otimes m'_1,\cdots,m_p\otimes m'_{n-1})D_p(\alpha_0,\cdots,\alpha_{n-1})^T,$$
$$x^{n-1}v_p=(m_{p-1}\otimes m'_0,m_{p-2}\otimes m'_1,\cdots,m_p\otimes m'_{n-1})B_p(\alpha_0,\cdots,\alpha_{n-1})^T,$$
 where
$$D_p=\left(\begin{array}{cccc}
\zeta_1'^{\frac{n_1}{n}}k_{n-p+1} & \zeta_2k'_n & \cdots & 0 \\
\vdots & \ddots & \ddots & \vdots \\
0 & \cdots & \zeta_1'^{\frac{n_1}{n}}k_{n-p-1}q^{(-n+2)n_1} & \zeta_2k'_2 \\
\zeta_2k'_1 & \cdots & 0 & \zeta_1'^{\frac{n_1}{n}}k_{n-p}q^{(-n+1)n_1} \\
\end{array}
\right),$$
$$B_p=\left(\begin{array}{ccc}
\zeta_1'^{\frac{(n-1)n_1}{n}}  & \cdots  & \binom{t-1}{1}_{q^{n_1}}\beta_1(\zeta_1'^{n_1}-1)\zeta_1'^{\frac{(n-2)n_1}{n}} \\
\vdots  & \ddots & \vdots \\
\beta_1(\zeta_1^{n_1}-1) & \cdots  & \zeta_1'^{\frac{(n-1)n_1}{n}}q^{(-n+1)(t-1)n_1} \\
\end{array}
\right).$$
Since $det(D_p)=det(C)=0$ and $D_p-\lambda B_p\neq 0$ for any $\lambda$,  there exists $\alpha_l$ for $0\leq l\leq n-1$ such that $yv_p=0$ and $x^{n-1}v_p\neq0$. In fact, $Hv_i\cap Hv_j=0$ for any $0\leq i<j\leq n-1$. If there is a $\lambda\in \bf{K}$ such that $v_j=\lambda x^{j-i}v_i$. Then $x^{n-1}v_1=\lambda x^{n-1+j-i}v_0=\lambda\beta_1((\zeta_1\zeta_1')^{n_1}-1)x^{j-i-1}v_i=0$, which is a contradiction. Thus, $V_I(\zeta_1,1,\zeta_2;0)\otimes V_I(\zeta_1',1,\zeta_2';0)=\bigoplus\limits^{n-1}_{p=0}H_\beta v_p$.
Notice that
\begin{eqnarray}
yx^{k'}(m_l\otimes m'_v)&=&q^{-n_1k'}x^{k'}y(m_l\otimes m'_v)\nonumber\\
&+&\frac{1-q^{-k'n_1}}{1-q^{-n_1}}\beta_3(q^{(2(-v-l)-k'+1)n_1}(\zeta_1\zeta_1')^{\frac{2n_1}{n}}-(\zeta_2\zeta_2'))x^{k'-1}(m_l\otimes m'_v).\nonumber
\end{eqnarray}

(4.2.1) In the case that there exists an integer $l$ such that $\zeta_2\zeta_2'-(\zeta_1\zeta_1')^{\frac{2n_1}{n}}q^{(-2p-l+1)n_1}=0$.
 Let $r_p$ be the minimal positive integer such that $\zeta_2\zeta_2'(\zeta_1\zeta_1')^{-\frac{2n_1}{n}}=q^{(-2p-r_p+1)n_1}$.  Then $1\leq r_p\leq t$ and $r_p=(r-2p)\ mod(t)$, where $r$ is the minimal positive integer such that $\zeta_2\zeta_2'(\zeta_1\zeta_1')^{-\frac{2n_1}{n}}=q^{(-r+1)n_1}$.
So $r_p$, which satisfies $1\leq r_p\leq t$, is the minimal positive integer such that $yx^{r_p}v_p=0$. Thus $yx^{kt+r_p}v_p=0$ for $0\leq k\leq u-1$, where $u=\frac{n}{t}$. Notice that $yx^{kt}v_p=0$ for $0\leq k\leq u-1$. Hence
$A_{p,2k+1}:=H_\beta x^{(u-k-1)t+r_p}v_p/H_\beta x^{(u-k)t}v_p\cong V_{t-r_p}(\zeta_1\zeta_1',1,\zeta_2\zeta_2';n-p+(k+1)t-r_p)\cong V_{t-r_p}(1,1,1;0)\otimes V_0(\zeta_1\zeta_1',1,\zeta_2\zeta_2';n-p+(k+1)t-r_p)\cong V_{t-r_p}(1,1,1;0)\otimes V_0(\zeta_1\zeta_1',1,\zeta_2\zeta_2';0)\otimes V_0(1,1,1;1)^{\otimes (n-p+(k+1)t-r_p)}$,
$A_{p,2(k+1)}:=H_\beta x^{(u-k-1)t}v_p/H_\beta x^{(u-k-1)t+r_p}v_p\cong V_{r_p}(\zeta_1\zeta_1',1,\zeta_2\zeta_2';n-p+(k+1)t)\cong V_{r_p}(1,1,1;0)\otimes V_0(\zeta_1\zeta_1',1,\zeta_2\zeta_2';n-p+(k+1)t)\cong V_{r_p}(1,1,1;0)\otimes V_0(\zeta_1\zeta_1',1,\zeta_2\zeta_2';0)\otimes V_0(1,1,1;1)^{\otimes(n-p+(k+1)t)}$
for $0\leq k\leq u-1$. Therefore $[A_{p,2k+1}]=z_{t-r_p}g^{(k+1)t-p-r_p}g_{\zeta_1\zeta_1',1,\zeta_2\zeta_2'}$ and $[A_{p,2(k+1)}]=z_{r_p}g^{(k+1)t-p}g_{\zeta_1\zeta_1',1,\zeta_2\zeta_2'}$.
Consequently, 

$$\begin{array}{lll}x_{\zeta_1,\zeta_2}x_{\zeta_1',\zeta_2'}=\sum\limits_{p=0}^{n-1}[H_\beta v_p]&= &\sum\limits_{p=0}^{n-1}\sum\limits^{2u}_{v=1}[A_{p,v}]\\
&=&\sum\limits_{p=0}^{n-1}(\sum\limits_{k=0}^{u-1}g^{(k+1)t-p-r_p}z_{t-r_p}+\sum\limits_{k=0}^{u-1}g^{(k+1)t-p}z_{r_p})g_{\zeta_1\zeta_1',1,\zeta_2\zeta_2'}\\
&=&\sum\limits_{p=0}^{n-1}(g^{n-p-r_p}z_{t-r_p}+ g^{n-p}z_{r_p})\mathfrak{s}'g_{\zeta_1\zeta_1',1,\zeta_2\zeta_2'},\end{array}$$ 
where $\mathfrak{s}'=\sum\limits_{k=1}^{u}g^{kt}$.

(4.2.2) If $\zeta_2\zeta_2'\neq (\zeta_1\zeta_1')^{\frac{2n_1}{n}}q^{vn_1}$ for any integer $v$, then $yx^{kt}v_p=0$ for $0\leq k\leq u-1$, where $u=\frac{n}{t}$. Thus $A_{p,v}:=H_\beta x^{(u-v)t}v_p/H_\beta x^{(u-v+1)t}v_p \cong V_{t}(\zeta_1\zeta_1',1,\zeta_2\zeta_2';n-p+vt)\cong V_0(\zeta_1\zeta_1',1,1;n-p+vt)\otimes V_t(1,1,\zeta_2\zeta_2';0)$ for $1\leq v\leq u$. Hence $[A_{p,v}]=g_{\zeta_1\zeta_1',1,1}g^{n-p+vt}z_{\zeta_2\zeta_2'}''$.
Consequently,  $x_{\zeta_1,\zeta_2}x_{\zeta_1',\zeta_2'}=\sum\limits^{n-1}_{p=0}\sum\limits^{u}_{v=1}[A_{p,v}]=u\mathfrak{s}g_{\zeta_1\zeta_1',1,1}z_{\zeta_2\zeta_2'}''$, where $u=\frac{n}{t}$.
\end{proof}

%%%%%%%%%%%%%%%%%% corollary%%%%%%%%%%%%%%%%%%%%%%%%%%%%%%%%%%%%%%%%%%%%%%%%%%%%%
%%%%%%%%%%%%%%%%%%%%%%%%%%%%%%%%%%%%%%%%%%%%%%%%%%%%%%%%%%%%%%%%%%%%%%%%%%%%%%%%%
%%%%%%%%%%%%%%%%%%%%%%%%%%%%%%%%%%%%%%%%%%%%%%%%%%%%%%%%%%%%%%%%%%%%%%%%%%%%%%%%%
%%%%%%%%%%%%%%%%%%%%%%%%%%%%%%%%%%%%%%%%%%%%%%%%%%%%%%%%%%%%%%%%%%%%%%%%%%%%%%%%%
Especially, we obtain the structure of the Grothendieck ring  of the Gelaki's Hopf alegbra where $\beta_1\beta_3\neq 0$.

\begin{corollary}\label{cor59} 
{\rm(1)}
Suppose that $\beta_1\beta_3\neq 0$ and $(N,nn_1,2n_1t)<(nn_1,N)$.

{\rm(1.1)} If either $2n\mid N$, or $2n\nmid N$ and $2(n,n_1)\nmid n$, then the Grothendieck ring $$S_3:=G_0(\mathcal{U}_{(n,N,n_1,q,\beta_1,0,\beta_3)})$$ is isomorphic to the commutative ring
$S_2[ x^*]=\mathbb{Z}[g,h_2,z_2,z'',x^*]$ with relations Equation (\ref{Eq*1}),
\begin{equation}\label{T2}
z''^{v_1}={\mathfrak{s}''}^{v_1-2}\left(\sum\limits^{t-1}_{p=0,r_p<t}(g^{n-p-r_p}z_{t-r_p}+g^{n-p}z_{r_p})+\sum\limits_{i=0,r_p=t}^{t-1}g^{n-p}z_t\right),
\end{equation}
$z_2z''=(1+g^{n-1})z'',$ $z_2x^*=(1+g^{n-1})x^*$ and ${x^*}^{\frac{N}{(N/n,n_1)}}=u\mathfrak{s}^{\frac{N}{(N/n,n_1)}-1}z''$.

{\rm(1.2)} If $2n\nmid N$ and $2(n,n_1)\mid n$, then the Grothendieck ring $S_3:=G_0(\mathcal{U}_{(n,N,n_1,q,\beta_1,0,\beta_3)})$ is isomorphic to the commutative ring
$S_2[ x^*]=\mathbb{Z}[g,h_2,z_3,z'',x^*]$ with relations  (\ref{T3}),  (\ref{T2}), $z_3z''=(1+g^{n-1}+g^{n-2})z'',$ $z_3x^*=(1+g^{n-1}+g^{n-2})x^*$ and ${x^*}^{\frac{N}{(N/n,n_1)}}=u\mathfrak{s}^{\frac{N}{(N/n,n_1)}-1}z''$.

{\rm(2)} Suppose that $\beta_1\beta_3\neq 0$ and $(N,nn_1,2n_1t)=(nn_1,N)$. Then the Grothendieck ring $$G_0(\mathcal{U}_{(n,N,n_1,q,\beta_1,0,\beta_3)})\cong
\begin{cases}
\mathbb{Z}[g,h_2,z_2,x^*], &\text{if}\ \ 2n\mid N \ \text{or}\ 2n\nmid N \ \text{and}\ 2(n,n_1)\nmid n\\
\mathbb{Z}[g,h_2,z_3,x^*], &\text{if}\ \ 2n\nmid N\ \text{and}\ 2(n,n_1)\mid n
\end{cases}$$ is a subring of $S_3$.
\end{corollary}

\begin{proof} Suppose that $V_I(\gamma,1,1;0)$ is an irreducible representation of
$\mathcal{U}_{(n,N,n_1,q,\beta_1,0,\beta_3)}$. Then $\gamma^{\frac{N}n}=1$ and $\gamma^{n_1}\neq 1$. Assume that $\gamma^{\frac{N}n}=\gamma^{n_1}=1$.
Then $\gamma=\mathfrak{q}^{nr}$ for some integer $r$ and $N|nn_1r$. Thus $\gamma\in\langle\mathfrak{q}^{\frac{Nn}{(N,nn_1)}}\rangle$. We further assume that $\gamma^\frac{2n_1}{n}\in \langle q^{n_1}\rangle$, then $\gamma\in \langle\mathfrak{q}^{\frac{Nn}{(N,nn_1)}}\rangle\cap \langle\mathfrak{q}^{\frac{Nn}{(N,2n_1t)}}\rangle=\langle\mathfrak{q}^{\frac{Nn}{(N,nn_1,2n_1t)}}\rangle$.
 Let $v_1=\frac{(nn_1,N)}{(nn_1,2n_1t,N)}$ and $\mathfrak{q}^{\frac{2n_1N}{(N,nn_1,2n_1t)}}=q^{(r-1)n_1}$ for some $2\leq r\leq t$. Then
\begin{equation}
z''^{v_1}={\mathfrak{s}''}^{v_1-2}\left(\sum\limits^{t-1}_{p=0,r_p<t}(g^{n-p-r_p}z_{t-r_p}+g^{n-p}z_{r_p})+\sum\limits_{i=0,r_p=t}^{t-1}g^{n-p}z_t\right),\nonumber
\end{equation}
$z_2z''=(1+g^{n-1})z''$ and $z_3z''=(1+g^{n-1}+g^{n-2})z'',$ where $z''=[V_t(q^{\frac{Nn}{(N,nn_1)}},1,1;0)]$. Let $x^*=[V_I(\mathfrak{q}^n,1,1;0)]$. Then $z_2x^*=(1+g^{n-1})x^*$, $z_3x^*=(1+g^{n-1}+g^{n-2})x^*$, and ${x^*}^{\frac{N}{(N/n,n_1)}}=u\mathfrak{s}^{\frac{N}{(N/n,n_1)}-1}z''$
by Theorem \ref{V}.
\end{proof}

\begin{theorem}\label{thmVI} Suppose that $\beta_1\neq 0$ and $\beta_2=\beta_3=0$. Then the Grothendieck ring $G_0(H_\beta)=R_2':=R_1[x_{\zeta,1}|\zeta\in\widehat{{\bf K}_0}]$ with relations
\begin{equation}
x_{\zeta,1}x_{\zeta',1}=\begin{cases}
\mathfrak{s}x_{\zeta\zeta',1}, \ &\text{if} \ (\zeta\zeta')^{n_1}\neq 1\\
n\mathfrak{s}g_{\zeta\zeta',1,1},\ &\text{if} \ (\zeta\zeta')^{n_1}=1.
\end{cases}
\end{equation}
\end{theorem}

\begin{proof} In the case that $(\zeta_1\zeta_1')^{n_1}-1\neq0$, according to the proof of Theorem \ref{V}, we have
\begin{eqnarray*}
[V_I(\zeta_1,1,1;0)][V_I(\zeta_1',1,1;0) ]&= & \sum\limits^{n-1}_{p=0}[V_I(\zeta_1\zeta_1',1,1;n-p)]\\
 & =& \sum\limits^{n-1}_{p=0}[V_I(\zeta_1\zeta_1',1,1;0)][V_0(1,1,1;n-p)].
\end{eqnarray*}
Hence, $x_{\zeta,1}x_{\zeta',1}=\sum\limits^{n-1}_{p=0}g^{n-p}x_{\zeta\zeta',1}=\mathfrak{s}x_{\zeta\zeta',1}$.

Next we discuss the case that $(\zeta_1\zeta_1')^{n_1}=1$. Let $\{m_0,\cdots,m_{n-1}\}$ and $\{m_0',\cdots,m_{n-1}'\}$ be the basis of $V_I(\zeta_1,1,1;0)$ and $V_I(\zeta_1',1,1;0)$ with the actions given by Equations (\ref{eq8})-(\ref{eq10}) respectively. Then $\{m_l\otimes m'_k,\; 0\leq l,k\leq n-1\}$ is a basis of $V_I(\zeta_1,1,1;0)\otimes V_I(\zeta_1',1,1;0)$. From the  proof of Theorem \ref{V}, we have $V_I(\zeta_1,1,1;0)\otimes V_I(\zeta_1',1,1;0)=\bigoplus\limits^{n-1}_{p=0}H_\beta v_p$,  where $v_p=\sum\limits^{p}_{l=0}\alpha_lm_{p-l}\otimes m'_l+\sum\limits^{n-1}_{l=p+1}\alpha_lm_{n+p-l}\otimes m'_l$ in the proof of Theorem \ref{V} for $0\leq p\leq n-1$. Since $\beta_3=0$, we obtain that $yxv_p=0$ for $0\leq p\leq n-1$. Hence
$$[V_I(\zeta_1,1,1;0)][V_I(\zeta_1',1,1;0)]=\sum\limits^{n-1}_{p=0}\sum\limits^{n}_{v=1}[V_0(\zeta_1\zeta_1',1,1;n-p-v)].$$ Since $g^n=1$, $\sum\limits_{p=0}^{n-1}\sum\limits_{v=1}^ng^{n-v-p}=\sum\limits_{p=0}^{n-1}g^{n-p}\sum\limits_{v=1}^{n}g^{n-v}=n\sum\limits_{p=0}^{n-1}g^p.$ Thus
 \begin{eqnarray*}
[V_I(\zeta_1,1,1;0)][V_I(\zeta_1',1,1;0) ]&= & \sum\limits_{p=0}^{n-1}\sum\limits_{v=1}^ng^{n-p-v}[V_0(\zeta_1\zeta_1',1,1;0)]\\
&=& n\sum\limits_{p=1}^ng^p[V_0(\zeta_1\zeta_1',1,1;0)]\\
 & =& n\mathfrak{s}g_{\zeta_1\zeta_1',1,1}.
\end{eqnarray*} 
\end{proof}

Especially, we obtain the structure of the Grothendieck ring  of the Gelaki's Hopf alegbra where $\beta_1\neq 0$.

\begin{corollary}\label{cor511} Suppose that $\beta_1\neq 0$. Then the Grothendieck ring $G_0(\mathcal{U}_{(n,N,n_1,q,\beta_1,0,0)}=S_1[x^*]=\mathbb{Z}[g,h,x^*]$ with relations
\begin{equation}
{x^*}^{\frac{N}{(N/n,n_1)}}=
n^{\frac{N}{(N/n,n_1)}-1}\mathfrak{s}h.
\end{equation}
\end{corollary}

\begin{proof}Suppose that $V_I(\gamma,1,1;0)$ is an irreducible representation of $\mathcal{U}_{(n,N,n_1,q,\beta_1,0,0)}$. Then $\gamma^{\frac{N}n}=1$ and $\gamma^{n_1}\neq 1$.
Thus $\gamma\in\langle\mathfrak{q}^{n}\rangle\setminus\langle\mathfrak{q}^{\frac{N}{(N/n,n_1)}}\rangle$ by the proof of Corollary \ref{cor59}.   Let $x^*=[V_I(\mathfrak{q}^n,1,1;0)]$.
Then $G_0(\mathcal{U}_{(n,N,n_1,q,\beta_12,0,0)})=\mathbb{Z}[g,h,x^*]$. 
By $\mathfrak{q}^{\frac{Nn_1}{(N/n,n_1)}}=1$, by Theorem \ref{thmVI} we get 
$${x^*}^{\frac{N}{(N/n,n_1)}}={x^*}^{\frac{N}{(N/n,n_1)}-1}x^*=
n\mathfrak{s}^{\frac{N}{(N/n,n_1)}-1}h.$$
\end{proof}

To determine the Grothendieck ring $G_0(H_\beta)$ for $\beta_2\neq 0$, we need the following lemma.
\begin{lemma}\label{L53}Suppose that $\beta_1(\gamma_1^{n_1}-\gamma_2^n)=0$, $\beta_2(\gamma_1^{n_1}-\gamma_3^n)\neq 0$ and $\beta_1(\gamma_1'^{n_1}-\gamma_2'^n)=\beta_2(\gamma_1'^{n_1}-\gamma_3'^n)=\beta_3({\gamma'_1}^{-\frac{2n_1}{n}}\gamma_2'\gamma_3'-q^{2n_1j})=0$. Then
\begin{eqnarray}
[V_{II}(\gamma_1,\gamma_2,\gamma_3;i)][V_0(\gamma'_1,\gamma'_2,\gamma'_3;j)]&=&[V_{II}(\gamma_1\gamma'_1,\gamma_2\gamma'_2,\gamma_3\gamma'_3;i+j)]\nonumber\\ &=&[V_0(\gamma'_1,\gamma'_2,\gamma'_3;j)][V_{II}(\gamma_1,\gamma_2,\gamma_3;i)].\nonumber
\end{eqnarray}
\end{lemma}

\begin{proof}
Since $\beta_2(\gamma_1^{n_1}-\gamma_3^n)\neq0$, $\beta_2\neq0$. Thus $\gamma_1'^{n_1}=\gamma_3'^n$ and $\gamma_3'=\gamma_1'^{\frac{n_1}{n}}q^s$ for some integer $1\leq s\leq n-1$. Hence $\beta_2((\gamma_1\gamma_1')^{n_1}-(\gamma_3\gamma_3')^n)=\beta_2(\gamma_1^{n_1}-\gamma_3^n)(\gamma'_3)^n\neq0$.
Since $\beta_1(\gamma_1^{n_1}-\gamma_2^n)=0$ and $\beta_1(\gamma_1'^{n_1}-\gamma_2'^n)=0$, $\beta_1((\gamma_1\gamma'_1)^{n_1}-(\gamma_2\gamma'_2)^n)=0$.
Let $\{m_0,m_1,\cdots,m_{n-1}\}$ be the basis of $V_{II}(\gamma_1,\gamma_2,\gamma_3;i)$ with the actions given by Equations (\ref{eq11})-(\ref{eq13}). Then $\{m_0\otimes 1,m_1\otimes 1,\cdots, m_{n-1}\otimes 1\}$ is a basis of $V_{II}(\gamma_1,\gamma_2,\gamma_3;i)\otimes V_0(\gamma'_1,\gamma'_2,\gamma'_3;j)$. With this basis, we have
$$a\cdot(m_l\otimes 1)=(a\cdot m_l)\otimes (a\cdot 1)=\sqrt[n]{\gamma_1\gamma'_1}q^{i+j+l}m_l\otimes 1,$$
$$b\cdot(m_l\otimes 1)=(b\cdot m_l)\otimes (b\cdot 1)=(\gamma_2\gamma'_2)m_l\otimes 1,$$
$$c\cdot(m_l\otimes 1)=(c\cdot m_l)\otimes (c\cdot 1)=(\gamma_3\gamma'_3)m_l\otimes 1,$$
$$x\cdot(m_l\otimes 1)=(x\cdot m_l)\otimes (a^{n_1}\cdot 1)+ (b\cdot m_l)\otimes (x\cdot 1) =k_l\gamma'^{\frac{n_1}{n}}_1q^{jn_1}m_{l-1}\otimes 1,\; for\; 1\leq l\leq n-1,$$
$$x\cdot(m_0\otimes 1)=(x\cdot m_0)\otimes (a^{n_1}\cdot 1)+ (b\cdot m_0)\otimes (x\cdot 1) =k_0\gamma'^{\frac{n_1}{n}}_1q^{jn_1}m_{n-1}\otimes 1,$$
$$y\cdot(m_l\otimes 1)=(y\cdot m_l)\otimes (a^{n_1}\cdot 1)+ (c\cdot m_l)\otimes (y\cdot 1) =\gamma'^{\frac{n_1}{n}}_1q^{jn_1}m_{l+1}\otimes 1,\; for\; 0\leq 0\leq n-2,$$
$$y\cdot(m_{n-1}\otimes 1)=(y\cdot m_{n-1})\otimes (a^{n_1}\cdot 1)+ (c\cdot m_{n-1})\otimes (y\cdot 1) =\beta_2(\gamma_1^{n_1}-\gamma_3^n)\gamma'^{\frac{n_1}{n}}_1q^{jn_1}m_0\otimes 1.$$
Let $u_l=\gamma'^{\frac{n_1l}{n}}_1q^{jn_1l}m_l\otimes 1$. Then $\{u_0,\cdots, u_{n-1}\}$ is also a basis of $V_{II}(\gamma_1,\gamma_2,\gamma_3;i)\otimes V_0(\gamma'_1,\gamma'_2,\gamma'_3;j)$. 
Then we have
\begin{eqnarray*}au_p=\sqrt[n]{\gamma_1\gamma_1'}q^{i+j+p}u_p;\ \
bu_p=\gamma_2\gamma_2'u_p;\  \  cu_p=\gamma_3\gamma_3'u_p\qquad for\quad 0\leq p\leq n-1;
\end{eqnarray*}
\begin{eqnarray*}xu_p=\gamma_1'^{\frac{2n_1}{n}}q^{2jn_1}k_pu_{p-1},\qquad for\quad 1\leq p\leq n-1;\qquad
xu_0=\gamma_1'^{\frac{(2-n)n_1}{n}}q^{2jn_1}k_0u_{n-1};
\end{eqnarray*}
\begin{eqnarray*}yu_p=u_{p+1},\qquad for\quad 0\leq p\leq
n-2;\qquad yu_{n-1}=\beta_2((\gamma_1\gamma_1')^{n_1}-(\gamma_2\gamma_2')^n)u_0.
\end{eqnarray*}
Thus $$V_{II}(\gamma_1,\gamma_2,\gamma_3;i)\otimes V_0(\gamma'_1,\gamma'_2,\gamma'_3;j)\cong V_{II}(\gamma_1\gamma'_1,\gamma_2\gamma'_2,\gamma_3\gamma'_3;i+j).$$
Similarly, we get $$V_0(\gamma'_1,\gamma'_2,\gamma'_3;j)\otimes V_{II}(\gamma_1,\gamma_2,\gamma_3;i)\cong V_{II}(\gamma_1\gamma'_1,\gamma_2\gamma'_2,\gamma_3\gamma'_3;i+j).$$
\end{proof}

\begin{remark}
	By Lemma \ref{L53}, we get
\begin{equation}\label{EQ3}
V_{II}(\gamma_1,\gamma_2,\gamma_3,i)\cong V_{II}(\gamma_1',\gamma_2',1;0)\otimes V_0(\gamma_3^{\frac {n}{n_1}},\gamma_3q^{2n_1i},\gamma_3;i)
\end{equation}
where $\gamma_1'=\gamma_1\gamma_3^{-\frac {n}{n_1}}$ satisfying $\gamma_1'^{n_1}\neq 1$, $\gamma_2'=\gamma_2\gamma_3^{-1}q^{-2n_1i}$. If $\beta_1=\beta_3=0$, then $V_{II}(\gamma_1,\gamma_2,1;0)\cong V_{II}(\gamma_1,1,1;0)\otimes V_0(1,\gamma_2,1;0)$.  If $\beta_1\neq 0$, then $V_{II}(\gamma_1,\gamma_2,1;0)=V_{II}(\gamma_1,\gamma_1^{\frac{n_1}{n}} q^u,1;0)=V_{II}(\gamma_1,\gamma_1^{\frac{n_1}{n}},1;0)\otimes V_0(1,q^u,1;0)$ for some integer $u$. If $\beta_1=0$, $\beta_3\neq 0$ and $\gamma_2=q^u$ for some integer $u$, then $V_{II}(\gamma_1,\gamma_2,1;0)\cong V_{II}(\gamma_1,1,1;0)\otimes V_0(1,\gamma_2,1;0)$.
\end{remark}

Let $y_{\epsilon_1,\epsilon_2}=[V_{II}(\epsilon_1,\epsilon_2,1;0)]$ for $\epsilon_1\in\widehat{{\bf K}_0}:=({\bf K}^*/\{\bar{\omega}^v\mid v\in \mathbb{Z}\})\setminus\{\overline{1}\}$, $\epsilon_2\in{\bf K}^*$. If $\beta_2\neq0$ and $\beta_1=\beta_3=0$, then the Grothendieck ring of $G_0(H_\beta)$ is isomorphic to a quotient ring of $R_2''=R_1[y_{\epsilon,1}|\epsilon\in\widehat{{\bf K}_0}]$. Similar to the proof of Theorem \ref{thmVI}, we have the following theorem.

\begin{theorem}\label{thmV} Suppose that $\beta_2\neq 0$ and $\beta_1=\beta_3=0$. Then the Grothendieck ring $G_0(H_\beta)=R_2^{\prime\prime}:=R_1[y_{\epsilon,1}|\epsilon\in\widehat{{\bf K}_0}]$ with the relations
\begin{equation}
y_{\epsilon_1,1}y_{\epsilon_2,1}=
\begin{cases}
\mathfrak{s}y_{\epsilon_1\epsilon_2},\ &(\epsilon_1\epsilon_2)^{n_1}\neq 1\\
n\mathfrak{s}g_{\epsilon_1\epsilon_2,1,1}, \ &(\epsilon_1\epsilon_2)^{n_1}=1.
\end{cases}
\end{equation}
\end{theorem}

Especially, we obtain the structure of the Grothendieck ring  of the Gelaki's Hopf alegbra where $\beta_2\neq 0$.

\begin{corollary}\label{cor514} Suppose that $\beta_2\neq 0$. Then the Grothendieck ring $G_0(\mathcal{U}_{(n,N,n_1,q,0,\beta_2,0)})=S_2^{\prime\prime}:=S_1[y^*]$ with the relation
\begin{equation}{y^*}^{\frac{N}{(N/n,n_1)}}=
n^{\frac{N}{(N/n,n_1)}-1}\mathfrak{s}h.
\end{equation}
\end{corollary}
\begin{proof}Similar to the proof of Corollary \ref{cor511}.\end{proof}

\begin{theorem}\label{T9}Suppose that $\beta_1\beta_2\neq 0$ and $\beta_3=0$. Then the Grothendieck ring $R_3'$ of $H_\beta$ is equal to $R_1[x_{\zeta_1,\zeta_2}, y_{\epsilon_1,\epsilon_2} |(\zeta_1,\zeta_2)\in{\widehat {\bf K_0}\times \bf{K}^*},  \epsilon_1\in {\widehat {\bf K_0}}, \epsilon_2=\epsilon_1^{\frac{n_1}{n}}]$ with relations
\begin{equation}\label{T51}
 x_{\zeta_1,\zeta_2}x_{\zeta_1',\zeta_2'}=
 \begin{cases}
\mathfrak{s}x_{\zeta_1\zeta_1',\zeta}, \ &\text{if}\ \ (\zeta_1\zeta_1')^{n_1}-1\neq0\\
\mathfrak{s}g_{\zeta^{\frac{n}{n_1}},\zeta,\zeta}y_{\zeta_1\zeta_1'\zeta^{-\frac{n}{n_1}},\zeta^{-1}},  \ &\text{if}\ \  (\zeta_1\zeta_1')^{n_1}=1, (\zeta_1\zeta_1')^{n_1}\neq \zeta^n\\
n\mathfrak{s}g_{\zeta_1\zeta_1',1,\zeta}, \ &\text{if}\ \ (\zeta_1\zeta_1')^{n_1}=\zeta^n=1
 \end{cases}
\end{equation}

\begin{equation}\label{T52}
y_{\epsilon_1,\epsilon_2}y_{\epsilon_1',\epsilon_2'}=
\begin{cases}
\mathfrak{s}y_{\epsilon_1\epsilon_1',\epsilon_2\epsilon_2'},\ &\text{if}\ \ (\epsilon_1\epsilon_1')^{n_1}\neq 1\\
n\mathfrak{s}g_{\epsilon_1\epsilon_1',\epsilon_2\epsilon_2',1}, \ &\text{if}\ \ (\epsilon_1\epsilon_1')^{n_1}=1
\end{cases}
\end{equation}
$$x_{\zeta_1,\zeta_2}y_{\epsilon_1,\epsilon_2}=\mathfrak{s}g_{\epsilon_1,\epsilon_2,\epsilon_2}x_{\zeta_1,\zeta_2\epsilon_2^{-1}}=y_{\epsilon_1,\epsilon_2}x_{\zeta_1,\zeta_2},$$
where $\zeta=\zeta_2\zeta_2'$.
\end{theorem}
\begin{proof} By the proof of Theorem \ref{thmVI} and Theorem \ref{thmV}, we have the Equation (\ref{T51}) except for the case of $(\zeta_1\zeta_1')^{n_1}-1=0, (\zeta_1\zeta_1')^{n_1}-(\zeta_2\zeta_2')^n\neq0$ and Equation (\ref{T52}).

(1) Here, we discuss the case in which $((\zeta_1\zeta_1')^{n_1}-1)\beta_1=0$ and $((\zeta_1\zeta_1')^{n_1}-(\zeta_2\zeta_2')^n)\beta_2\neq0$. Then $(\zeta_1\zeta_1')^{n_1}=1$ and $(\zeta_2\zeta_2')^n\neq (\zeta_1\zeta_1')^{n_1}$.
Let $\{m_0,\cdots,m_{n-1}\}$ and $\{m_0',\cdots,m_{n-1}'\}$ be the basis of $V_I(\zeta_1,1,\zeta_2;0)$ and $V_I(\zeta_1',1,\zeta_2';0)$ with the actions given by Equations (\ref{eq8})-(\ref{eq10}) respectively. Then $\{m_l\otimes m'_k,\; 0\leq l,k\leq n-1\}$ is a basis of $V_I(\zeta_1,1,\zeta_2;0)\otimes V_I(\zeta_1',1,\zeta_2';0)$. Let $v_p=m_0\otimes m'_p+\sum\limits^{n-1}_{l=p+1}\alpha_lm_{n+p-l}\otimes m'_l$ for $0\leq p\leq n-1$. Assume that $xv_0=\theta y^{n-1}v_0$. Similar to the proof of Theorem \ref{thmVI}, we can determine $\alpha_l, 1\leq l\leq n-1$.
Hence, $H_\beta v_0\cong V_{II}(\zeta_1\zeta_1',1,\zeta_2\zeta_2';0)\cong V_0((\zeta_2\zeta_2')^\frac{n}{n_1},\zeta_2\zeta_2',\zeta_2\zeta_2';0)\otimes V_{II}(\zeta_1\zeta_1'(\zeta_2\zeta_2')^{-\frac{n}{n_1}},(\zeta_2\zeta_2')^{-1},1;0)$ and $H_\beta v_0$ is an irreducible submodule of $V_I(\zeta_1,1,\zeta_2;0)\otimes V_I(\zeta_1',1,\zeta_2';0)$.

Similarly, we can show that
\begin{eqnarray*}
(H_\beta v_p+\sum^{p-1}_{l=0}H_\beta v_l)/\sum\limits^{p-1}_{l=0}H_\beta v_l &\cong &V_{II}(\zeta_1\zeta_1',1,\zeta_2\zeta_2';n-p) \\
&\cong &V_{II}(\zeta_1\zeta_1'(\zeta_2\zeta_2')^{-\frac{n}{n_1}},(\zeta_2\zeta_2')^{-1},1;0)\otimes V_0((\zeta_2\zeta_2')^\frac{n}{n_1},\zeta_2\zeta_2',\zeta_2\zeta_2';n-p)
\end{eqnarray*}
 and $(H_\beta v_p+\sum^{p-1}_{l=0}H_\beta v_l)/\sum\limits^{p-1}_{l=0}H_\beta v_l$ is an irreducible submodule of $V_I(\zeta_1,1,\zeta_2;0)\otimes V_I(\zeta_1',1,\zeta_2';0)/\sum\limits^{p-1}_{l=0}H_\beta v_l$ for $1\leq p\leq n-1$.
Hence $[V_I(\zeta_1,1,\zeta_2;0)][V_I(\zeta_1',1,\zeta_2';0)]=\sum\limits^{n-1}_{p=0}[V_{II}(\zeta_1\zeta_1',1,\zeta_2\zeta_2';n-p)]$. Consequently, $$x_{\zeta_1,\zeta_2}x_{\zeta_1',\zeta_2'}=c_2y_{\zeta_1\zeta_1'(\zeta_2\zeta_2')^{-\frac{n}{n_1}},(\zeta_2\zeta_2')^{-1}}$$ where $c_2=\sum\limits^{n-1}_{p=0}[V_0((\zeta_2\zeta_2')^\frac{n}{n_1},\zeta_2\zeta_2',\zeta_2\zeta_2';n-p)]=\sum\limits^{n-1}_{p=0}g^p[V_0((\zeta_2\zeta_2')^\frac{n}{n_1},\zeta_2\zeta_2',\zeta_2\zeta_2';0)]
$.

(2) Let $\{m_0,\cdots,m_{n-1}\}$ and $\{m_0'',\cdots,m_{n-1}''\}$ be the basis of $V_I(\zeta_1,1,\zeta_2;0)$ and $V_{II}(\epsilon_1,\epsilon_2,1;0)$ with the actions given by Equations(\ref{eq8})-(\ref{eq10}) and (\ref{eq11})-(\ref{eq13}) respectively. Then $\{m_l\otimes m_k'',\; 0\leq l,k\leq n-1\}$ is a basis of $V_I(\zeta_1,1,\zeta_2;0)\otimes V_{II}(\epsilon_1,\epsilon_2,1;0)$. Let $v_p'=\sum\limits^{n-p-1}_{v=0}\alpha_vm_{p+v}\otimes m_v''$ for $\alpha_0=1$. 
Similar to the proof of Theorem \ref{V}, we can prove $$\sum^{p}_{l=0}H_\beta v_l''/\sum^{p-1}_{l=0}H_\beta v_l''\cong V_I(\zeta_1\epsilon_1,\epsilon_2,\zeta_2;n-p)\cong V_I(\zeta_1,1,\zeta_2\epsilon_2^{-1};0)\otimes V_0(\epsilon_1,\epsilon_2,\epsilon_2;n-p)$$
 and $\sum^{p}_{l=0}H_\beta v_l''/\sum^{p-1}_{l=0}H_\beta v_l''$ is an irreducible submodule of $V_I(\zeta_1,1,\zeta_2;0)\otimes V_{II}(\epsilon_1,\epsilon_2,1;0)/\sum\limits^{p-1}_{l=0}H_\beta v_l''$. Hence 
 $$V_I(\zeta_1,1,\zeta_2;0)\otimes V_{II}(\epsilon_1,\epsilon_2,1;0)=\sum\limits^{n-1}_{p=0}V_I(\zeta_1,1,\zeta_2\epsilon_2^{-1};0)\otimes V_0(\epsilon_1,\epsilon_2,\epsilon_2;n-p).$$
  Similarly, we have $V_{II}(\epsilon_1,\epsilon_2,1;0)\otimes V_I(\zeta_1,1,\zeta_2;0)=\sum\limits^{n-1}_{p=0}V_I(\zeta_1,1,\zeta_2\epsilon_2^{-1};0)\otimes V_0(\epsilon_1,\epsilon_2,\epsilon_2;n-p)$. Consequently, we obtain $x_{\zeta_1,\zeta_2}y_{\epsilon_1,\epsilon_2}=c_3x_{\zeta_1,\zeta_2\epsilon_2^{-1}}=y_{\epsilon_1,\epsilon_2}x_{\zeta_1,\zeta_2}$, where $c_3=\sum\limits^{n-1}_{p=0}[V_0(\epsilon_1,\epsilon_2,\epsilon_2;n-p)]$.
\end{proof}

Especially, we obtain the structure of the Grothendieck ring  of the Gelaki's Hopf alegbra where $\beta_1\beta_2\neq 0$.

\begin{corollary}\label{cor516} Suppose that $\beta_1\beta_2\neq 0$. Then the Grothendieck ring $G_0(\mathcal{U}_{(n,N,n_1,q,\beta_1,\beta_2,0)})=S_3^{\prime}:=S_1[x^*]=\mathbb{Z}[g,h,x^*]$ with  relation
\begin{equation}{x^*}^{\frac{N}{(N/n,n_1)}}=
n^{\frac{N}{(N/n,n_1)}-1}\mathfrak{s}h.
\end{equation}
\end{corollary}
\begin{proof}Since $\beta_1''=0$ implies $\gamma^{n_1}=1$,  $\mathcal{U}_{(n,N,n_1,q,\beta_1,\beta_2,0)}$ has no irreducible representation $V_{II}(\gamma,1,1;0)$. Thus
$S_3'=S_1[x^*]$. The rest of the proof is
similar to that of Corollary \ref{cor511}.\end{proof}

Let $\{m_0',\cdots,m_{r-1}'\}$ be the basis of $V_r(\gamma'_1,\gamma'_2,\gamma'_3;j)$ with the actions given by Equations (\ref{eq15})-(\ref{eq17}). Since $k'_l\neq0$ for $1\leq l\leq r-1$, $\{m''_l:=\prod^{r-1}_{v=r-l}k'_vm'_{r-l-1}|0\leq l\leq r-1\}$ is also a basis of $V_r(\gamma'_1,\gamma'_2,\gamma'_3;j)$. The action of $H_\beta$  on this basis is given by

$am''_l=\prod\limits^{r-1}_{v=r-l}k'_vam'_{r-l-1}=(\gamma'_1)^\frac{1}{n}q^{j-r+1+l}m''_l,$ \quad for $0\leq l\leq r-1$;

$xm''_l=k'_{r-l}m''_{l-1}$,\quad\quad for $1\leq l\leq r-1$; \quad\quad\quad\quad$xm''_0=0$;

$ym''_l=m''_{l+1}$,\quad\quad\quad\quad for $0\leq l\leq r-2$; \quad\quad\quad\quad$ym''_{r-1}=0$.

Let $\overline{\bf{K_0}}:=({\bf{K}^*}/\langle q \rangle)\setminus \{\bar{1}\}$. By the same method as Theorem \ref{V}, we can get the following theorem.

\begin{theorem}\label{T10} 
Suppose that $\beta_2\beta_3\neq 0$ and $\beta_1=0$.  Let $\widetilde{z}_{\xi}=[V_t(1,\xi,1;0)]$, where $\xi\in\overline{{\bf K}^*}$.
Then the Grothendieck ring $R_3''$ of $H_\beta$ is isomorphic to the commutative ring  $R_1[z_2,\tilde{z}_{\xi}, y_{\epsilon_1,\epsilon_2}|(\epsilon_1,\epsilon_2)\in\widehat{{\bf K}_0}\times \overline{{\bf K}_0 }, \xi\in\overline{{\bf K}^*}]$ with  relations $$y_{\epsilon_1,\epsilon_2}z_2=y_{\epsilon_1q^\frac{n}{2},\epsilon_2}+g^{n-1}y_{\epsilon_1q^\frac{n}{2},\epsilon_2},\qquad y_{\epsilon_1,\epsilon_2}\widetilde{z}_{\xi}=\mathfrak{s}''y_{\epsilon_1,\epsilon_2\xi},$$
$$z_2\tilde{z}_\xi=\eta'(\tilde{z}_{\xi q^{-n_1}}+g^{n-1}\tilde{z}_{\xi q^{-n_1}}),\quad \tilde{z}_\xi\tilde{z}_{\xi'}=\mathfrak{s}''\tilde{z}_{\xi\xi'}, \ for \ \xi\xi'\in \overline{{\bf K}^*},$$
$$\tilde{z}_\xi\tilde{z}_{\xi'}=\sum\limits^{t-1}_{p=0,r_p<t}g^{n-p}(g^{-r_p}g_{q^{-\frac{(t-r_p-1)n}{2}},\xi\xi',1}z_{t-r_p}+g_{q^{-\frac{(r_p-1)n}{2}},\xi\xi',1}z_{r_p})+\sum\limits_{p=0,r_p=t}^{t-1}g^{n-p}g_{q^{-\frac{(t-1)n}{2}},\xi\xi',1}z_t$$ for $\xi\xi'=q^{rn_1},$
\begin{equation}\label{eq29}
y_{\epsilon_1,\epsilon_2}y_{\epsilon_1',\epsilon_2'}=
\begin{cases}
\mathfrak{s}y_{\epsilon_1\epsilon_1',\epsilon_2\epsilon_2},\ &\text{if}\ \ (\epsilon_1\epsilon_1')^{n_1}\neq 1\\
u\mathfrak{s}g_{\epsilon_1\epsilon_1',1,1}\widetilde{z}_{\epsilon_2\epsilon_2'}, \ &\text{if}\ \ (\epsilon_1\epsilon_1')^{n_1}=1, \epsilon_1\epsilon_2\in \overline{{\bf K}^*}\\
\mathfrak{s}g_{\epsilon_1\epsilon_1',\epsilon_2\epsilon_2',1}\sum\limits_{p=0}^{n-1}g^{1-p}(z_{t-r_p}+g^{-r_p}z_{r_p}), \ &\text{if}\ \ (\epsilon_1\epsilon_1')^{n_1}=1, \epsilon_1\epsilon_2=q^{rn_1}
\end{cases}
\end{equation}
where $\eta'=[V_0(q^\frac{n}{2},q^{n_1},1;0)]$ and $r_p=(r-2p)\ mod(t)$, $r$ is the minimal positive integer such that $\epsilon_1\epsilon_2=q^{rn_1}$.
\end{theorem}

Especially, we obtain the structure of the Grothendieck ring  of the Gelaki's Hopf alegbra where $\beta_2\beta_3\neq 0$.

\begin{corollary}\label{cor518} 
Suppose that $\beta_2\beta_3\neq 0$.

\begin{itemize}
	\item [(1)]Suppose $(N,nn_1,2n_1t)<(nn_1,N)$.
 If either $2n\mid N$, or $2n\nmid N$ and $2(n,n_1)\nmid n$, then the Grothendieck ring $S_3'=G_0(\mathcal{U}_{(n,N,n_1,q,0,\beta_2,\beta_3)})$ is isomorphic to the commutative ring
$S_2[ y^*]=\mathbb{Z}[g,h_2,z_2,z'',y^*]$ with relations (\ref{Eq*1}),  (\ref{T2}), $z_2z''=(1+g^{n-1})z''$,
$z_2y^*=(1+g^{n-1})y^*$ and ${y^*}^{\frac{N}{(N/n,n_1)}}=u\mathfrak{s}^{\frac{N}{(N/n,n_1)}-1}z''$.
If $2n\nmid N$ and $2(n,n_1)\mid n$, then the Grothendieck ring $S_3':=G_0(\mathcal{U}_{(n,N,n_1,q,0,\beta_2,\beta_3)})$ is isomorphic to the commutative ring
$S_2[ y^*]=\mathbb{Z}[g,h_2,z_3,z'',y^*]$ with relations  (\ref{T3}), (\ref{T2}), $z_3z''=(1+g^{n-1}+g^{n-2})z''$, $z_3y^*=(1+g^{n-1}+g^{n-2})y^*$ and ${y^*}^{\frac{N}{(N/n,n_1)}}=u\mathfrak{s}^{\frac{N}{(N/n,n_1)}-1}z''$.

	\item [(2)]Suppose that $\beta_2\beta_3\neq 0$ and $(N,nn_1,2n_1t)=(nn_1,N)$. Then the Grothendieck ring $$G_0(\mathcal{U}_{(n,N,n_1,q,0,\beta_2,\beta_3)})\cong
\begin{cases}
\mathbb{Z}[g,h_2,z_2,y^*], &\text{if}\ \ 2n\mid N \ \text{or}\ 2n\nmid N \ \text{and}\ 2(n,n_1)\nmid n\\
\mathbb{Z}[g,h_2,z_3,y^*], &\text{if}\ \ 2n\nmid N\ \text{and}\ 2(n,n_1)\mid n
\end{cases}$$ is a subring of $S_3'$.

\end{itemize} 
\end{corollary}

Similar to the proof of Theorem \ref{L55}-Theorem \ref{T10}, we have the following theorem.
\begin{theorem}\label{Th5.12}
 Suppose that $\beta_1\beta_2\beta_3\neq 0$. Then the Grothendieck ring $R_4$ of $H_\beta$ is equal to $R_1[z_2,z_\xi', x_{\zeta_1,\zeta_2}, $ $y_{\epsilon_1,\epsilon_1}|\xi\in\widetilde{\bf K_0}, (\zeta_1,\zeta_2)\in \overline{\bf K_0}\otimes{\bf K}^*,  (\epsilon_1,\epsilon_2)\in {\widehat{ {\bf K}_0}}\times \overline{{\bf K}_0}, \epsilon_1^{n_1}=\epsilon_2^n]$ with relations (\ref{Eq*1}),  (\ref{eq29}),
 $$z_2z_\xi'=z_{\xi q^\frac{n}{2}}'+g^{n-1}z_{\xi q^\frac{n}{2}}',\qquad
z_2x_{\zeta_1,\zeta_2}=x_{\zeta_1q^\frac{n}{2},\zeta_2}+g^{n-1}x_{\zeta_1q^\frac{n}{2},\zeta_2},\qquad x_{\zeta_1,\zeta_2}z_\xi'=g^{n-t}\mathfrak{s}''x_{\zeta_1\xi,\zeta_2}, $$
$$ x_{\zeta_1,\zeta_2}y_{\epsilon_1,\epsilon_2}=\mathfrak{s}g_{\epsilon_1,\epsilon_2,\epsilon_2}x_{\zeta_1,\zeta_2\epsilon_2^{-1}},\qquad
y_{\epsilon_1,\epsilon_2}z_2=y_{\epsilon_1q^\frac{n}{2},\epsilon_2}+g^{n-1}y_{\epsilon_1q^\frac{n}{2},\epsilon_2}, \qquad y_{\epsilon_1,\epsilon_2}z_{\xi}'=\mathfrak{s}''y_{\epsilon_1\xi,\epsilon_2},$$
\begin{equation}
z_\xi'z_{\xi'}'=
\begin{cases}
g^{n-t}\mathfrak{s}''z'_{\xi\xi'} \ &\text{if}\ \ \xi\xi'\in \widetilde{\bf K}_0,\\
\sum\limits^{t-1}_{p=0,r_p<t}(g^{n-p-r_p}z_{t-r_p}+g^{n-p}z_{r_p})+\sum\limits_{i=0,r_p=t}^{t-1}g^{n-p}z_t, \  &\text{if}\ \ \xi\xi'\in\{q^{un_1}\mid u\in{\mathbb{Z}}\}
\end{cases},\nonumber
\end{equation}
and 

\begin{equation}
 x_{\zeta_1,\zeta_2}x_{\zeta_1',\zeta_2'}=
 \begin{cases}
\mathfrak{s}x_{\zeta_1\zeta_1',\zeta_2\zeta_2'},  \ &\text{if}\ \ (\zeta_1\zeta_1')^{n_1}\neq 1\\
\mathfrak{s}g_{\zeta,\zeta^*,\zeta^*}y_{\zeta^{\prime\prime}\zeta^{-1},(\zeta^*)^{-1}},  \ &\text{if}\ \ (\zeta_1\zeta_1')^{n_1}=1, (\zeta_1\zeta_1')^{n_1}\neq (\zeta_2\zeta_2')^n\\
\mathfrak{s}'g_{\zeta^{\prime\prime},1,\zeta^*}(\sum\limits_{p=0}^{n-1}g^{r-p}z_{t-r_p}+g^pz_{r_p}), \ &\text{if}\ \ (\zeta'')^{n_1}=(\zeta^*)^n=1,\ \zeta^*=(\zeta'')^{\frac{2n_1}{n}}q^{(-r+1)n_1} \\
u\mathfrak{s}g_{\zeta^{\frac12},1,\zeta^*}z_{\zeta''\zeta^{-\frac12}}', \ &\text{if}\ \ (\zeta_1\zeta_1')^{n_1}=(\zeta_2\zeta_2')^n=1,(\zeta'')^{-\frac{2n_1}{n}}\zeta^*\notin\langle q^{n_1}\rangle,
 \end{cases}\nonumber
\end{equation}
where $\zeta=(\zeta_2\zeta_2')^{\frac{n}{n_1}}$, $\zeta^*=\zeta_2\zeta_2'$, $\zeta^{\prime\prime}=\zeta_1\zeta_1'$,
$r_p=(r-2p)\ mod(t)$ and $1\leq r_p\leq t$, $r$ is the minimal positive integer such that  $\zeta_2\zeta_2'=(\zeta_1\zeta_1')^{\frac{2n_1}{n}}q^{(-r+1)n_1} $.
\end{theorem}

Especially, we obtain the structure of the Grothendieck ring  of the Gelaki's Hopf alegbra where $\beta_1\beta_2\beta_3\neq 0$.

\begin{corollary}\label{cor520}
 Suppose that $\beta_1\beta_2\beta_3\neq 0$.
 \begin{itemize}
 	\item[(1)] Suppose $(N,nn_1,2n_1t)<(nn_1,N)$. If either $2n\mid N$, or $2n\nmid N$ and $2(n,n_1)\nmid n$, then the Grothendieck ring $S_3''=G_0(\mathcal{U}_{(n,N,n_1,q,\beta_1,\beta_2,\beta_3)})$ is isomorphic to the commutative ring
$S_2[ x^*]=\mathbb{Z}[g,h_2,z_2,z'',x^*]$ with relations in Corollary \ref{cor59}(1)(I).
If $2n\nmid N$ and $2(n,n_1)\mid n$, then the Grothendieck ring $S_3'':=G_0(\mathcal{U}_{(n,N,n_1,q,\beta_1,\beta_2,\beta_3)})$ is isomorphic to the commutative ring
$S_2[ x^*]=\mathbb{Z}[g,h_2,z_3,z'',x^*]$ with relations in Corollary \ref{cor59}(1)(II).

\item[(2)] Suppose that $\beta_1\beta_2\beta_3\neq 0$ and $(N,nn_1,2n_1t)=(nn_1,N)$. Then the Grothendieck ring $$G_0(\mathcal{U}_{(n,N,n_1,q,\beta_1,\beta_2,\beta_3)})\cong
\begin{cases}
\mathbb{Z}[g,h_2,z_2,x^*], &\text{if}\ \ 2n\mid N \ \text{or}\ 2n\nmid N \ \text{and}\ 2(n,n_1)\nmid n\\
\mathbb{Z}[g,h_2,z_3,x^*], &\text{if}\ \ 2n\nmid N\ \text{and}\ 2(n,n_1)\mid n
\end{cases}$$ is a subring of $S_3''$.

 \end{itemize}
 
\end{corollary}

Let $\omega$ be a primitive $N$-th root of unity. If 
 $N\nmid n_1^2$, then $U_{(N,n_1,\omega)}={\mathcal  U}_{(N/(N,n_1),N,n_1,\omega^{n_1},0,0,1)}$  and  ${\mathcal
U}_{(N/(N,n_1),N,n_1,\omega^{n_1},0,0,\gamma)}\simeq
U_{(N,n_1,\omega)}$ as Hopf algebras for any $\gamma\in {\bf
K}^*$. Hence $V_0(\gamma_1,1,1,i)$ and $V_r(\gamma_1,1,1,i)$ are irreducible representations of $U_{(N,n_1,\omega)}$, where $\gamma_1^{(N,n_1)}=1$. Thus, the Grotendieck ring  of $U_{(N,n_1,\omega)}$ is the same as the Grothendieck ring $G_0({\mathcal U}_{(n,N,n_1,q,0,0,\beta_3)})$ with $\beta_3\neq0$ in Corollary \ref{cor56}.

\begin{corollary}\label{corford}
 Let $h'=[V_0(\omega^{n'},1,1;0)]$ and $z'=[V_t(\omega^{n},1,1;0)]$. Then $g^n=h'^{\frac{N}{(N,n')}}=1$, where $n=\frac{N}{(N,\nu)}$ and $n'=\frac{N^2}{(N^2,N\nu,2\nu^2)}$.
\begin{itemize}
	\item[(1)] Suppose that $(\nu^2,N)\nmid 2\nu$.
If either $2\mid (N,\nu)$, or $2\nmid (N,\nu)$ and $2(N,\nu^2)\nmid N$, then the Grothendieck ring $S=G_0(U_{(N,\nu,\omega)})\cong\mathbb{Z}[g,h',z_2,z']$ with relations (\ref{Eq*1}), (\ref{EQ23}) and $z_2z'=(1+g^{n-1})z'$.
If $2\nmid (N,\nu)$ and $2(N,\nu^2)\mid N$, then the Grothendieck ring $S=G_0(U_{(N,\nu,\omega)})\cong\mathbb{Z}[g,h',z_3,z']$ with relations  (\ref{EQ23}),  (\ref{T3}) and $z_3z'=(1+g^{n-1}+g^{n-2})z'$.
 \item[(2)]Suppose that $2\nu=(\nu^2,N)$. Then the Grothendieck ring $$G_0(U_{(N,\nu,\omega)})\cong
\begin{cases}
\mathbb{Z}[g,h',z_2], &\text{if}\ \  2\mid (N,\nu), \ \text{or}\ 2\nmid (N,\nu) \ \text{and}\ 2(N,\nu^2)\nmid N\\
\mathbb{Z}[g,h',z_3], &\text{if}\ \  2\nmid (N,\nu)\ \text{and}\ 2(N,\nu^2)\mid N
\end{cases}$$ is a subring of $S$.
\end{itemize}
\end{corollary}

\begin{remark}
  By Corollary \ref{cor56}-Corollary \ref{cor520}, we have
$$G_0({\mathcal U}_{(n,N,n_1,q,\beta_1,\beta_2,0)})=G_0({\mathcal U}_{(n,N,n_1,q,\beta_1,0,0)})\cong G_0({\mathcal U}_{(n,N,n_1,q,0,\beta_2,0)})$$ and $$G_0({\mathcal U}_{(n,N,n_1,q,\beta_1,\beta_2,\beta_3)})=G_0({\mathcal U}_{(n,N,n_1,q,\beta_1,0,\beta_3)})\cong G_0({\mathcal U}_{(n,N,n_1,q,0,\beta_2,\beta_3)}).$$
But ${\mathcal U}_{(n,N,n_1,q,\beta_1,\beta_2,\beta_3)}\ncong{\mathcal U}_{(n,N,n_1,q,\beta_1,0,\beta_3)}$ by \cite[Proposition 3.2]{G}.
\end{remark}

%\section*{ACKNOWLEDGMENT}

\end{document}